\documentclass[peerreview,a4paper,12pt]{IEEEtran}

\usepackage{amsthm}
\usepackage{amsmath}
\usepackage{amsfonts,amssymb,amsmath}
\usepackage{graphicx}
\usepackage{epsfig}
\usepackage{caption}
\usepackage[applemac]{inputenc}
\usepackage{latexsym}
\usepackage{dsfont}
\usepackage{tikz}
\usetikzlibrary{arrows,shapes,backgrounds,patterns,decorations.pathreplacing}

\newtheorem{mydef}{Definition}
\newtheorem{mycor}{Corollary}
\newtheorem{myprop}{Proposition}
\newtheorem{mythe}{Theorem}
\newtheorem{mylem}{Lemma}
\newtheorem{mynot}{Notation}
\newtheorem{myrem}{Remark}

\newtheorem{myex}{Example}

\DeclareMathOperator*{\per}{per}
\DeclareMathOperator*{\scon}{scon}
\DeclareMathOperator*{\con}{con}
\DeclareMathOperator*{\degg}{deg}
\DeclareMathOperator{\proj}{Proj}
\DeclareMathOperator*{\lcm}{lcm}
\DeclareMathOperator*{\cor}{cor}

\begin{document}

\sloppy

\title{Ergodic Theory Meets Polarization. I: \\ An Ergodic Theory for Binary Operations} 

\author{
Rajai Nasser\\
School of Computer and Communication Sciences, EPFL\\
Lausanne, Switzerland\\
Email: rajai.nasser@epfl.ch
\thanks{This paper was presented in part at the IEEE International Symposium on Information Theory, Hong Kong, June 2015.}
}

\maketitle

\setcounter{page}{1}

\begin{abstract}
An open problem in polarization theory is to determine the binary operations that always lead to polarization (in the general multilevel sense) when they are used in Ar{\i}kan style constructions. This paper, which is presented in two parts, solves this problem by providing a necessary and sufficient condition for a binary operation to be polarizing. This (first) part of the paper introduces the mathematical framework that we will use in the second part \cite{RajErgII} to characterize the polarizing operations. We define uniformity preserving, irreducible, ergodic and strongly ergodic operations and we study their properties. The concepts of a stable partition and the residue of a stable partition are introduced. We show that an ergodic operation is strongly ergodic if and only if all its stable partitions are their own residues. We also study the products of binary operations and the structure of their stable partitions. We show that the product of a sequence of binary operations is strongly ergodic if and only if all the operations in the sequence are strongly ergodic. In the second part of the paper, we provide a foundation of polarization theory based on the ergodic theory of binary operations that we develop in this part.
\end{abstract}

\section{Introduction}

Polar codes are a class of codes invented by Ar{\i}kan \cite{Arikan} which achieves the capacity of symmetric binary-input memoryless channels with low encoding and decoding complexities. Ar{\i}kan and Telatar showed that the probability of error of the successive cancellation decoder of polar codes is equal to $o(2^{-N^{1/2-\epsilon}})$ \cite{ArikanTelatar}.

Ar{\i}kan's construction is based on a basic transformation that is applied recursively. The basic transformation starts with two identical and independent copies of a single user channel $W$ and transforms them to two channels $W^-$ and $W^+$ such that $I(W^-)+I(W^+)=2I(W)$, which means that the total capacity is preserved. It is shown that if $W$ is not extremal, i.e., if $0< I(W)<\log 2$, then $W^-$ (resp. $W^+$) is strictly worse (resp. strictly better) than $W$. This fact was used to show that if we apply the basic transformation recursively, we can convert a set of identical and independent copies of a given single user binary-input channel, into a set of ``almost perfect'' and ``almost useless" channels while preserving the total capacity. This phenomenon is called \emph{polarization} and it is used to construct capacity-achieving polar codes.

Ar{\i}kan's basic construction uses the XOR operation. Therefore, any attempt to generalize Ar{\i}kan's technique to channels having a non-binary input alphabet $\mathcal{X}$, has to replace the XOR operation by a binary operation $\ast$ on the input alphabet $\mathcal{X}$. The first operation that was investigated is the addition modulo $q$, where $q=|\mathcal{X}|$ and $\mathcal{X}$ is endowed with the algebraic structure $\mathbb{Z}_{q}$. \c{S}a\c{s}o\u{g}lu et al. \cite{SasogluTelAri} showed that if $q$ is prime, then the addition modulo $q$ leads to the same polarization phenomenon as in the binary input case.

Park and Barg \cite{ParkBarg} showed that if $q=2^r$ with $r>0$, then the addition modulo $q$ leads to a polarization phenomenon which is different from the polarization in the binary input case, but it can still be used to construct capacity-achieving polar codes. They showed that we have a multilevel polarization: while we do not always have polarization to ``almost perfect" or ``almost useless" channels, we always have polarization to channels which are easy to use for communication. Sahebi and Pradhan \cite{SahebiPradhan} showed that multilevel polarization also happens if any Abelian group operation on the alphabet $\mathcal{X}$ is used. This allows the construction of polar codes for arbitrary discrete memoryless channels (DMC) since any alphabet can be endowed with an Abelian group structure.

Polar codes for arbitrary DMCs were also constructed by \c{S}a\c{s}o\u{g}lu \cite{SasS} using a special quasigroup operation which ensures two-level polarization. The author and Telatar have shown in \cite{RajTel} that all quasigroup operations are polarizing (in the general multilevel sense) and can be used to construct capacity-achieving polar codes for arbitrary DMCs \cite{RajTelA}.

Quasigroups are the largest class of binary operations that is known to be polarizing. This paper, which is presented in two parts, determines all the polarizing operations by providing a necessary and sufficient condition for a binary operation to be polarizing. This part introduces the mathematical framework that we will use in the second part \cite{RajErgII} to characterize the polarizing operations.

In section II we introduce the notion of uniformity preserving operations. A \emph{uniformity preserving operation} $\ast$ on $\mathcal{X}$ is a binary operation such that the mapping $f_{\ast}:\mathcal{X}^2\rightarrow \mathcal{X}^2$ defined by $f_{\ast}(x,y)=(x\ast y,y)$ is bijective.  It is called uniformity preserving since for any pair of random variables $(X_1,X_2)$ in $\mathcal{X}^2$, $(X_1\ast X_2,X_2)$ is uniform in $\mathcal{X}^2$ if and only if $(X_1,X_2)$ is uniform in $\mathcal{X}^2$. As we will see in \cite{RajErgII}, if $\ast$ is not uniformity preserving, then the Ar{\i}kan style construction that is based on $\ast$ does not conserve the symmetric capacity. Hence being uniformity preserving is a necessary condition to be polarizing. On the other hand, being a quasigroup operation is a sufficient condition \cite{RajTel}. Therefore, a necessary and sufficient condition must be a property that is stronger than uniformity preserving and weaker than quasigroup. A reasonable strategy to search for a necessary and sufficient condition is to relax the quasigroup property while keeping the uniformity preserving property.

The difference between a quasigroup operation and a uniformity preserving operation is that in the case of a quasigroup operation, any element is reachable from any other element by one multiplication on the right. This property does not always hold for a uniformity preserving operation.

One possible relaxation of the quasigroup property is to consider uniformity preserving operations where all the elements are reachable from each other by multiple multiplications on the right. Irreducible and ergodic operations --- which are defined and studied in section III --- satisfy this property. The concepts of irreducible and ergodic operations are very similar to the concepts of irreducible and ergodic Markov chains. The reason why we consider such binary operations is because of their good connectability properties: if the elements of $\mathcal{X}$ are well connected under $\ast$, this will create strong correlations between the inputs of the synthetic channels which should ultimately lead to a polarization phenomenon. 

Although ergodic operations seem to have good connectability properties, this is not enough to ensure polarization as we will see in Part II \cite{RajErgII}. It turns out that we need a stronger notion of ergodicity. But in order to define this stronger notion of ergodicity, we first need to define stable partitions. Section IV introduces balanced, periodic and stable partitions and investigates their properties. Stable partitions are a generalization of the concept of quotient groups. In section V, we introduce and study the notion of the residue of a stable partition and in section VI we define and investigate strongly ergodic operations. We show that an ergodic operation is strongly ergodic if and only if each stable partition is its own residue. Strong ergodicity is a novel concept and does not have a similar concept in the ergodic theory of Markov chains. We will show in Part II \cite{RajErgII} that a binary operation is polarizing if and only if it is uniformity preserving and its right-inverse is strongly ergodic. 

Generated stable partitions are introduced and studied in section VII. This concept is needed to show that the strong ergodicity of the right-inverse operation is a sufficient condition for polarization.

The products of binary operations are defined in section VIII and the structure of their stable partitions is studied. We show that the product of a sequence of binary operations is strongly ergodic if and only if every operation in the sequence is strongly ergodic. The products of binary operations and their stable partitions are important for the study of polarization theory for multiple access channels \cite{RajErgII}.

The main results which will be proven in Part II \cite{RajErgII} are:
\begin{itemize}
\item If $\ast$ is a binary operation on a set $\mathcal{X}$, then $\ast$ is polarizing (in the general multi-level sense) if and only if $\ast$ is uniformity preserving and $/^\ast$ (the right-inverse of $\ast$) is strongly ergodic.
\item The exponent of any polarizing operation is at most $\frac{1}{2}$, which is achieved by quasigroup operations.
\item If $\ast_1,\ldots,\ast_m$ are $m$ binary operations on $\mathcal{X}_1,\ldots,\mathcal{X}_m$ respectively, then $(\ast_1,\ldots,\ast_m)$ is MAC-polarizing if and only if $\ast_1,\ldots,\ast_m$ are polarizing.
\end{itemize}

\section{Uniformity preserving operations}

All the sets that are considered in this paper are finite.

\begin{mydef}
A \emph{uniformity preserving operation} $\ast$ on $\mathcal{X}$ is a binary operation such that the mapping $f_{\ast}:\mathcal{X}^2\rightarrow \mathcal{X}^2$ defined by $f_{\ast}(x,y)=(x\ast y,y)$ is bijective. It is called uniformity preserving since for any pair of random variables $(X_1,X_2)$ in $\mathcal{X}^2$, $(X_1\ast X_2,X_2)$ is uniform in $\mathcal{X}^2$ if and only if $(X_1,X_2)$ is uniform in $\mathcal{X}^2$.
\end{mydef}

\begin{myrem}
It is easy to see that $\ast$ is uniformity preserving if and only if it satisfies the following condition:
\begin{itemize}
\item The multiplication-on-the-right mappings $\pi_b:\mathcal{X}\rightarrow \mathcal{X}$ defined by $\pi_b(x)=x\ast b$ are bijective for all $b\in \mathcal{X}$. We denote $\pi_b^{-1}(a)$ as $a/^{\ast}b$. The binary operation $/^{\ast}$ is called the \emph{right-inverse} of $\ast$.
\end{itemize}
It is easy to see that if $\ast$ is uniformity preserving then $/^{\ast}$ is uniformity preserving as well.
\end{myrem}

\begin{mydef}
A uniformity preserving operation is said to be a quasigroup operation if it also satisfies the following:
\begin{itemize}
\item The multiplication-on-the-left mappings $\eta_b:\mathcal{X}\rightarrow \mathcal{X}$ defined by $\eta_b(x)=b\ast x$ are bijective for all $b\in \mathcal{X}$. We denote $\eta_b^{-1}(a)$ as $b\backslash_\ast a$. The binary operation $\backslash_{\ast}$ is called the \emph{left-inverse} of $\ast$
\end{itemize}
It is easy to see that if $\ast$ is a quasigroup operation then $/^{\ast}$ and $\backslash_{\ast}$ are quasigroup operations as well.

Note that for a general quasigroup operation $\ast$, we may find $a,b\in\mathcal{X}$ such that $\pi_b^{-1}(a)=a/^\ast b\neq b\backslash_\ast a=\eta_b^{-1}(a)$. This is why we use different notations for left and right inverses.
\end{mydef}

\begin{mynot}
Let $A$ and $B$ be two subsets of  $\mathcal{X}$. We define the set:
$$A\ast B:=\{a\ast b:\; a\in A,b\in B\}.$$
For $a,b\in\mathcal{X}$, we denote $\{a\}\ast B$ and $A\ast\{b\}$ by $a\ast B$ and $A\ast b$ respectively.
\end{mynot}

It is easy to see that if $\ast$ is uniformity preserving and $B$ is non-empty, then $|A\ast B|\geq |A|$. On the other hand, the relation $|A\ast B|\geq |B|$ does not hold in general unless $\ast$ is a quasigroup operation and $A$ is non-empty.

\section{Irreducible and ergodic operations}

In this section and throughout the paper, $\ast$ is always a uniformity preserving operation.

\begin{mydef}
Let $\ast$ be a uniformity preserving operation on a set $\mathcal{X}$. We say that $a\in\mathcal{X}$ is $\ast$-connectable to $b\in\mathcal{X}$ in $l$-steps if there exist $l$ elements $x_0,\ldots,x_{l-1}\in \mathcal{X}$ satisfying $(\ldots((a\ast x_0)\ast x_1)\ldots\ast x_{l-1})=b$. We denote this relation by $a\stackrel{\ast,l}{\longrightarrow} b$.

We say that $a$ is $\ast$-connectable to $b$ if there exists $l>0$ such that $a\stackrel{\ast,l}{\longrightarrow} b$. We denote this relation by $a\stackrel{\ast}{\longrightarrow} b$.
\end{mydef}

\begin{mydef}
A uniformity preserving operation $\ast$ is said to be \emph{irreducible} if all the elements of $\mathcal{X}$ are $\ast$-connectable to each other. If $\ast$ is irreducible, we define the period of an element $a\in\mathcal{X}$ as $\per(\ast,a):=\gcd\{l> 0:\;a\stackrel{\ast,l}{\longrightarrow} a\}$, and we define the period of $\ast$ as:
$$\per(\ast):=\gcd\left\{\per(\ast,a):\; a\in\mathcal{X}\right\}=\gcd\left\{l> 0:\;\exists a\in\mathcal{X},a\stackrel{\ast,l}{\longrightarrow} a\right\}.$$

If there exists $l>0$ such that all the elements of $\mathcal{X}$ are $\ast$-connectable to each other in $l$ steps, we say that the operation $\ast$ is \emph{ergodic}. In this case, we call the minimum integer $l>0$ which satisfies this property the \emph{connectability} of the operation $\ast$, and we denote it by $\con(\ast)$.
\end{mydef}

\begin{myrem}
In order to justify our choice of terminology in the previous definition, consider a sequence $(X_n')_{n\geq 0}$ of independent and uniformly distributed random variables in $\mathcal{X}$. Define $(X_n)_{n\geq 0}$ recursively as follows: $X_0=X_0'$ and $X_n=X_{n-1}\ast X_n'$ for $n>0$. It is easy to see that $(X_n)_{n\geq 0}$ is a stationary Markov chain. We have the following:
\begin{itemize}
\item $\ast$ is irreducible if and only if $(X_n)_{n\geq 0}$ is irreducible.
\item $\ast$ is ergodic if and only if $(X_n)_{n\geq 0}$ is ergodic.
\end{itemize}
\end{myrem}

The following proposition shows the important properties of irreducible and ergodic operations. These properties will be used in Part II \cite{RajErgII} to show that every polarizing operation is ergodic.

\begin{myprop}
\label{lemerg}
We have the following:
\begin{enumerate}
\item Every quasigroup operation is ergodic, and every ergodic operation is irreducible.
\item If $\ast$ is uniformity preserving but not irreducible, there exists two disjoint non-empty subsets $A_1$ and $A_2$ of $\mathcal{X}$ such that $A_1\cup A_2=\mathcal{X}$, $A_1\ast \mathcal{X}=A_1$ and $A_2\ast \mathcal{X}=A_2$.
\item If $\ast$ is irreducible, we have $\per(\ast,a)=\per(\ast)$ for all $a\in\mathcal{X}$.
\item If $\ast$ is irreducible, there exists a partition $\mathcal{E}_{\ast}$ of $\mathcal{X}$ containing $n=\per(\ast)$ subsets $H_0$, \ldots, $H_{n-1}$ such that $H_{i}\ast \mathcal{X}=H_{i+1\bmod n}$ for all $0\leq i < n$. Moreover, we have $|H_0|=\ldots=|H_{n-1}|$.
\item If $\ast$ is irreducible, there exists an integer $d>0$ such that for any $0\leq i< n=\per(\ast)$, any element of $H_i$ is $\ast$-connectable to any element of $H_{i+d\bmod n}$ in $d$ steps. We call the least integer $d>0$ satisfying this property the \emph{connectability} of the irreducible operation $\ast$ and we denote it $\con(\ast)$ (This definition is consistent with the definition of the connectability of ergodic operations. I.e., the connectability of an ergodic operation when it is seen as an irreducible operation is the same as its connectability when it is seen as an ergodic operation).
\item If $\ast$ is irreducible, then for every $s\geq\con(\ast)$ and every $0\leq i< n=\per(\ast)$, any element of $H_i$ is $\ast$-connectable to any element of $H_{i+s\bmod n}$ in $s$ steps.
\item If $\ast$ is irreducible, $\per(\ast)=1$ if and only if $\ast$ is ergodic.
\item If $\ast$ is ergodic, all the elements of $\mathcal{X}$ are $\ast$-connectable to each other in $s$-steps for any $s\geq \con(\ast)$.
\item If $\ast$ is ergodic, $\con(\ast)=1$ if and only if $\ast$ is a quasigroup operation.
\item If $\ast$ is irreducible (resp. ergodic), then $/^{\ast}$ is irreducible (resp. ergodic) as well.
\end{enumerate}
\end{myprop}
\begin{proof}
See Appendix \ref{appA}.
\end{proof}

\section{Balanced, periodic and stable partitions}

\label{sec4}

\begin{mynot}
Let $\mathcal{H}$ be a set of subsets of a set $\mathcal{X}$, we define the following:
\begin{itemize}
\item $\displaystyle \|\mathcal{H}\|_{\wedge}=\min_{H\in\mathcal{H}}|H|$.
\item $\displaystyle \|\mathcal{H}\|_{\vee}=\max_{H\in\mathcal{H}}|H|$.
\end{itemize}
\end{mynot}

\begin{mydef}
A partition $\mathcal{H}$ of a set $\mathcal{X}$ is said to be a \emph{balanced partition} if all the elements of $\mathcal{H}$ have the same size. We denote the common size of its elements by $\|\mathcal{H}\|$. The number of elements in $\mathcal{H}$ is denoted by $|\mathcal{H}|$. Clearly, $|\mathcal{X}|=|\mathcal{H}|\cdot\|\mathcal{H}\|$ and $\|\mathcal{H}\|=\|\mathcal{H}\|_{\wedge}=\|\mathcal{H}\|_{\vee}$ for such a partition.
\end{mydef}

\begin{mydef}
Let $\mathcal{H}$ be a partition of a set $\mathcal{X}$. We define the \emph{projection onto} $\mathcal{H}$ as the mapping $\proj_{\mathcal{H}}:\mathcal{X}\rightarrow\mathcal{H}$, where $\proj_{\mathcal{H}}(x)$ is the unique element $H\in\mathcal{H}$ such that $x\in H$.
\end{mydef}

\begin{mynot}
Let $\mathcal{A}$ and $\mathcal{B}$ be two sets of subsets of $\mathcal{X}$. We define $\mathcal{A}\ast\mathcal{B}$ as follows:
$$\mathcal{A}\ast\mathcal{B}=\{A\ast B:\; A\in\mathcal{A},B\in\mathcal{B}\}.$$
\end{mynot}

\begin{mydef}
Let $\mathcal{H}$ be a set of subsets of $\mathcal{X}$, and let $\ast$ be a uniformity preserving operation on $\mathcal{X}$. We define the set $\mathcal{H}^{\ast}=\mathcal{H}\ast\mathcal{H}=\{A\ast B:\; A,B\in\mathcal{H}\}$, and we define the sequence $(\mathcal{H}^{n\ast})_{n\geq 0}$ recursively as follows:
\begin{itemize}
\item $\mathcal{H}^{0\ast}=\mathcal{H}$.
\item $\mathcal{H}^{n\ast}:=(\mathcal{H}^{(n-1)\ast})^{\ast}=\mathcal{H}^{(n-1)\ast} \ast\mathcal{H}^{(n-1)\ast}$ for all $n>0$.
\end{itemize}
\end{mydef}

\begin{mydef}
A partition $\mathcal{H}$ of $\mathcal{X}$ is said to be a \emph{periodic partition} of $(\mathcal{X},\ast)$ if there exists $n>0$ such that $\mathcal{H}^{n\ast}=\mathcal{H}$. In this case, the minimum integer $n>0$ which satisfies $\mathcal{H}^{n\ast}=\mathcal{H}$ is called the \emph{period} of $\mathcal{H}$, and it is denoted by $\per(\mathcal{H})$.

A partition $\mathcal{H}$ of $\mathcal{X}$ is said to be a \emph{stable partition} of $(\mathcal{X},\ast)$ if $\mathcal{H}$ is both balanced and periodic.
\end{mydef}

Throughout the paper, we write that $\mathcal{H}$ is a periodic (resp. stable) partition of $\mathcal{X}$ if the binary operation $\ast$ is clear from the context.

\begin{myex}
Let $Q=\mathbb{Z}_{n}\times\mathbb{Z}_{n}$, define $(x_1,y_1)\ast (x_2,y_2)=(x_1+y_1+x_2+y_2,y_1+y_2)$ which is a quasigroup operation. For each $j\in\mathbb{Z}_n$ and each $0\leq i< n$, define $H_{i,j}=\{(j+ik,k):\; k\in\mathbb{Z}_n\}$. Let $\mathcal{H}_i=\{H_{i,j}:\; j\in\mathbb{Z}_n\}$ for $0\leq i< n$. It is easy to see that $\mathcal{H}_i^{\ast}=\mathcal{H}_{i+1}$ for $0\leq i< n-1$ and $\mathcal{H}_{n-1}^{\ast}=\mathcal{H}_0$. Therefore, $\mathcal{H}:=\mathcal{H}_0$ is a periodic partition of $(Q,\ast)$ and $\per(\mathcal{H})=n$. Moreover, $\mathcal{H}$ is balanced with $\|\mathcal{H}\|=n$, hence $\mathcal{H}$ is a stable partition.
\end{myex}

\begin{myprop}
\label{propPerPer}
Let $\mathcal{H}$ be a periodic partition of $(\mathcal{X},\ast)$. For every $n>0$, we have:
\begin{enumerate}
\item $\mathcal{H}^{n\ast}$ is a periodic partition and has the same period as $\mathcal{H}$, i.e., $\per(\mathcal{H}^{n\ast})=\per(\mathcal{H})$.
\item $|\mathcal{H}^{n\ast}|=|\mathcal{H}|$.
\end{enumerate}
\end{myprop}
\begin{proof}
see Appendix \ref{appB}.
\end{proof}

\begin{mylem}
\label{lemsize}
$\|\mathcal{H}^{\ast}\|_{\wedge}\geq\|\mathcal{H}\|_{\wedge}$ and $\|\mathcal{H}^{\ast}\|_{\vee}\geq\|\mathcal{H}\|_{\vee}$.
\end{mylem}
\begin{proof}
Let $A\in\mathcal{H}$ be such that $A=\|\mathcal{H}\|_{\vee}$, then $A\ast A\in\mathcal{H}^{\ast}$. Thus,
$\|\mathcal{H}^{\ast}\|_{\vee}\geq |A\ast A|\geq |A|=\|\mathcal{H}\|_{\vee}$.

Now let $B$ and $C$ be two elements of $\mathcal{H}$ such that $|B\ast C|=\|\mathcal{H}^{\ast}\|_{\wedge}$. We have 
$|B\ast C|\geq |B|\geq \|\mathcal{H}\|_{\wedge}$. This implies that $\|\mathcal{H}^{\ast}\|_{\wedge}\geq\|\mathcal{H}\|_{\wedge}$.
\end{proof}

\begin{myprop}
Let $\mathcal{H}$ be a stable partition of $(\mathcal{X},\ast)$. For every $n>0$, $\mathcal{H}^{n\ast}$ is a stable partition satisfying $\per(\mathcal{H}^{n\ast})=\per(\mathcal{H})$ and $\|\mathcal{H}^{n\ast}\|=\|\mathcal{H}\|$.
\end{myprop}
\begin{proof}
Proposition \ref{propPerPer} shows that $\mathcal{H}^{n\ast}$ is a periodic partition of period $\per(\mathcal{H}^{n\ast})=\per(\mathcal{H})$. It remains to show that $\mathcal{H}^{n\ast}$ is balanced and that $\|\mathcal{H}^{n\ast}\|=\|\mathcal{H}\|$. Let $p>0$ be the smallest multiple of $\per(\mathcal{H})$ which is greater than $n$, i.e., $p=\min\{k\cdot \per(\mathcal{H}):\;k> 0,\;k\cdot \per(\mathcal{H})>n\}$. We have $\mathcal{H}^{p\ast}=\mathcal{H}$ since $\per(\mathcal{H})$ divides $p$. By Lemma \ref{lemsize} we have:
\begin{itemize}
\item $\|\mathcal{H}\|=\|\mathcal{H}\|_{\wedge}\leq\|\mathcal{H}^{\ast}\|_{\wedge}\leq\ldots\leq \|\mathcal{H}^{n\ast}\|_{\wedge} \leq\ldots\leq \|\mathcal{H}^{p\ast}\|_{\wedge}=\|\mathcal{H}\|_{\wedge}=\|\mathcal{H}\|$.
\item $\|\mathcal{H}\|=\|\mathcal{H}\|_{\vee}\leq\|\mathcal{H}^{\ast}\|_{\vee}\leq\ldots\leq \|\mathcal{H}^{n\ast}\|_{\vee} \leq\ldots\leq \|\mathcal{H}^{p\ast}\|_{\vee}=\|\mathcal{H}\|_{\vee}=\|\mathcal{H}\|$.
\end{itemize}

Therefore, $\|\mathcal{H}^{n\ast}\|_{\wedge}=\|\mathcal{H}^{n\ast}\|_{\vee}=\|\mathcal{H}\|$, which means that for every $A\in\mathcal{H}^{n\ast}$ we have $|A|=\|\mathcal{H}\|$. We conclude that $\mathcal{H}^{n\ast}$ is balanced and $\|\mathcal{H}^{n\ast}\|=\|\mathcal{H}\|$.
\end{proof}

\begin{mylem}
\label{stasta}
If $\ast$ is ergodic then every periodic partition is stable.
\end{mylem}
\begin{proof}
Let $\mathcal{H}$ be a periodic partition of $\mathcal{X}$. We only need to show that $\mathcal{H}$ is balanced.

Let $n=\per(\mathcal{H})$ and $m=\min\{kn: k>0\;\text{and}\;kn>\con(\ast)\}$. Clearly, $\mathcal{H}^{m\ast}=\mathcal{H}$. Moreover, statement 8 of Proposition \ref{lemerg} shows that all the elements of $\mathcal{X}$ are $\ast$-connectable to each other in $m$ steps. Let $H\in\mathcal{H}$ be chosen such that $|H|$ is maximal and let $H'$ be any element of $\mathcal{H}$. Let $h\in H$ and $h'\in H'$. We have $h\stackrel{\ast,m}{\longrightarrow} h'$ so there exist $m$ elements $x_0,\ldots,x_{m-1}\in \mathcal{X}$ satisfying $(\ldots((h\ast x_0)\ast x_1)\ldots\ast x_{m-1})=h'$.

Since $\mathcal{H}$ covers $\mathcal{X}$, then each of $\mathcal{H}^{\ast}$, $\mathcal{H}^{2\ast}$, \ldots, and $\mathcal{H}^{(m-1)\ast}$ covers $\mathcal{X}$ as well. And so there exist $X_0\in\mathcal{H}$, $X_1\in\mathcal{H}^{\ast}$, \ldots, and $X_{m-1}\in\mathcal{H}^{(m-1)\ast}$ such that $x_0\in X_0$, $x_1\in X_1$, \ldots, and $x_{m-1}\in X_{m-1}$. Now since $(\ldots((h\ast x_0)\ast x_1)\ldots\ast x_{m-1})=h'$ and since $h\in H$, we have $h'\in H'' := (\ldots((H\ast X_0)\ast X_1)\ldots\ast X_{m-1})$. From the definition of $H''$, we have $H''\in \mathcal{H}^{m\ast}=\mathcal{H}$. Moreover, $h'\in H'\cap H''$, so $H'=H''$ since $\mathcal{H}$ is a partition. We conclude that $H' = (\ldots((H\ast X_0)\ast X_1)\ldots\ast X_{m-1})$ which implies that $|H'|\geq |H|$. On the other hand, we have $|H|\geq |H'|$ since $H$ was chosen so that $|H|$ is maximal. We conclude that $|H'|=|H|$ for all $H'\in\mathcal{H}$, hence $\mathcal{H}$ is balanced.
\end{proof}

\begin{myrem}
The ergodicity condition in the previous Lemma cannot be replaced by irreducibility. Consider the following irreducible (but not ergodic) operation:

\begin{center}
  \begin{tabular}{ | c || c | c | c | c | }
    \hline
    $\ast$ & 0 & 1 & 2 & 3 \\ \hline \hline
    0 & 2 & 3 & 2 & 2 \\ \hline
    1 & 3 & 2 & 3 & 3 \\ \hline
    2 & 0 & 0 & 0 & 1 \\ \hline
    3 & 1 & 1 & 1 & 0 \\ \hline
  \end{tabular}
\end{center}

Although the partition $\mathcal{H}=\{\{0,1\},\{2\},\{3\}\}$ is not balanced, we have $\mathcal{H}^{2\ast}=\mathcal{H}$.
\end{myrem}

The following proposition shows that the concept of periodic partitions generalizes the concept of quotient groups:

\begin{myprop}
Let $(G,\ast)$ be a finite group, and let $\mathcal{H}$ be a periodic partition of $(G,\ast)$. There exists a normal subgroup $H$ of $G$ such that $\mathcal{H}$ is the quotient group of $G$ by $H$ (denoted by $G/H$).
\end{myprop}
\begin{proof}
Since every group operation is ergodic, Lemma \ref{stasta} implies that $\mathcal{H}$ is stable, i.e., it is also balanced.

Let $H$ be the element of $\mathcal{H}$ containing the neutral element $e$ of $G$. For every $H'\in\mathcal{H}$, we have $|H'|=|H\ast H'|=|H'\ast H|=\|\mathcal{H}\|$ since $H\ast H'\in\mathcal{H}^{\ast}$, $H'\ast H\in\mathcal{H}^{\ast}$ and $\|\mathcal{H}^\ast\|=\|\mathcal{H}\|$. On the other hand, we have $H'=e\ast H'\subset H\ast H'$ and $H'=H'\ast e \subset H'\ast H$. We conclude that $H\ast H' = H'\ast H = H'$. Therefore,
\begin{itemize}
\item $H\ast H=H$, hence $x\ast y\in H$ for every $x,y\in H$.
\item For every $x\in H$, we have $|H\ast x|=|H|$. On the other hand, $H\ast x\subset H\ast H=H$. Therefore, $H\ast x=H$ which implies that $e\in H\ast x$ and so there exists $x'\in H$ such that $x'\ast x=e$. We conclude that the inverse of any element of $H$ is also in $H$.
\item For every $x\in G$ let $H_x\in\mathcal{H}$ be such that $x\in H_x$. We have $x\ast H \subset H_x\ast H=H_x$ and $|x\ast H|\stackrel{(a)}{=}|H|=|H_x|$, where (a) follows from the fact that $\ast$ is a group operation. Therefore, $x\ast H=H_x$. Similarly, we can show that $H\ast x=H_x$. Hence $x\ast H=H\ast x=H_x$ for every $x\in G$.
\end{itemize}
We conclude that $H$ is a normal subgroup of $G$, and $\mathcal{H}$ is the quotient group of $G$ by $H$.
\end{proof}

\begin{mydef}
A periodic partition $\mathcal{H}_1$ is said to be a sub-periodic partition of another periodic partition $\mathcal{H}_2$ if for any $H_1\in\mathcal{H}_1$, there exists $H_2\in\mathcal{H}_2$ such that $H_1\subset H_2$. We denote this relation by $\mathcal{H}_1 \preceq\mathcal{H}_2$, and we say that $\mathcal{H}_1$ is finer than $\mathcal{H}_2$.

If $\mathcal{H}_1$ and $\mathcal{H}_2$ are two stable partitions satisfying $\mathcal{H}_1 \preceq\mathcal{H}_2$, we say that $\mathcal{H}_1$ is a sub-stable partition of $\mathcal{H}_2$ (in such case, we clearly have $\|\mathcal{H}_1\|$ divides $\|\mathcal{H}_2\|$).
\end{mydef}

\begin{myrem}
Let $(G,\ast)$ be a group and let $\mathcal{H}_1$ be a sub-periodic partition of a periodic partition $\mathcal{H}_2$. If $H_{\mathcal{H}_1}$ and $H_{\mathcal{H}_2}$ are the normal subgroups associated with $\mathcal{H}_1$ and $\mathcal{H}_2$ respectively, then $H_{\mathcal{H}_1}$ is a normal subgroup of $H_{\mathcal{H}_2}$.
\end{myrem}

\begin{mydef}
\label{defWedge}
For any two partitions $\mathcal{H}_1$ and $\mathcal{H}_2$ of a set $\mathcal{X}$, we define:
$$\mathcal{H}_1 \wedge \mathcal{H}_2=\{H_1\cap H_2:\; H_1\in\mathcal{H}_1,\; H_2\in\mathcal{H}_2,\; H_1\cap H_2\neq \o\}.$$
\end{mydef}

\begin{myprop}
\label{lemWedge}
If $\mathcal{H}_1$ and $\mathcal{H}_2$ are periodic partitions then $\mathcal{H}_1\wedge\mathcal{H}_2$ is a periodic partition of period at most $\lcm\{\per(\mathcal{H}_1),\per(\mathcal{H}_2)\}$. Moreover, we have $(\mathcal{H}_1\wedge\mathcal{H}_2)^{n\ast}=\mathcal{H}_1^{n\ast}\wedge \mathcal{H}_2^{n\ast}$ for every $n\geq 0$.
\end{myprop}
\begin{proof}
See Appendix \ref{appB}.
\end{proof}

\begin{mycor}
\label{corWedge}
Let $\ast$ be an ergodic operation. If $\mathcal{H}_1$ and $\mathcal{H}_2$ are two stable partitions then $\mathcal{H}_1\wedge\mathcal{H}_2$ is a stable partition of period at most $\lcm\{\per(\mathcal{H}_1),\per(\mathcal{H}_2)\}$.
\end{mycor}
\begin{proof}
The corollary follows from Proposition \ref{lemWedge} and Lemma \ref{stasta}.
\end{proof}

\begin{myrem}
Let $(G,\ast)$ be a group. If $\mathcal{H}_1$ and $\mathcal{H}_2$ are two periodic partitions of $(G,\ast)$, then $H_{\mathcal{H}_1\wedge\mathcal{H}_2}=H_{\mathcal{H}_1}\cap H_{\mathcal{H}_2}$.
\end{myrem}

\begin{myrem}
\label{remIrrInd}
The ergodicity condition in Corollary \ref{corWedge} cannot be replaced by irreducibility. Consider the following irreducible (but not ergodic) operation:

\begin{center}
  \begin{tabular}{ | c || c | c | c | c | c | c | c | c | }
    \hline
    $\ast$ & 0 & 1 & 2 & 3 & 4 & 5 & 6 & 7 \\ \hline \hline
    0 & 4 & 5 & 6 & 7 & 4 & 4 & 4 & 4 \\ \hline
    1 & 5 & 4 & 7 & 6 & 5 & 5 & 5 & 5 \\ \hline
    2 & 6 & 7 & 4 & 5 & 6 & 6 & 6 & 6 \\ \hline
    3 & 7 & 6 & 5 & 4 & 7 & 7 & 7 & 7 \\ \hline
    4 & 0 & 0 & 0 & 0 & 0 & 1 & 2 & 3 \\ \hline
    5 & 1 & 1 & 1 & 1 & 1 & 0 & 3 & 2 \\ \hline
    6 & 2 & 2 & 2 & 2 & 2 & 3 & 0 & 1 \\ \hline
    7 & 3 & 3 & 3 & 3 & 3 & 2 & 1 & 0 \\ \hline
  \end{tabular}
\end{center}

Define: $$\mathcal{H}_1=\big\{\{0,1\},\{2,3\},\{4,5\},\{6,7\}\big\},$$
$$\mathcal{H}_2=\big\{\{0,2\},\{1,3\},\{4,5\},\{6,7\}\big\}.$$
While both $\mathcal{H}_1$ and $\mathcal{H}_2$ are stable partitions of periods 1 and 2 respectively, the partition $\mathcal{H}_1\wedge\mathcal{H}_2=\big\{\{0\},\{1\},\{2\},\{3\},\{4,5\},\{6,7\}\big\}$ is periodic but it is not stable as it is not balanced.
\end{myrem}

\section{The residue of a stable partition}

Let $\mathcal{H}$ be a stable partition. Let $H\in\mathcal{H}$ and $x\in H$. For any sequence $(X_n)_{n\geq 0}$ satisfying $X_n\in\mathcal{H}^{n\ast}$ for all $n\geq 0$, define the sequences $(A_n)_{n\geq 0}$ and $(H_n)_{n\geq 0}$ recursively as follows:
\begin{itemize}
\item $A_0=\{x\}$ and $H_0=H$.
\item $A_n=A_{n-1}\ast X_{n-1}=(\ldots((x\ast X_0)\ast X_1)\ldots\ast X_{n-1})$.
\item $H_n=H_{n-1}\ast X_{n-1}=(\ldots((H\ast X_0)\ast X_1)\ldots\ast X_{n-1})$.
\end{itemize}
Since $x\in H$, we can show by induction on $n$ that $A_n\subset H_n\in \mathcal{H}^{n\ast}$ and so $|A_n|\leq|H_n|=\|\mathcal{H}^{n\ast}\|=\|\mathcal{H}\|$ for all $n\geq 0$. Therefore, $|H_n|$ is constant. On the other hand, $|A_n|\geq |A_{n-1}|$ since $A_n=A_{n-1}\ast X_{n-1}$. Hence, $|A_n|$ is increasing and it is upper bounded by $\|\mathcal{H}\|$.

Does $|A_n|$ reach $\|\mathcal{H}\|$ or does $|A_n|$ remain strictly less than $\|\mathcal{H}\|$ for all $n\geq 0$? In other words, do we have $A_n=H_n$ for some $n>0$ or does $A_n$ remain a strict subset of $H_n$ for all $n\geq0$? The answer depends of course on the sequence $(X_n)_{n\geq 0}$, so one can ask: Is it possible to choose at least one sequence $(X_n)_{n\geq 0}$ for which $|A_n|=\|\mathcal{H}\|$ and $A_n=H_n$ for some $n>0$?

What are the stable partitions $\mathcal{H}$ for which it is always possible to reach a set in $\mathcal{H}^{n\ast}$ for some $n>0$ starting from an arbitrary singleton in $\mathcal{X}$ and then recursively multiplying on the right by sets chosen from $\mathcal{H}^{i\ast}$ ($0\leq i<n$)?

It is easy to see that for the trivial stable partition $\mathcal{H}=\{\mathcal{X}\}$, the above condition is equivalent to ergodicity. Therefore, satisfying the above condition for every stable partition is a stronger notion of ergodicity. Strong ergodicity turns out to be important for polarization theory as we will see in Part II of this paper \cite{RajErgII}. In this section, we introduce the notions and concepts that are necessary to understand strong ergodicity.

\begin{mynot}
Let $\mathfrak{X}=(X_i)_{0\leq i<k}$ be a sequence of subsets $X_i$ of $\mathcal{X}$. We denote the length $k$ of the sequence $\mathfrak{X}$ by $|\mathfrak{X}|$.

For any $A\subset\mathcal{X}$, we denote $(\ldots((A\ast X_0)\ast X_1)\ldots)\ast X_{k-1})$ by $A\ast \mathfrak{X}$. If $A=\{a\}$, we write $a\ast \mathfrak{X}$ to denote $\{a\}\ast \mathfrak{X}$.

The $n^{th}$ power of the sequence $\mathfrak{X}=(X_i)_{0\leq i<k}$ is the sequence $\mathfrak{X}^n=(X_i')_{0\leq i<kn}$, where $X_i'=X_{i\bmod k}$ for $0\leq i<kn$. I.e., $\mathfrak{X}^n$ is obtained by concatenating $n$ copies of $\mathfrak{X}$.
\end{mynot}

\begin{mydef}
Let $\mathcal{H}$ be a stable partition of $(\mathcal{X},\ast)$ where $\ast$ is uniformity preserving. A sequence $\mathfrak{X}=(X_i)_{0\leq i<k}$ is said to be $\mathcal{H}$-\emph{sequence} if $X_0\in\mathcal{H}$, $X_1\in\mathcal{H}^{\ast}$, \ldots, $X_{k-1}\in\mathcal{H}^{(k-1)\ast}$. If we also have that $\per(\mathcal{H})$ divides $|\mathfrak{X}|=k$, we say that the sequence is $\mathcal{H}$-\emph{repeatable}.

An $\mathcal{H}$-repeatable sequence $\mathfrak{X}$ is said to be $\mathcal{H}$-\emph{augmenting} if $A\subset A\ast \mathfrak{X}$ for all $A\subset \mathcal{X}$.
\end{mydef}

\begin{myrem}
\label{remtemdem123}
If $\mathfrak{X}$ is $\mathcal{H}$-repeatable, then $\mathfrak{X}^l$ is an $\mathcal{H}$-sequence for every $l>0$. This is not necessarily true if $\mathfrak{X}$ is an $\mathcal{H}$-sequence which is not repeatable.

If a sequence is $\mathcal{H}$-augmenting then it is also $\mathcal{H}$-repeatable by definition. Therefore, whenever we need to show that a sequence is $\mathcal{H}$-augmenting, we have to show first that it is $\mathcal{H}$-repeatable.

If $\mathfrak{X}$ is $\mathcal{H}$-augmenting then $\mathfrak{X}^l$ is $\mathcal{H}$-augmenting for every $l>0$.
\end{myrem}

\begin{mythe}
\label{theres}
Let $\mathcal{H}$ be a stable partition of $(\mathcal{X},\ast)$ where $\ast$ is ergodic. There exists a unique sub-stable partition $\mathcal{K}_{\mathcal{H}}$ of $\mathcal{H}$ such that:
\begin{itemize}
\item For every $K\in\mathcal{K}_{\mathcal{H}}$ and every $\mathcal{H}$-sequence $\mathfrak{X}$, $K\ast\mathfrak{X} \in {\mathcal{K}_{\mathcal{H}}}^{|\mathfrak{X}|\ast}$.
\item For every $K\in\mathcal{K}_{\mathcal{H}}$ and every $x\in K$, there exists an $\mathcal{H}$-augmenting sequence $\mathfrak{X}$ such that $x\ast \mathfrak{X}=K$.
\item For every $K\in\mathcal{K}_{\mathcal{H}}$, every $x\in K$, and every $\mathcal{H}$-augmenting sequence $\mathfrak{X}'$, $x\ast \mathfrak{X}'\subset K$.
\end{itemize}
$\mathcal{K}_{\mathcal{H}}$ is called the \emph{first residue} of the stable partition $\mathcal{H}$. We also have ${\mathcal{K}_{\mathcal{H}}}^{l\ast}=\mathcal{K}_{\mathcal{H}^{l\ast}}$ for all $l\geq 0$.
\end{mythe}
\begin{proof}
See Appendix \ref{appC}.
\end{proof}

\begin{myrem}
Theorem \ref{theres} implies that an ergodic operation is strongly ergodic if and only if $\mathcal{K}_{\mathcal{H}}=\mathcal{H}$ for every stable partition $\mathcal{H}$ of $\mathcal{X}$. This will be explained and proven in detail in section VI.
\end{myrem}

\begin{myrem}
It is possible to prove a more general theorem for the periodic partitions of an arbitrary uniformity preserving operation:

Let $\mathcal{H}$ be a periodic partition of $(\mathcal{X},\ast)$ where $\ast$ is an arbitrary uniformity preserving operation, there exists a unique sub-periodic partition $\mathcal{K}_{\mathcal{H}}$ of $\mathcal{H}$ such that:
\begin{itemize}
\item For every $K\in\mathcal{K}_{\mathcal{H}}$ and every $\mathcal{H}$-sequence $\mathfrak{X}$, $K\ast\mathfrak{X} \in {\mathcal{K}_{\mathcal{H}}}^{|\mathcal{X}|\ast}$.
\item For every $K\in\mathcal{K}_{\mathcal{H}}$ and every $x\in K$, there exists an $\mathcal{H}$-augmenting sequence $\mathfrak{X}$ such that $x\ast \mathfrak{X}=K$.
\item For every $K\in\mathcal{K}_{\mathcal{H}}$, every $x\in K$, and every $\mathcal{H}$-augmenting sequence $\mathfrak{X}'$, $x\ast \mathfrak{X}'\subset K$.
\end{itemize}
$\mathcal{K}_{\mathcal{H}}$ is called the \emph{first residue} of the periodic partition $\mathcal{H}$. We also have ${\mathcal{K}_{\mathcal{H}}}^{l\ast}=\mathcal{K}_{\mathcal{H}^{l\ast}}$ for all $l\geq 0$.

We will not prove this general theorem here since Theorem \ref{theres} is sufficient for our purposes. The proof of the general theorem is more complicated but follows similar steps as the proof of Theorem \ref{theres}.

Note that if the operation $\ast$ is not ergodic, $\mathcal{K}_{\mathcal{H}}$ may not be a stable partition even if $\mathcal{H}$ is a stable partition. Consider the following irreducible (but not ergodic) operation:

\begin{center}
  \begin{tabular}{ | c || c | c | c | c | c | c | c | c | }
    \hline
    $\ast$ & 0 & 1 & 2 & 3 & 4 & 5 & 6 & 7 \\ \hline \hline
    0 & 4 & 5 & 4 & 5 & 4 & 4 & 4 & 4 \\ \hline
    1 & 5 & 4 & 5 & 4 & 5 & 5 & 5 & 5 \\ \hline
    2 & 6 & 7 & 6 & 7 & 6 & 6 & 6 & 6 \\ \hline
    3 & 7 & 6 & 7 & 6 & 7 & 7 & 7 & 7 \\ \hline
    4 & 2 & 2 & 2 & 2 & 2 & 3 & 2 & 3 \\ \hline
    5 & 3 & 3 & 3 & 3 & 3 & 2 & 3 & 2 \\ \hline
    6 & 0 & 0 & 0 & 0 & 0 & 1 & 0 & 1 \\ \hline
    7 & 1 & 1 & 1 & 1 & 1 & 0 & 1 & 0 \\ \hline
  \end{tabular}
\end{center}
Let $\mathcal{H}=\big\{\{0,2\},\{1,3\},\{4,5\},\{6,7\}\big\}$, which is a stable partition of period 2. The reader can check that $\mathcal{K}_{\mathcal{H}}=\big\{\{0\},\{1\},\{2\},\{3\},\{4,5\},\{6,7\}\big\}$ which is periodic but not stable as it is not balanced.
\end{myrem}

\begin{mydef}
Let $\mathcal{H}$ be a stable partition of $(\mathcal{X},\ast)$ where $\ast$ is ergodic. For every $n\geq 0$, we define the $n^{th}$ residue $\mathcal{R}_n(\mathcal{H})$ of $\mathcal{H}$ recursively as follows:
\begin{itemize}
\item $\mathcal{R}_0(\mathcal{H})=\mathcal{H}$.
\item $\mathcal{R}_1(\mathcal{H})=\mathcal{K}_{\mathcal{H}}$.
\item $\mathcal{R}_{n+1}(\mathcal{H})=\mathcal{R}_1(\mathcal{R}_n(\mathcal{H}))=\mathcal{K}_{ \mathcal{R}_n(\mathcal{H})}$ for every $n\geq 1$.
\end{itemize}
The \emph{residual degree} $\degg_{\mathcal{R}}(\mathcal{H})$ of $\mathcal{H}$ is the smallest integer $n\geq 0$ satisfying $\mathcal{R}_{n+1}(\mathcal{H})=\mathcal{R}_n(\mathcal{H})$. The residue of $\mathcal{H}$ is defined as $\mathcal{R}(\mathcal{H}):=\mathcal{R}_{\degg_{\mathcal{R}}(\mathcal{H})}(\mathcal{H})$. Clearly $\mathcal{R}_1(\mathcal{R}(\mathcal{H}))=\mathcal{K}_{\mathcal{R}(\mathcal{H})}=\mathcal{R}(\mathcal{H})$ and $\mathcal{R}(\mathcal{R}(\mathcal{H}))=\mathcal{R}(\mathcal{H})$.
\end{mydef}

\begin{myrem}
In the application to polarization theory, we will only need the first residue. We just note here that for every $n\geq 0$, there exists an ergodic operation and a stable partition $\mathcal{H}$ of residual degree $n$. In other words, there are stable partitions of arbitrary residual degrees.
\end{myrem}

\section{Strongly ergodic operations}

\begin{mydef}
\label{defdef}
A uniformity preserving operation $\ast$ is said to be \emph{strongly ergodic} if for any stable partition $\mathcal{H}$ and for any $x\in\mathcal{X}$, there exists an integer $n=n(x,\mathcal{H})$ such that for any $H\in\mathcal{H}^{n\ast}$, there exists an $\mathcal{H}$-sequence $\mathfrak{X}_{x,H}$ of length $n$ such that $x\ast\mathfrak{X}_{x,H}=H$.
\end{mydef}

\begin{mythe}
\label{thestrong}
We have the following:
\begin{enumerate}
\item If $\ast$ is strongly ergodic then it is ergodic.
\item If $\ast$ is strongly ergodic, there exists an integer $d>0$ such that for any $s\geq d$, any stable partition $\mathcal{H}$,  any $x\in\mathcal{X}$ and any $H\in\mathcal{H}^{s\ast}$, there exists an $\mathcal{H}$-sequence $\mathcal{X}_{x,H}$ of length $s$ satisfying $x\ast \mathcal{X}_{x,H}=H$. If $d$ is minimal with this property, we call it the \emph{strong connectability} of $\ast$, and we denote it by $\scon(\ast)$.
\item If $\ast$ is ergodic, $\ast$ is strongly ergodic if and only if $\mathcal{K}_{\mathcal{H}}=\mathcal{H}$ for every stable partition $\mathcal{H}$ (i.e., every stable partition $\mathcal{H}$ is its own residue, and so the residual degree is zero).
\item If $\ast$ is a quasigroup operation then it is strongly ergodic.
\end{enumerate}
\end{mythe}
\begin{proof}
1) Suppose that $\ast$ is strongly ergodic and consider the trivial stable partition $\{\mathcal{X}\}$. For every $x\in\mathcal{X}$, there exists $n_x>0$ such that $x\ast(\mathcal{X})^{n_x}=\mathcal{X}$. This shows that for every $y\in\mathcal{X}$, $x\stackrel{\ast,n_x}{\longrightarrow}y$ which shows that $\ast$ is irreducible. Let $n=\per(\ast)$ and let $H_0,\ldots,H_{n-1}$ be the equally sized subsets of $\mathcal{X}$ given by the fourth point of Proposition \ref{lemerg}.

Let $x\in H_0$. We have $\mathcal{X}=x\ast(\mathcal{X})^{n_x}\subset H_0\ast(\mathcal{X})^{n_x}=H_{n_x\bmod n}$, where the last equality follows from the fourth point of Proposition \ref{lemerg}. Therefore, $H_{n_x\bmod n}=\mathcal{X}$ which implies that $n=1$ since $\{H_0,\ldots,H_{n-1}\}$ is a partition. Therefore, $\per(\ast)=1$ and so $\ast$ is ergodic by the seventh point of Proposition \ref{lemerg}.

\vspace*{2mm}
2)  Let $\ast$ be strongly ergodic, and define $\displaystyle d=\max_{x,\mathcal{H}}n(x,\mathcal{H}),$ where $n(x,\mathcal{H})$ is as in Definition \ref{defdef}. Now fix $x\in\mathcal{X}$ and fix a stable partition $\mathcal{H}$. Let $s\geq d$ and fix $H\in\mathcal{H}^{s\ast}$. If $s=n(x,\mathcal{H})$, there is nothing to prove. Now suppose that $s>n:=n(x,H)$, then there exists $H'\in\mathcal{H}^{n\ast}$ and an $\mathcal{H}^{n\ast}$-sequence $\mathfrak{X}$ of length $s-n$ such that $H'\ast \mathfrak{X}=H$. Moreover, there exists an $\mathcal{H}$-sequence $\mathfrak{X}_{x,H'}$ of length $n$ such that $x\ast \mathfrak{X}_{x,H'}=H'$. We conclude that $x\ast(\mathfrak{X}_{x,H'},\mathfrak{X})=H$.

\vspace*{2mm}
3) Let $\mathcal{H}$ be a stable partition of $(\mathcal{X},\ast)$ where $\ast$ is strongly ergodic, and let $x\in\mathcal{X}$, $K\in\mathcal{K}_{\mathcal{H}}$ and $H\in\mathcal{H}$ be chosen so that $x\in K\subset H$. Let $s= \scon(\ast)\cdot\per(\mathcal{H})$. We have $\mathcal{H}^{s\ast}=\mathcal{H}$ since $\per(\mathcal{H})$ divides $s$. Now since $s\geq \scon(\ast)$ and $H\in\mathcal{H}=\mathcal{H}^{s\ast}$, there exists an $\mathcal{H}$-sequence $\mathfrak{X}_{x,H}$ of length $s$ such that $x\ast\mathfrak{X}_{x,H}=H$. We have $x\in H=x\ast\mathfrak{X}_{x,H}\subset K\ast \mathfrak{X}_{x,H}$, so $x\in K\ast \mathfrak{X}_{x,H}$ which implies that $K\cap (K\ast \mathfrak{X}_{x,H})\neq \o$ (since we also have $x\in K$). On the other hand, Theorem \ref{theres} implies that $K \ast\mathfrak{X}_{x,H}\in {\mathcal{K}_{\mathcal{H}}}^{s\ast}=\mathcal{K}_{\mathcal{H}}$. Therefore, $K\ast \mathfrak{X}_{x,H}=K$ since $\mathcal{K}_{\mathcal{H}}$ is a partition. We conclude that $H=x\ast \mathfrak{X}_{x,H}\subset K\ast \mathfrak{X}_{x,H}=K$ which implies that $H=K$ since we also have $K\subset H$. Therefore, $\|\mathcal{K}_{\mathcal{H}}\|=\|\mathcal{H}\|$ and so $\mathcal{K}_{\mathcal{H}}=\mathcal{H}$.

Now suppose that $\ast$ is an ergodic operation which satisfies $\mathcal{K}_{\mathcal{H}}=\mathcal{H}$ for every stable partition $\mathcal{H}$. Let $x\in\mathcal{X}$ and let $\mathcal{H}$ be a stable partition. Let $k=\con(\ast)\cdot\per(\mathcal{H})\geq \con(\ast)$, and for each $H\in\mathcal{H}$ fix $x_H\in H$ and let $\mathfrak{X}_{H}$ be an $\mathcal{H}$-augmenting sequence such that $x_H\ast\mathfrak{X}_H=H$ (such $\mathfrak{X}_{H}$ exists due to Theorem \ref{theres}). Define $\displaystyle n(x,H)=k+\sum_{H\in\mathcal{H}}|\mathfrak{X}_{H}|$ and define $\mathfrak{X}'$ to be the $\mathcal{H}$-augmenting sequence obtained by concatenating all the $\mathfrak{X}_{H}$ sequences (the order of the concatenation is not important). It is easy to see that $x_H\ast\mathfrak{X}'=H$ for all $H\in\mathcal{H}$: We have $x_H\ast \mathfrak{X}'\subset H$ from Theorem \ref{theres}. On the other hand, $H\subset x_H\ast \mathfrak{X}'$ follows from the fact that $\mathfrak{X}'$ is the concatenation of a collection of $\mathcal{H}$-augmenting sequences containing $\mathfrak{X}_H$ and that $x_H\ast\mathfrak{X}_H=H$. We also have $\displaystyle |\mathfrak{X}'|=\sum_{H\in\mathcal{H}}|\mathfrak{X}_{H}|$. Now since $k\geq \con(\ast)$, it follows from Proposition \ref{lemerg} that for every $H\in\mathcal{H}$ there exists a sequence $x_0,\ldots,x_{k-1}$ satisfying $(\ldots((x\ast x_0)\ast x_1) \ldots\ast x_{k-1})=x_H$. Let $\mathfrak{X}_H'=(X_0,\ldots,X_{k-1})$ be an $\mathcal{H}$-sequence of length $k$ such that $x_i\in X_i$ for all $0\leq i<k$. Clearly, $x_H\in x\ast\mathfrak{X}_H'$. It is easy to see that the sequence $\mathfrak{X}_H''=(\mathfrak{X}_{H}',\mathfrak{X}')$ is of length $n(x,\mathcal{H})$ and satisfies $H\subset x\ast\mathfrak{X}_H''$. Now let $H_x\in\mathcal{H}$ be chosen so that $x\in H_x$. Since $H_x\in\mathcal{H}=\mathcal{K}_{\mathcal{H}}$, Theorem \ref{theres} implies that we have $H_x\ast\mathfrak{X}_H''\in\mathcal{K}_{\mathcal{H}^{n(x,\mathcal{H})\ast}}=\mathcal{K}_{\mathcal{H}}=\mathcal{H}$ (note that $\mathcal{H}^{n(x,\mathcal{H})\ast}=\mathcal{H}$ since $\per(\mathcal{H})$ divides $n(x,\mathcal{H})$). We conclude that $H\subset x\ast\mathfrak{X}_H'' \subset H_x\ast\mathfrak{X}_H''\in \mathcal{H}$, which implies that $H=x\ast\mathfrak{X}_H''=H_x\ast\mathfrak{X}_H''$ since we have $H\in\mathcal{H}$ and $\mathcal{H}$ is a partition. Therefore, for every $H\in\mathcal{H}=\mathcal{H}^{n(x,\mathcal{H})\ast}$, there exists an $\mathcal{H}$-sequence $\mathfrak{X}_H''$ of length $n(x,\mathcal{H})$ such that $x\ast\mathfrak{X}_H''=H$. Thus, $\ast$ is a strongly ergodic operation.

\vspace*{2mm}
4)  Let $\mathcal{H}$ be a stable partition of a quasigroup operation $\ast$. For any $K\in\mathcal{K}_{\mathcal{H}}$ and any $x\in K$, there exists an $\mathcal{H}$-augmenting sequence $\mathfrak{X}=(X_i)_{0\leq i< k}$ such that $K=x\ast\mathfrak{X}$, which implies that $|K|=|x\ast \mathfrak{X}|=\big|\big(x\ast(X_i)_{0\leq i< k-1}\big)\ast X_{k-1}\big|\stackrel{(a)}{\geq} |X_{k-1}|=\|\mathcal{H}\|$, where (a) is true because $\ast$ is a quasigroup operation. We conclude that $\|\mathcal{K}_{\mathcal{H}}\|=\|\mathcal{H}\|$ which implies that $\mathcal{K}_{\mathcal{H}}=\mathcal{H}$.
\end{proof}

\begin{myrem}
\label{remexamperg}
While every strongly ergodic operation $\ast$ is ergodic, the converse is not true. Consider the following operation:
\vspace*{-1mm}
\begin{center}
  \begin{tabular}{ | c || c | c | c | c | }
    \hline
    $\ast$ & 0 & 1 & 2 & 3 \\ \hline \hline
    0 & 2 & 2 & 0 & 0 \\ \hline
    1 & 3 & 3 & 1 & 1 \\ \hline
    2 & 1 & 1 & 3 & 3 \\ \hline
    3 & 0 & 0 & 2 & 2 \\ \hline
  \end{tabular}
\end{center}
\vspace*{4mm}
The first residue of the stable partition $\mathcal{H}=\big\{\{0,1\},\{2,3\}\big\}$ is $\mathcal{K}_{\mathcal{H}}=\big\{\{0\},\{1\},\{2\},\{3\}\big\}\neq \mathcal{H}$.

Also, a strongly ergodic operation need not be a quasigroup operation, here is an example:
\begin{center}
  \begin{tabular}{ | c || c | c | c | c | }
    \hline
    $\ast$ & 0 & 1 & 2 & 3 \\ \hline \hline
    0 & 3 & 3 & 3 & 3 \\ \hline
    1 & 0 & 1 & 0 & 0 \\ \hline
    2 & 1 & 0 & 1 & 1 \\ \hline
    3 & 2 & 2 & 2 & 2 \\ \hline
  \end{tabular}
\end{center}
\end{myrem}

\section{Generated stable partitions}

\label{sec7}

\begin{mydef}
Let $\mathcal{A}$ and $\mathcal{B}$ be two sets of subsets of $\mathcal{X}$. We say that $\mathcal{A}$ is \emph{finer} than $\mathcal{B}$ (or $\mathcal{B}$ is \emph{coarser} than $\mathcal{A}$) if for every $A\in\mathcal{A}$ there exists $B\in\mathcal{B}$ such that $A\subset B$. We write  $\mathcal{A}\preceq \mathcal{B}$ to denote the relation ``$\mathcal{A}$ is finer than $\mathcal{B}$".
\end{mydef}

Let $\mathcal{A}$ be a set of subsets of $\mathcal{X}$. Is it possible to find a periodic partition of $(\mathcal{X},\ast)$ which is coarser than $\mathcal{A}$ and finer than every other periodic partition that is coarser than $\mathcal{A}$? Similarly, is it possible to find a stable partition of $(\mathcal{X},\ast)$ which is coarser than $\mathcal{A}$ and finer than every other stable partition that is coarser than $\mathcal{A}$? The following answer these two questions.

\begin{myprop}
\label{propDefInd}
Let $\ast$ be a uniformity preserving operation on $\mathcal{X}$, and let $\mathcal{A}$ be a set of subsets of $\mathcal{X}$. There exists a unique periodic partition $\langle\mathcal{A}\rangle$ which satisfies the following:
\begin{itemize}
\item $\mathcal{A}\preceq \langle\mathcal{A}\rangle$.
\item For every periodic partition $\mathcal{H}$ of $\mathcal{X}$, if $\mathcal{A}\preceq\mathcal{H}$ then $\langle\mathcal{A}\rangle\preceq\mathcal{H}$.
\end{itemize}
In other words, $\langle\mathcal{A}\rangle$ is the finest periodic partition that is coarser than $\mathcal{A}$. $\langle\mathcal{A}\rangle$ is called the \emph{periodic partition generated} by $\mathcal{A}$.
\end{myprop}
\begin{proof}
Define
\begin{equation}
\label{eqWedgelala}
\langle\mathcal{A}\rangle=\bigwedge_{\substack{\mathcal{H}\text{ is a periodic partition}\\\mathcal{A}\preceq\mathcal{H}}}\mathcal{H}.
\end{equation}
Proposition \ref{lemWedge} implies that $\langle\mathcal{A}\rangle$ is a periodic partition. Moreover, it follows from \eqref{eqWedgelala} and from the definition of the wedge operator (Definition \ref{defWedge}) that for every periodic partition $\mathcal{H}$ satisfying $\mathcal{A}\preceq\mathcal{H}$, we have $\langle\mathcal{A}\rangle\preceq\mathcal{H}$.

Now let $A\in\mathcal{A}$. We have:
\begin{itemize}
\item If $A=\o$, then $A\subset B$ for any $B\in\langle\mathcal{A}\rangle$.
\item If $A\neq\o$, then for every periodic partition $\mathcal{H}$ satisfying $\mathcal{A}\preceq\mathcal{H}$, choose $B_\mathcal{H}\in\mathcal{H}$ such that $A\subset B_\mathcal{H}$. Define
$$B=\bigcap_{\substack{\mathcal{H}\text{ is a periodic partition}\\\mathcal{A}\preceq\mathcal{H}}}B_\mathcal{H}.$$
Clearly, $A\subset B$ which implies that $B\neq \o$ and so $B\in\langle\mathcal{A}\rangle$ (see Definition \ref{defWedge}).
\end{itemize}
We conclude that for every $A\in\mathcal{A}$, there exists $B\in\langle\mathcal{A}\rangle$ such that $A\subset B$. Therefore, $\mathcal{A}\preceq\langle\mathcal{A}\rangle$.

Now let $\mathcal{H}'$ be a periodic partition satisfying the conditions of the proposition. I.e.,
\begin{itemize}
\item $\mathcal{A}\preceq \mathcal{H}'$.
\item For every periodic partition $\mathcal{H}$ of $\mathcal{X}$, if $\mathcal{A}\preceq\mathcal{H}$ then $\mathcal{H}'\preceq\mathcal{H}$.
\end{itemize}
Since $\mathcal{A}\preceq \langle\mathcal{A}\rangle$, we have $\mathcal{H}'\preceq \langle\mathcal{A}\rangle$. Similarly, since $\mathcal{A}\preceq \mathcal{H}'$ we have $\langle\mathcal{A}\rangle\preceq\mathcal{H}'$. Therefore, $\mathcal{H}'=\langle\mathcal{A}\rangle$ and so $\langle\mathcal{A}\rangle$ is unique.
\end{proof}

\begin{myrem}
It is possible to show that $\langle\mathcal{A}\rangle^{n\ast}=\langle\mathcal{A}^{n\ast}\rangle$ for every $n>0$, but we will not prove this here since we do not need this property for our purposes.
\end{myrem}

\begin{mycor}
\label{corInd}
Let $\ast$ be an ergodic operation on $\mathcal{X}$, and let $\mathcal{A}$ be a set of subsets of $\mathcal{X}$. There exists a unique stable partition $\langle\mathcal{A}\rangle$ which satisfies the following:
\begin{itemize}
\item $\mathcal{A}\preceq \langle\mathcal{A}\rangle$.
\item For every stable partition $\mathcal{H}$ of $\mathcal{X}$, if $\mathcal{A}\preceq\mathcal{H}$ then $\langle\mathcal{A}\rangle\preceq\mathcal{H}$.
\end{itemize}
In other words, $\langle\mathcal{A}\rangle$ is the finest stable partition that is coarser than $\mathcal{A}$. $\langle\mathcal{A}\rangle$ is called the \emph{stable partition generated} by $\mathcal{A}$.
\end{mycor}
\begin{proof}
The corollary follows from Proposition \ref{propDefInd} and from the fact that if $\ast$ is an ergodic operation on $\mathcal{X}$ then every periodic partition is stable (see Lemma \ref{stasta}).
\end{proof}

\begin{myrem}
The ergodicity condition in Corollary \ref{corInd} cannot be replaced by irreducibility. Consider the irreducible (but not ergodic) operation $\ast$ of Remark \ref{remIrrInd}, and let $\mathcal{A}=\big\{\{0,1\},\{2,3\}\big\}$. Notice that there is no stable partition that is both coarser than $\mathcal{A}$ and finer than every stable partition that is coarser than $\mathcal{A}$. Therefore, if $\ast$ is not ergodic, the concept of ``generated stable partitions" is not always well defined.
\end{myrem}

Let $\mathcal{A}$ be a set of subsets of $\mathcal{X}$ which covers $\mathcal{X}$ and does not contain the empty set as an element. We have $\mathcal{A}\preceq\langle\mathcal{A}\rangle$ which implies that $\mathcal{A}^{n\ast}\preceq\langle\mathcal{A}\rangle^{n\ast}$ for every $n>0$. Can we find $n>0$ for which $\mathcal{A}^{n\ast}=\langle\mathcal{A}\rangle^{n\ast}$? The rest of this section is dedicated to show that the answer to this question is affirmative if $\ast$ is strongly ergodic. This property of strongly ergodic operations turns out to be important for polarization theory as we will see in Part II of this paper \cite{RajErgII}.

\begin{mydef}
Let $\mathcal{A}$ be a set of subsets of $\mathcal{X}$. We say that $\mathcal{A}$ is an $\mathcal{X}$\emph{-cover} if $\o\notin\mathcal{A}$ and $\displaystyle \mathcal{X}=\bigcup_{A\in\mathcal{A}} A$.

We say that an $\mathcal{X}$-cover $\mathcal{A}$ is \emph{periodic} if $\mathcal{A}^{n\ast}=\mathcal{A}$ for some $n>0$. The least integer $n>0$ satisfying $\mathcal{A}^{n\ast}=\mathcal{A}$ is called the \emph{period} of $\mathcal{A}$, and it is denoted by $\per(\mathcal{A})$.

We say that an $\mathcal{X}$-cover $\mathcal{A}$ is \emph{balanced} if for every $A_1,A_2\in\mathcal{A}$ we have $|A_1|=|A_2|$. An $\mathcal{X}$-cover $\mathcal{A}$ is said to be \emph{stable} if it is both periodic and balanced.
\end{mydef}

\begin{myprop}
\label{stastaerg}
If $\ast$ is a strongly ergodic operation on a set $\mathcal{X}$, then every periodic $\mathcal{X}$-cover is a stable partition.
\end{myprop}
\begin{proof}
See Appendix \ref{appD}.
\end{proof}

\begin{myrem}
\label{remIndStErgNotlala}
The strong ergodicity condition in Proposition \ref{stastaerg} cannot be replaced by ergodicity. Consider the following ergodic (but not strongly ergodic) operation:
\begin{center}
  \begin{tabular}{ | c || c | c | c | c | c | c | }
    \hline
    $\ast$ & 0 & 1 & 2 & 3 & 4 & 5 \\ \hline \hline
    0 & 3 & 3 & 3 & 0 & 0 & 0 \\ \hline
    1 & 4 & 4 & 4 & 1 & 1 & 1 \\ \hline
    2 & 5 & 5 & 5 & 2 & 2 & 2 \\ \hline
    3 & 1 & 1 & 1 & 5 & 5 & 5 \\ \hline
    4 & 2 & 2 & 2 & 3 & 3 & 3 \\ \hline
    5 & 0 & 0 & 0 & 4 & 4 & 4 \\ \hline
  \end{tabular}
\end{center}
The set $\big\{\{0,1\},\{0,2\},\{1,2\},\{3,4\},\{3,5\},\{4,5\}\big\}$ is a periodic $\mathcal{X}$-cover of period 1, but it is not a partition.
\end{myrem}

\begin{mythe}
\label{theinduced}
Let $\ast$ be a strongly ergodic operation on a set $\mathcal{X}$. For every $\mathcal{X}$-cover $\mathcal{A}$, there exists an integer $n<2^{2^{|\mathcal{X}|}}$ such that $\mathcal{A}^{n\ast}=\langle\mathcal{A}\rangle$ and $\per(\langle\mathcal{A}\rangle)$ divides $n$, i.e., $\mathcal{A}^{n\ast}=\langle\mathcal{A}\rangle=\langle\mathcal{A}\rangle^{n\ast}$.
\end{mythe}
\begin{proof}
$2^{|\mathcal{X}|}$ is the number of subsets of $\mathcal{X}$, and $2^{2^{|\mathcal{X}|}}$ is the number of sets of subsets of $\mathcal{X}$. Thus, the sets $\mathcal{A}^{i\ast}$ for $0\leq i\leq 2^{2^{|\mathcal{X}|}}$ cannot be pairwise different. Therefore, there exist at least two integers $0\leq n_1<n_2\leq 2^{2^{|\mathcal{X}|}}$ such that $\mathcal{A}^{n_1\ast}=\mathcal{A}^{n_2\ast}$. Since $(\mathcal{A}^{n_1\ast})^{(n_2-n_1)\ast}=\mathcal{A}^{n_1\ast}$, $\mathcal{A}^{n_1\ast}$ is a stable partition by Proposition \ref{stastaerg}. There exists an integer $n_3$ such that $0\leq n_3<n_2-n_1$ and $n_3\equiv -n_1 \bmod (n_2-n_1)$. Let $n=n_1+n_3$, we have $n_1\leq n<n_2\leq 2^{2^{|\mathcal{X}|}}$, $\mathcal{A}^{n\ast}=(\mathcal{A}^{n_1\ast})^{n_3\ast}$ is a stable partition, and $n_2-n_1$ divides $n$. But since $(\mathcal{A}^{n_1\ast})^{(n_2-n_1)\ast}=\mathcal{A}^{n_2\ast}=\mathcal{A}^{n_1\ast}$, $\per(\mathcal{A}^{n_1\ast})$ divides $n_2-n_1$ which divides $n$. On the other hand, $\per(\mathcal{A}^{n_1\ast})=\per\big((\mathcal{A}^{n_1\ast}) ^{n_3\ast}\big)=\per(\mathcal{A}^{n\ast})$. We conclude that $\per(\mathcal{A}^{n\ast})$ divides $n$.

Now let $A\in\mathcal{A}$ and let $a$ be an arbitrary element of $\mathcal{X}$. Define the mapping $\pi:\mathcal{X}\rightarrow\mathcal{X}$ as $\pi(x)=x\ast a$. Since $\pi$ is a permutation, there exists $k>0$ such that $\pi^k(x)=x$ for every $x\in\mathcal{X}$. Now for every $0\leq i< kn$, let $X_i\in\mathcal{A}^{i\ast}$ be such that $a\in X_i$ and let $\mathfrak{X}=(X_i)_{0\leq i< kn}$. We have:
\begin{itemize}
\item $A\ast\mathfrak{X}\in\mathcal{A}^{kn\ast}$.
\item $A\subset A\ast\mathfrak{X}$ since $\pi^{kn}(x)=x$ for every $x\in\mathcal{X}$.
\item $\mathcal{A}^{kn\ast}=(\mathcal{A}^{n\ast})^{(k-1)n\ast}= \mathcal{A}^{n\ast}$ since $\mathcal{A}^{n\ast}$ is a stable partition whose period divides $n$ (and so divides $(k-1)n$ as well).
\end{itemize}
We conclude that $A\subset A\ast \mathfrak{X}\in\mathcal{A}^{n\ast}$. Therefore, $A\preceq\mathcal{A}^{n\ast}$. Now since $\mathcal{A}^{n\ast}$ is a stable partition (hence it is also periodic), we must have $\langle\mathcal{A}\rangle\preceq \mathcal{A}^{n\ast}$ by Proposition \ref{propDefInd}. On the other hand, we have:
\begin{itemize}
\item Since $\mathcal{A}\preceq\langle\mathcal{A}\rangle$ then $\mathcal{A}^{np\ast}\preceq\langle\mathcal{A}\rangle^{np\ast}$, where $p=\per(\langle\mathcal{A}\rangle)$.
\item $\mathcal{A}^{np\ast}=(\mathcal{A}^{n\ast})^{(p-1)n\ast}\stackrel{(a)}{=}\mathcal{A}^{n\ast}$, where (a) follows from the fact that $n$ divides $\per(\mathcal{A}^{n\ast})$.
\item $\langle\mathcal{A}\rangle^{np\ast}=\langle\mathcal{A}\rangle$ since $p=\per(\langle\mathcal{A}\rangle)$.
\end{itemize}
We conclude that $\mathcal{A}^{n\ast}\preceq\langle\mathcal{A}\rangle$, which implies that $\mathcal{A}^{n\ast}=\langle\mathcal{A}\rangle$ as we have already shown that $\langle\mathcal{A}\rangle\preceq\mathcal{A}^{n\ast}$.
\end{proof}

\begin{myrem}
The strong ergodicity condition in Theorem \ref{theinduced} cannot be replaced by ergodicity. Consider the ergodic operation $\ast$ of Remark \ref{remIndStErgNotlala}, and consider the the $\mathcal{X}$-cover $$\mathcal{A}=\big\{\{0,1\},\{0,2\},\{1,2\},\{3,4\},\{3,5\},\{4,5\}\big\},$$ which is not a partition. We have the following:
\begin{itemize}
\item It is easy to see that $\mathcal{A}^{n\ast}=\mathcal{A}$ for every $n\geq 0$.
\item Since $\mathcal{A}$ is not a partition, $\mathcal{A}^{n\ast}=\mathcal{A}$ is not a partition for every $n\geq0$.
\end{itemize}
Therefore, $\mathcal{A}^{n\ast}\neq\langle\mathcal{A}\rangle^{n\ast}$ for every $n\geq0$.
\end{myrem}

\section{Product of binary operations}

\label{sec8}

\begin{mydef}
Let $\mathcal{X}_1,\ldots,\mathcal{X}_m$ be $m$ sets, and let $\ast_1,\ldots,\ast_m$ be $m$ binary operations on $\mathcal{X}_1,\ldots,\mathcal{X}_m$ respectively. We define the product of $\ast_1,\ldots,\ast_m$, denoted $\ast=\ast_1\otimes\ldots\otimes\ast_m$, as the binary operation $\ast$ on $\mathcal{X}_1\times\ldots\times\mathcal{X}_m$ defined by:
$$(x_1,x_2,\ldots,x_m)\ast(x_1',x_2',\ldots,x_m')=(x_1\ast_1x_1',x_2\ast_2x_2',\ldots,x_m\ast_mx_m').$$
\end{mydef}

\begin{myprop}
\label{trivialprod}
Let $\ast_1,\ldots,\ast_m$ be $m$ binary operations on $\mathcal{X}_1,\ldots,\mathcal{X}_m$ respectively. Let $\mathcal{X}=\mathcal{X}_1\times\ldots\times\mathcal{X}_m$ and $\ast=\ast_1\otimes\ldots\otimes\ast_m$. We have:
\begin{enumerate}
\item $\ast$ is uniformity preserving if and only if $\ast_1,\ldots,\ast_m$ are uniformity preserving.
\item If $\ast$ is irreducible then $\ast_1,\ldots,\ast_m$ are irreducible. The converse is not necessarily true.
\item $\ast$ is ergodic if and only if $\ast_1,\ldots,\ast_m$ are ergodic. Moreover, $\con(\ast)=\max\{\con(\ast_1),\ldots,\con(\ast_m)\}$.
\end{enumerate}
\end{myprop}
\begin{proof}
1) Suppose that $\ast_1,\ldots,\ast_m$ are uniformity preserving. Fix $b=(b_1,\ldots,b_m)\in\mathcal{X}$ and define the mapping $\pi_b:\mathcal{X}\rightarrow\mathcal{X}$ as $\pi_b(x)=x\ast b$ for all $x\in\mathcal{X}$. Now let $y=(y_1,\ldots,y_m)\in\mathcal{X}$. For every $1\leq i\leq m$, $\ast_i$ is uniformity preserving and so there exists $x_i\in\mathcal{X}_i$ such that $x_i\ast_i b_i=y_i$. Define $x=(x_1,\ldots,x_m)$. We have $\pi_b(x)= x\ast b=y$. Therefore, $\pi_b$ is surjective which implies that it is bijective. Since this is true for every $b\in\mathcal{X}$, $\ast$ is uniformity preserving.

Conversely, suppose that $\ast$ is uniformity preserving and let $1\leq i\leq m$. Fix $b_i\in\mathcal{X}_i$ and define the mapping $\pi_{b_i}:\mathcal{X}_i\rightarrow\mathcal{X}_i$ as $\pi_{b_i}(x_i)=x_i\ast_i b_i$ for all $x_i\in\mathcal{X}_i$. Now let $y_i\in\mathcal{X}_i$ and choose arbitrarily $y_j\in\mathcal{X}_j$ for each $j\neq i$. Define $y=(y_1,\ldots,y_m)\in\mathcal{X}$. Since $\ast$ is uniformity preserving, there exists $x=(x_1,\ldots,x_m)\in\mathcal{X}$ such that $y=x\ast b$ which implies that $y_i=x_i\ast_i b_i$. Therefore, $\pi_{b_i}$ is surjective which implies that it is bijective. Since this is true for every $b_i\in\mathcal{X}_i$, $\ast_i$ is uniformity preserving.

\vspace*{2mm}
2) Suppose that $\ast$ is irreducible and fix $1\leq i\leq m$. Let $a_i,b_i\in\mathcal{X}_i$ and choose arbitrarily $a_j,b_j\in\mathcal{X}_j$ for each $j\neq i$. Define $a=(a_1,\ldots,a_m)\in\mathcal{X}$ and $b=(b_1,\ldots,b_m)\in\mathcal{X}$. Since $\ast$ is irreducible, $a$ is $\ast$-connectable to $b$ and so there exists $l>0$ and $x_0,\ldots,x_{l-1}\in\mathcal{X}$ such that $b=(\ldots((a\ast x_0)\ast x_1)\ldots\ast x_{l-1})$. For each $0\leq k<l$, let $x_k=(x_{1,k},\ldots,x_{m,k})$ and so $x_{i,k}\in\mathcal{X}_i$. It is easy to see that we have $b_i=(\ldots((a_i\ast_i x_{i,0})\ast_i x_{i,1})\ldots\ast_i x_{i,l-1})$. Therefore, $a_i$ is $\ast_i$-connectable to $b_i$ for all $a_i,b_i\in\mathcal{X}_i$, hence $\ast_i$ is irreducible.

In order to see that the converse is not necessarily true, let $\mathcal{X}_1=\mathcal{X}_2=\{0,1\}$ and define $x\ast_1 y=x\ast_2 y= x\oplus 1$ for every $x,y\in\{0,1\}$. It is easy to see that $\ast_1$ and $\ast_2$ are irreducible and $\per(\ast_1)=\per(\ast_2)=2$. Let $\ast=\ast_1\otimes\ast_2$. It is easy to see that $(0,0)$ is not $\ast$-connectable to $(0,1)$. Therefore, $\ast$ is not irreducible.

\vspace*{2mm}
3) Suppose that $\ast_1,\ldots,\ast_m$ are ergodic and let $d=\max\{\con(\ast_1),\ldots,\con(\ast_m)\}$. Let $a=(a_1,\ldots,a_m)\in\mathcal{X}$ and $b=(b_1,\ldots,b_m)\in\mathcal{X}$. For each $1\leq i\leq m$, since $d\geq \con(\ast_i)$ there exist $x_{i,0},\ldots,x_{i,d-1}\in\mathcal{X}_i$ such that $b_i=(\ldots((a_i\ast_i x_{i,0})\ast_i x_{i,1})\ldots\ast_i x_{i,d-1})$. For each $0\leq k< d$ define $x_k=(x_{1,k},\ldots,x_{m,k})\in\mathcal{X}$. It is easy to see that $b=(\ldots((a\ast x_0)\ast x_1)\ldots\ast x_{d-1})$. Therefore, all the elements of $\mathcal{X}$ are $\ast$-connectable to each other in $d$ steps. We conclude that $\ast$ is ergodic and $\con(\ast)\leq d=\max\{\con(\ast_1),\ldots,\con(\ast_m)\}$.

Conversely, suppose that $\ast$ is ergodic and let $1\leq i\leq m$. Let $a_i,b_i\in\mathcal{X}_i$ and choose arbitrarily $a_j,b_j\in\mathcal{X}_j$ for each $j\neq i$. Define $a=(a_1,\ldots,a_m)\in\mathcal{X}$ and $b=(b_1,\ldots,b_m)\in\mathcal{X}$. Since $\ast$ is ergodic, $a$ is $\ast$-connectable to $b$ in $\con(\ast)$ steps. It follows that $a_i$ is $\ast_i$-connectable to $b_i$ in $\con(\ast)$ steps (we use the same argument that we used for the irreducible case). Since this is true for every $a_i,b_i\in\mathcal{X}_i$, we conclude that $\ast_i$ is ergodic and $\con(\ast_i)\leq\con(\ast)$. We conclude that $\max\{\con(\ast_1),\ldots,\con(\ast_m)\}\leq\con(\ast)$ which implies that $\con(\ast)=\max\{\con(\ast_1),\ldots,\con(\ast_m)\}$ since we have  $\con(\ast)\leq\max\{\con(\ast_1),\ldots,\con(\ast_m)\}$ from the previous paragraph.
\end{proof}

\begin{mydef}
Let $\mathcal{H}_1,\ldots,\mathcal{H}_m$ be $m$ stable partitions of $(\mathcal{X}_1,\ast_1),\ldots,(\mathcal{X}_m,\ast_m)$ respectively. The product of $\mathcal{H}_1,\ldots,\mathcal{H}_m$, denoted $\mathcal{H}=\mathcal{H}_1\otimes\ldots\otimes\mathcal{H}_m$ is defined as
$$\mathcal{H}=\{A_1\times \ldots\times A_m:\; A_1\in\mathcal{H}_1,\ldots,A_m\in\mathcal{H}_m\}.$$
It is easy to see that $\mathcal{H}$ is a stable partition of $(\mathcal{X}_1\times\ldots\times\mathcal{X}_m,\ast_1\otimes\ldots\otimes\ast_m)$ of period $\per(\mathcal{H})=\lcm\{\per(\mathcal{H}_1),\ldots,\per(\mathcal{H}_m)\}$.
\end{mydef}

\begin{mythe}
Let $\ast_1$ and $\ast_2$ be two ergodic operations on $\mathcal{X}_1$ and $\mathcal{X}_2$ respectively. Let $\mathcal{X}=\mathcal{X}_1\times\mathcal{X}_2$ and $\ast=\ast_1\otimes\ast_2$ (thus, $\ast$ is ergodic). Let $\mathcal{H}$ be a stable partition of $\mathcal{X}$. There exist two unique stable partitions $\mathcal{L}_1:=\mathcal{L}_1(\mathcal{H})$ and $\mathcal{U}_1:=\mathcal{U}_1(\mathcal{H})$ of $\mathcal{X}_1$ and two unique stable partitions $\mathcal{L}_2:=\mathcal{L}_2(\mathcal{H})$ and $\mathcal{U}_2:=\mathcal{U}_2(\mathcal{H})$ of $\mathcal{X}_2$ such that:
\begin{itemize}
\item $\mathcal{L}_1\preceq\mathcal{U}_1$, $\mathcal{L}_2\preceq\mathcal{U}_2$ and $\frac{\|\mathcal{U}_1\|}{\|\mathcal{L}_1\|}=\frac{\|\mathcal{U}_2\|}{\|\mathcal{L}_2\|}=n$ for some integer $n>0$.
\item $\mathcal{L}_1\otimes \mathcal{L}_2\preceq \mathcal{H}\preceq\mathcal{U}_1\otimes \mathcal{U}_2$.
\item For every $H\in\mathcal{H}$, there exist $n$ disjoint sets $H_{1,1},\ldots,H_{1,n}\in \mathcal{L}_1$ and $n$ disjoint sets $H_{2,1},\ldots,H_{2,n}\in \mathcal{L}_2$ such that:
\begin{itemize}
\item $H_{1,1}\cup \ldots \cup H_{1,n}\in\mathcal{U}_1$.
\item $H_{2,1}\cup \ldots \cup H_{2,n}\in\mathcal{U}_2$.
\item $H=(H_{1,1}\times H_{2,1})\cup \ldots \cup (H_{1,n}\times H_{2,n})$.
\end{itemize}
Therefore, $\|\mathcal{H}\|=n\cdot\|\mathcal{L}_1\|\cdot\|\mathcal{L}_2\|=\|\mathcal{L}_1\|\cdot\|\mathcal{U}_2\|=\|\mathcal{U}_1\|\cdot\|\mathcal{L}_2\|$.
\end{itemize}
The integer $n$ is called the \emph{correlation} of $\mathcal{H}$ and is denoted by $\cor_{\ast_1,\ast_2}(\mathcal{H})$.

We also have $\mathcal{L}_1(\mathcal{H})^{i\ast_1}=\mathcal{L}_1(\mathcal{H}^{i\ast})$, $\mathcal{L}_2(\mathcal{H})^{i\ast_2}=\mathcal{L}_2(\mathcal{H}^{i\ast})$, $\mathcal{U}_1(\mathcal{H})^{i\ast_1}=\mathcal{U}_1(\mathcal{H}^{i\ast})$ and $\mathcal{U}_2(\mathcal{H})^{i\ast_2}=\mathcal{U}_2(\mathcal{H}^{i\ast})$ for every $i\geq 0$.
\label{theprod}
\end{mythe}
\begin{proof}
See Appendix \ref{appE}.
\end{proof}

\begin{myrem}
If $\mathcal{H}=\mathcal{H}_1\otimes\mathcal{H}_2$, then $\mathcal{L}_1(\mathcal{H})=\mathcal{U}_1(\mathcal{H})=\mathcal{H}_1$, $\mathcal{L}_2(\mathcal{H})=\mathcal{U}_2(\mathcal{H})=\mathcal{H}_2$ and $\cor_{\ast_1,\ast_2}(\mathcal{H})=1$.
\end{myrem}

\begin{myex}
The following figure shows an element $H$ of a stable partition $\mathcal{H}$ of correlation $n=\cor_{\ast_1,\ast_2}(\mathcal{H})=3$. $H$ is represented by the regions that are enclosed in thick lines.
\begin{figure}[h]
\centering
\begin{center}
\begin{tikzpicture}[xscale =2.5,yscale=1]
\draw[line width=.5] (0,0)--(0,3);
\draw[line width=.5] (1,0)--(1,3);
\draw[line width=.5] (2,0)--(2,3);
\draw[line width=.5] (3,0)--(3,3);

\draw[line width=.5] (0,0)--(3,0);
\draw[line width=.5] (0,1)--(3,1);
\draw[line width=.5] (0,2)--(3,2);
\draw[line width=.5] (0,3)--(3,3);

\draw [line width=2,pattern=north east lines, pattern color=yellow] (0,0) rectangle (1,1);
\draw [line width=2,pattern=north east lines, pattern color=yellow] (1,2) rectangle (2,3);
\draw [line width=2,pattern=north east lines, pattern color=yellow] (2,1) rectangle (3,2);

\node at (.5,.5) {$H_{1,1}\times H_{2,1}$};
\node at (1.5,2.5) {$H_{1,2}\times H_{2,2}$};
\node at (2.5,1.5) {$H_{1,3}\times H_{2,3}$};

\draw [decorate,decoration={brace,amplitude=8pt,mirror},xshift=0pt,yshift=0pt]
(.05,-.15) -- (.95,-.15);
\node at (.5,-.7) {\small $H_{1,1}\in\mathcal{L}_1(\mathcal{H})$};
\draw [decorate,decoration={brace,amplitude=8pt,mirror},xshift=0pt,yshift=0pt]
(1.05,-.15) -- (1.95,-.15);
\node at (1.5,-.7) {\small $H_{1,2}\in\mathcal{L}_1(\mathcal{H})$};
\draw [decorate,decoration={brace,amplitude=8pt,mirror},xshift=0pt,yshift=0pt]
(2.05,-.15) -- (2.95,-.15);
\node at (2.5,-.7) {\small $H_{1,3}\in\mathcal{L}_1(\mathcal{H})$};

\draw [decorate,decoration={brace,amplitude=10pt},xshift=0pt,yshift=0pt]
(.05,3.15) -- (2.95,3.15);
\node at (1.5,3.75) {\small $H_{1,1}\cup H_{1,2}\cup H_{1,3}\in\mathcal{U}_1(\mathcal{H})$};

\draw [decorate,decoration={brace,amplitude=5pt},xshift=0pt,yshift=0pt]
(-.05,.05) -- (-.05,.95);
\draw [decorate,decoration={brace,amplitude=5pt},xshift=0pt,yshift=0pt]
(-.05,1.05) -- (-.05,1.95);
\draw [decorate,decoration={brace,amplitude=5pt},xshift=0pt,yshift=0pt]
(-.05,2.05) -- (-.05,2.95);
\node at (-.57,.5) {\small $H_{2,1}\in\mathcal{L}_2(\mathcal{H})$};
\node at (-.57,1.5) {\small $H_{2,3}\in\mathcal{L}_2(\mathcal{H})$};
\node at (-.57,2.5) {\small $H_{2,2}\in\mathcal{L}_2(\mathcal{H})$};

\draw [decorate,decoration={brace,amplitude=10pt,mirror},xshift=0pt,yshift=0pt]
(3.05,0.1) -- (3.05,2.95);
\node at (4.09,1.5) {\small $H_{2,1}\cup H_{2,2}\cup H_{2,3}\in\mathcal{U}_2(\mathcal{H})$};
\end{tikzpicture}
\captionsetup{justification=centering}
\caption{$H=(H_{1,1}\times H_{2,1})\cup(H_{1,2}\times H_{2,2})\cup (H_{1,3}\times H_{2,3})\in\mathcal{H}$.}
\end{center}
\end{figure}
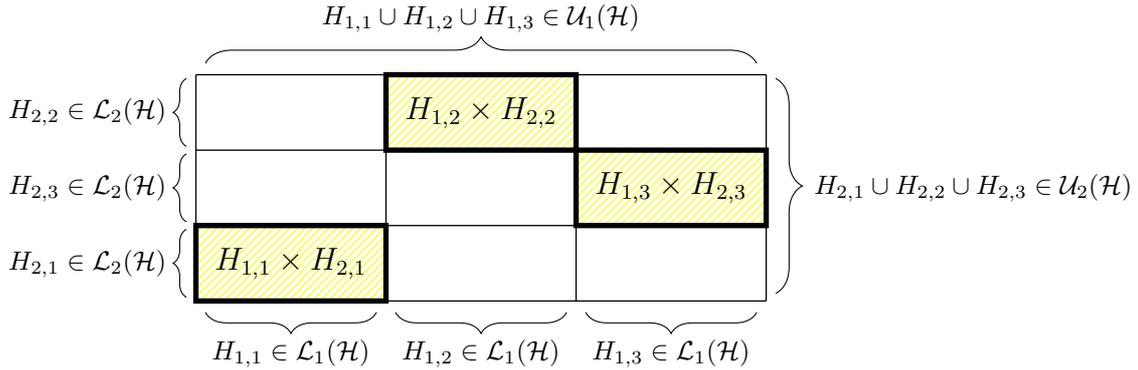
\end{myex}

\begin{myex}
Let $\mathcal{X}_1=\mathcal{X}_2=\{0,1\}$ and define $\ast_1$ and $\ast_2$ as $x\ast_1 y= x\ast_2 y= x\oplus y$ for every $x,y\in\{0,1\}$. Let $\mathcal{X}=\mathcal{X}_1\times\mathcal{X}_2$, $\ast=\ast_1\otimes\ast_2$ and $\mathcal{H}=\big\{\{(0,0),(1,1)\},\{(0,1),(1,0)\}\big\}$. It is easy to see that $\mathcal{H}$ is a stable partition of $(\mathcal{X},\ast)$. We have:
\begin{itemize}
\item $\mathcal{L}_1(\mathcal{H})=\mathcal{L}_2(\mathcal{H})=\big\{\{0\},\{1\}\big\}$.
\item $\mathcal{U}_1(\mathcal{H})=\mathcal{U}_2(\mathcal{H})=\big\{\{0,1\}\big\}$.
\item $n=\cor_{\ast_1,\ast_2}(\mathcal{H})=2$.
\end{itemize} 
For $H=\{(0,1),(1,0)\}\in\mathcal{H}$, we have:
\begin{itemize}
\item $H_{1,1}=\{0\}$, $H_{1,2}=\{1\}$ and $H_{1,1}\cup H_{1,2}=\{0,1\}\in\mathcal{U}_1(\mathcal{H})$.
\item $H_{2,1}=\{1\}$, $H_{2,2}=\{0\}$ and $H_{2,1}\cup H_{2,2}=\{0,1\}\in\mathcal{U}_2(\mathcal{H})$.
\item $(H_{1,1}\times H_{2,1})\cup(H_{1,2}\times H_{2,2})=\{(0,1),(1,0)\}=H$.
\end{itemize} 
\end{myex}


\vspace*{3mm}

Theorem \ref{theprod} shows that the stable partitions of the product of two ergodic operations have a very particular structure. This structure will be useful to prove the following theorem:

\begin{mythe}
\label{theprodstrong}
Let $\ast_1,\ldots,\ast_m$ be $m\geq 2$ binary operations on $\mathcal{X}_1,\ldots,\mathcal{X}_m$ respectively. Let $\mathcal{X}=\mathcal{X}_1\times\ldots\times\mathcal{X}_m$ and $\ast=\ast_1\otimes\ldots\otimes\ast_m$. Then $\ast$ is strongly ergodic if and only if $\ast_1,\ldots,\ast_m$ are strongly ergodic.
\end{mythe}
\begin{proof}
See Appendix \ref{appE}.
\end{proof}

\begin{mynot}
\label{notIdent}
Let $\ast_1,\ldots,\ast_m$ be $m\geq 2$ ergodic operations on $\mathcal{X}_1,\ldots,\mathcal{X}_m$ respectively. Define $\mathcal{X}=\mathcal{X}_1\times\ldots\times\mathcal{X}_m$ and $\ast=\ast_1\otimes\ldots\otimes\ast_m$.  Let $A=\{i_1,\ldots,i_{|A|}\}$ and $B=\{j_1,\ldots,j_{|B|}\}$ be two non-empty subsets of $I_m:=\{1,\ldots,m\}$ which form a non-trivial partition (i.e., $A\cup B=I_m$, $A\cap B=\o$, $A\neq\o$ and $B\neq\o$). Define $\mathcal{X}_A=\mathcal{X}_{i_1}\times\ldots\times\mathcal{X}_{i_{|A|}}$, $\mathcal{X}_B=\mathcal{X}_{j_1}\times\ldots\times\mathcal{X}_{j_{|B|}}$, $\ast_A=\ast_{i_1}\otimes\ldots\otimes\ast_{i_{|A|}}$ and $\ast_B=\ast_{j_1}\otimes\ldots\otimes\ast_{j_{|B|}}$. Define the mapping $f_{A,B}:\mathcal{X}\rightarrow \mathcal{X}_A\times\mathcal{X}_B$ as $f_{A,B}(x_1,\ldots,x_m)=\big((x_{i_1},\ldots,x_{i_{|A|}}),(x_{j_1},\ldots,x_{j_{|B|}})\big)$. Clearly, $f_{A,B}$ is a bijection. We call $f_{A,B}$ the \emph{canonical bijection} between $\mathcal{X}$ and $\mathcal{X}_A\times\mathcal{X}_B$. Throughout this paper, we identify $(\mathcal{X},\ast)$ with $(\mathcal{X}_A\ast \mathcal{X}_B,\ast_A\otimes\ast_B)$ through the canonical bijection $f_{A,B}$.

Let $\mathcal{H}$ be a stable partition of $(\mathcal{X},\ast)$. Since $\ast_A$ and $\ast_B$ are ergodic, there are two unique stable partitions $\mathcal{L}_A(\mathcal{H})\preceq\mathcal{U}_A(\mathcal{H})$ of $(\mathcal{X}_A,\ast_A)$ and two unique stable partitions $\mathcal{L}_B(\mathcal{H})\preceq\mathcal{U}_B(\mathcal{H})$ of $(\mathcal{X}_B,\ast_B)$ and $n_A=\cor_{\ast_A,\ast_B}(\mathcal{H})=\cor_{\ast_B,\ast_A}(\mathcal{H})=n_B>0$ satisfying the conditions of Theorem \ref{theprod}. We adopt the convention that $\mathcal{U}_{I_m}(\mathcal{H})=\mathcal{H}$.

If $A=\{i\}$ contains only one element $i$, we denote $\mathcal{L}_{\{i\}}(\mathcal{H})$ and $\mathcal{U}_{\{i\}}(\mathcal{H})$ by $\mathcal{L}_i(\mathcal{H})$ and $\mathcal{U}_i(\mathcal{H})$ respectively.
\end{mynot}

\begin{mynot}
For each $A\subset B\subset I_m=\{1,\ldots,m\}$ we define the mapping $P_{B\rightarrow A}:\mathcal{X}_B\rightarrow\mathcal{X}_A$ as $P_{B\rightarrow A}(x_{j_1},\ldots,x_{j_{|B|}})=(x_{i_1},\ldots,x_{i_{|A|}})$, where $A=\{i_1,\ldots,i_{|A|}\}\subset\{j_1,\ldots,j_{|B|}\}=B$. If $A$ contains only one element $i$, we denote $P_{B\rightarrow \{i\}}$ by $P_{B\rightarrow i}$.

Now for each $A\subsetneq B\subset I_m=\{1,\ldots,m\}$, each $x_{B\setminus A}\in\mathcal{X}_{B\setminus A}$ and each $X_B\subset \mathcal{X}_B$, we define the set $P_{B\rightarrow A | x_{B\setminus A}}(X_B):=\{x_A\in\mathcal{X}_A:\; (x_A,x_{B\setminus A})\in X_B\}\subset \mathcal{X}_A$. If $A$ contains only one element $i$, we denote $P_{B\rightarrow \{i\}|x_{B\setminus \{i\}}}(X_B)$ by $P_{B\rightarrow i|x_{B\setminus i}}(X_B)$.

It is easy to see that if $A\subset B \subset C\subset \{1,\ldots,m\}$ then we have $P_{B\rightarrow A}\circ P_{C\rightarrow B}=P_{C\rightarrow A}$. Similarly, if $A\subsetneq B\subsetneq C\subset \{1,\ldots,m\}$, then for each $X_C\in\mathcal{X}_C$, each $x_{C\setminus B}\in\mathcal{X}_{C\setminus B}$ and each $x_{B\setminus A}\in\mathcal{X}_{B\setminus A}$, we have $P_{B\rightarrow A| x_{B\setminus A}}\big(P_{C\rightarrow B| x_{C\setminus B}}(X_C)\big) = P_{C\rightarrow A| (x_{C\setminus B},x_{B\setminus A})}(X_C)$. Here we have $(x_{C\setminus B},x_{B\setminus A})\in \mathcal{X}_{C\setminus A}$ since we are identifying $\mathcal{X}_{C\setminus A}$ with $\mathcal{X}_{C\setminus B}\times \mathcal{X}_{B\setminus A}$ through the canonical bijection.
\end{mynot}

\begin{myrem}
\label{remremudsjbd}
Let $\mathcal{H}$ be a stable partition of $(\mathcal{X},\ast)=(\mathcal{X}_1\times\ldots\times\mathcal{X}_m,\ast_1\otimes\ldots\otimes\ast_m)$, where $\ast$ is ergodic. If $A\subset I_m=\{1,\ldots,m\}$, we have from Definition \ref{defProjGG}:
$$\mathcal{U}_A(\mathcal{H})=\{P_{I_m\rightarrow A}(H):\; H\in\mathcal{H}\}. $$
and if $A\subsetneq I_m=\{1,\ldots,m\}$, we have from Definition \ref{defProjEE}:
\begin{align*}
\mathcal{L}_A(\mathcal{H})&=\{P_{I_m\rightarrow A|x_{I_m\setminus A}}(H):\; H\in\mathcal{H},\; x_{I_m\setminus A}\in \mathcal{X}_{I_m\setminus A},\; P_{I_m\rightarrow A|x_{I_m\setminus A}}(H)\neq\o\}\\
&\stackrel{(a)}{=}\{P_{I_m\rightarrow A|x_{I_m\setminus A}}(H):\; H\in\mathcal{H},\; x_{I_m\setminus A}\in P_{I_m\rightarrow I_m\setminus A}(H)\}.
\end{align*}
(a) follows from the fact that $P_{I_m\rightarrow A|x_{I_m\setminus A}}(H)\neq\o$ if and only if $x_{I_m\setminus A}\in P_{I_m\rightarrow I_m\setminus A}(H)$.
\end{myrem}

\begin{myprop}
Let $\ast_1,\ldots,\ast_m$ be $m\geq 2$ ergodic operations on $\mathcal{X}_1,\ldots,\mathcal{X}_m$ respectively. Define $\mathcal{X}=\mathcal{X}_1\times\ldots\times\mathcal{X}_m$ and $\ast=\ast_1\otimes\ldots\otimes\ast_m$. Let $\mathcal{H}$ be a stable partition of $(\mathcal{X},\ast)$ and $A\subsetneq B\subsetneq I_m=\{1,\ldots,m\}$. Then $\mathcal{L}_A\big(\mathcal{L}_B(\mathcal{H})\big)= \mathcal{L}_A(\mathcal{H})$ and $\mathcal{U}_A\big(\mathcal{U}_B(\mathcal{H})\big) =\mathcal{U}_A(\mathcal{H})$.
\label{proplamlemlak}
\end{myprop}
\begin{proof}
From Remark \ref{remremudsjbd} we have:
\begin{align*}
\mathcal{U}_A\big(\mathcal{U}_B(\mathcal{H})\big)&=\{P_{B\rightarrow A}(H_B):\; H_B\in\mathcal{U}_B(\mathcal{H})\}=\big\{P_{B\rightarrow A}\big(P_{I_m\rightarrow B}(H)\big):\; H\in\mathcal{H}\big\}\\
&=\{P_{I_m\rightarrow A}(H):\; H\in\mathcal{H}\}=\mathcal{U}_A(\mathcal{H}).
\end{align*}

On the other hand, we have:
\begin{align*}
\mathcal{L}_A\big(\mathcal{L}_B(\mathcal{H})\big)&\stackrel{(a)}{=}\{P_{B\rightarrow A|x_{B\setminus A}}(H_B):\; H_B\in\mathcal{L}_B(\mathcal{H}),\; x_{B\setminus A}\in \mathcal{X}_{B\setminus A},\; P_{B\rightarrow A|x_{B\setminus A}}(H_B)\neq o\}\\
&\begin{aligned}
\stackrel{(b)}{=}\Big\{P_{B\rightarrow A|x_{B\setminus A}}\big(P_{I_m\rightarrow B|x_{I_m\setminus B}}(H)\big):\; H\in &\mathcal{H},\; x_{I_m\setminus B}\in\mathcal{X}_{I_m\setminus B},\; P_{I_m\rightarrow B|x_{I_m\setminus B}}(H)\neq\o,\\
&x_{B\setminus A}\in \mathcal{X}_{B\setminus A},\; P_{B\rightarrow A|x_{B\setminus A}}\big(P_{I_m\rightarrow B|x_{I_m\setminus B}}(H)\big)\neq o\Big\}
\end{aligned}
\\
&
\begin{aligned}
\stackrel{(c)}{=}\Big\{P_{B\rightarrow A|x_{B\setminus A}}\big(P_{I_m\rightarrow B|x_{I_m\setminus B}}(H)\big):\; H\in\mathcal{H},\; x_{I_m\setminus B}\in &\mathcal{X}_{I_m\setminus B},\; x_{B\setminus A} \in \mathcal{X}_{B\setminus A},\\
&\;\;\;\;P_{B\rightarrow A|x_{B\setminus A}}\big(P_{I_m\rightarrow B|x_{I_m\setminus B}}(H)\big)\neq o\Big\}
\end{aligned}
\\
&
\begin{aligned}
\stackrel{(d)}{=}\Big\{P_{I_m\rightarrow A|(x_{I_m\setminus B},x_{B\setminus A})}(H):\; H\in\mathcal{H},\; x_{I_m\setminus B}\in\mathcal{X}_{I_m\setminus B},\; x_{B\setminus A}\in &\mathcal{X}_{B\setminus A},\\
& P_{I_m\rightarrow A|(x_{I_m\setminus B},x_{B\setminus A})}(H) \neq \o\Big\}
\end{aligned}
\\
&\stackrel{(e)}{=}\{P_{I_m\rightarrow A|x_{I_m\setminus A}}(H):\; H\in\mathcal{H},\; x_{I_m\setminus A}\in\mathcal{X}_{I_m\setminus A},\; P_{I_m\rightarrow A|x_{I_m\setminus A}}(H)\neq o\}=\mathcal{L}_A(\mathcal{H}).
\end{align*}
(a) and (b) follow from Remark \ref{remremudsjbd}. (c) follows from the fact that $P_{B\rightarrow A|x_{B\setminus A}}\big(P_{I_m\rightarrow B|x_{I_m\setminus B}}(H)\big)\neq \o$ entails $P_{I_m\rightarrow B|x_{I_m\setminus B}}(H)\neq\o$. (d) follows from the fact that $P_{B\rightarrow A|x_{B\setminus A}}\big(P_{I_m\rightarrow B|x_{I_m\setminus B}}(H)\big)=P_{I_m\rightarrow A|(x_{I_m\setminus B},x_{B\setminus A})}(H)$. (e) follows from the fact that $I_m\setminus A=(I_m\setminus B)\cup(B\setminus A)$ and so $\mathcal{X}_{I_m\setminus A}$ is identified to $\mathcal{X}_{I_m\setminus B}\times \mathcal{X}_{B\setminus A}$.
\end{proof}

\begin{mydef}
\label{defCanFac}
Let $\ast_1,\ldots,\ast_m$ be $m\geq 2$ ergodic operations on $\mathcal{X}_1,\ldots,\mathcal{X}_m$ respectively. Let $\mathcal{X}=\mathcal{X}_1\times\ldots\times\mathcal{X}_m$ and $\ast=\ast_1\otimes\ldots\otimes\ast_m$. Let $\mathcal{H}$ be a stable partition of $(\mathcal{X},\ast)$. The \emph{canonical factorization} of $\mathcal{H}$ is the sequence $(\mathcal{H}_i)_{1\leq i\leq m}$ defined as:
\begin{itemize}
\item $\mathcal{H}_m=\mathcal{U}_m(\mathcal{H})$.
\item For each $1\leq i< m$, $\mathcal{H}_i=\mathcal{U}_i\big(\mathcal{L}_{I_i}(\mathcal{H})\big)$, where $I_i=\{1,\ldots,i\}$.
\end{itemize}
\end{mydef}

\begin{mylem}
Let $\ast_1,\ldots,\ast_m$ be $m\geq 2$ ergodic operations on $\mathcal{X}_1,\ldots,\mathcal{X}_m$ respectively. Let $\mathcal{X}=\mathcal{X}_1\times\ldots\times\mathcal{X}_m$ and $\ast=\ast_1\otimes\ldots\otimes\ast_m$. Let $\mathcal{H}$ be a stable partition of $(\mathcal{X},\ast)$. If $(\mathcal{H}_i)_{1\leq i\leq m}$ is the canonical factorization of $\mathcal{H}$, then $(\mathcal{H}_i)_{1\leq i\leq m-1}$ is the canonical factorization of $\mathcal{L}_{I_{m-1}}(\mathcal{H})$, where $I_{m-1}=\{1,\ldots,m-1\}$.
\label{lembdhsbhsssdfgaaaq1}
\end{mylem}
\begin{proof}
For each $1\leq i\leq m$, define $I_i=\{1,\ldots,i\}$. Let $\{\mathcal{H}_i'\} _{1\leq i\leq m-1}$ be the canonical factorization of $\mathcal{L}_{I_{m-1}}(\mathcal{H})$. We have:
\begin{itemize}
\item $\mathcal{H}_{m-1}'=\mathcal{U}_{m-1}\big( \mathcal{L}_{I_{m-1}}(\mathcal{H})\big)=\mathcal{H}_{m-1}$.
\item For each $1\leq i<m-1$, we have $\mathcal{H}_i'=\mathcal{U}_i\big(\mathcal{L}_{I_i}\big( \mathcal{L}_{I_{m-1}}(\mathcal{H})\big)\big)\stackrel{(a)}{=}\mathcal{U}_i\big(\mathcal{L}_{I_i}(\mathcal{H})\big)=\mathcal{H}_i$, where (a) follows from Proposition \ref{proplamlemlak}.
\end{itemize}
\end{proof}

\begin{mydef}
\label{defSect}
Let $\mathcal{H}$ be a partition of a set $\mathcal{X}$. A section of $\mathcal{H}$ is a subset $C\subset\mathcal{X}$ such that:
\begin{itemize}
\item $|C|=\mathcal{H}$.
\item For each $H\in\mathcal{H}$, there exists a unique $x\in C$ such that $x\in H$. In other words, the mapping $\proj_{\mathcal{H}}$, restricted to $C$, is a bijection between $C$ and $\mathcal{H}$.
\end{itemize}
\end{mydef}

\begin{mylem}
Let $\ast_1$ and $\ast_2$ be two ergodic operations on $\mathcal{X}_1$ and $\mathcal{X}_2$ respectively. Let $\mathcal{X}=\mathcal{X}_1\times\mathcal{X}_2$ and $\ast=\ast_1\otimes\ast_2$ (thus, $\ast$ is ergodic). Let $\mathcal{H}$ be a stable partition of $\mathcal{X}$. If $C_1$ and $C_2$ are sections of $\mathcal{L}_1(\mathcal{H})$ and $\mathcal{U}_2(\mathcal{H})$ respectively, then $C=C_1\times C_2$ is a section of $\mathcal{H}$.
\label{lembdhsbhsssdfgaaaq2}
\end{mylem}
\begin{proof}
Let $f_{C,\mathcal{H}}:C\rightarrow\mathcal{H}$ be the mapping $\proj_{\mathcal{H}}$ restricted to $C$, i.e., $f_{C,\mathcal{H}}(x)=\proj_{\mathcal{H}}(x)$ for every $x\in C$.

Let $H\in\mathcal{H}$ and $I_2=\{1,2\}$. We have $P_{I_2\rightarrow 2}(H)\in \mathcal{U}_2(\mathcal{H})$ by Remark \ref{remremudsjbd}. Now since $C_2$ is a section of $\mathcal{U}_2(\mathcal{H})$, there exists a unique $x_2\in C_2$ such that $x_2\in P_{I_2\rightarrow 2}(H)$.

Since $x_2\in P_{I_2\rightarrow 2}(H)$, we have $P_{I_2\rightarrow 1|x_2}(H)\in\mathcal{L}_1(\mathcal{H})$ by Remark \ref{remremudsjbd}. But $C_1$ is a section of $\mathcal{L}_1(\mathcal{H})$, so there exists a unique $x_1\in C_1$ such that $x_1\in P_{I_2\rightarrow 1|x_2}(H)$, which means that $(x_1,x_2)\in H$. Therefore, there exists $(x_1,x_2)\in C_1\times C_2=C$ such that $f_{C,\mathcal{H}}(x_1,x_2)=\proj_{\mathcal{H}}(x_1,x_2)=H$. We conclude that $f_{C,\mathcal{H}}$ is surjective.

On the other hand, we have $|C|=|C_1\times C_2|=|C_1|\cdot|C_2|=|\mathcal{L}_1(\mathcal{H})|\cdot|\mathcal{U}_2(\mathcal{H})|=|\mathcal{H}|$, where the last equality follows from Theorem \ref{theprod}. Therefore, $f_{C,\mathcal{H}}$ is bijective since $f_{C,\mathcal{H}}:C\rightarrow\mathcal{H}$ is surjective and $|C|=|\mathcal{H}|$. Hence, $C=C_1\times C_2$ is a section of $\mathcal{H}$.
\end{proof}

\begin{myprop}
\label{PropProdProdProd}
Let $\ast_1,\ldots,\ast_m$ be $m\geq 2$ ergodic operations on $\mathcal{X}_1,\ldots,\mathcal{X}_m$ respectively. Let $\mathcal{X}=\mathcal{X}_1\times\ldots\times\mathcal{X}_m$ and $\ast=\ast_1\otimes\ldots\otimes\ast_m$. Let $\mathcal{H}$ be a stable partition of $(\mathcal{X},\ast)$ and $(\mathcal{H}_i)_{1\leq i\leq m}$ be the canonical factorization of $\mathcal{H}$. We have:
\begin{itemize}
\item $|\mathcal{H}|=|\mathcal{H}_1|\cdots|\mathcal{H}_m|$.
\item If $C_i$ is a section of $\mathcal{H}_i$ for every $1\leq i\leq m$, then $C=C_1\times\ldots\times C_m$ is a section of $\mathcal{H}$.
\end{itemize}
\end{myprop}
\begin{proof}
For each $1\leq i\leq m$, we define $I_i=\{1,\ldots,i\}$. We will prove the proposition by induction on $m$. If $m=2$, we have:
\begin{itemize}
\item $\mathcal{H}_1=\mathcal{U}_1\big(\mathcal{L}_{I_1}(\mathcal{H})\big)=\mathcal{U}_1\big(\mathcal{L}_1(\mathcal{H})\big)= \mathcal{L}_1(\mathcal{H})$ and $\mathcal{H}_2=\mathcal{U}_2(\mathcal{H})$.
\item By Theorem \ref{theprod}, we have $|\mathcal{H}|=|\mathcal{L}_1(\mathcal{H})|\cdot|\mathcal{U}_2(\mathcal{H})|=|\mathcal{H}_1|\cdot|\mathcal{H}_2|$.
\item If $C_1$ and $C_2$ are sections of $\mathcal{H}_1=\mathcal{L}_1(\mathcal{H})$ and $\mathcal{H}_2=\mathcal{U}_2(\mathcal{H})$ respectively, then Lemma \ref{lembdhsbhsssdfgaaaq2} shows that $C=C_1\times C_2$ is a section of $\mathcal{H}$
\end{itemize}
Therefore the proposition is true for $m=2$.

Now let $m>2$ and suppose that the proposition is true for $m-1$. By Lemma \ref{lembdhsbhsssdfgaaaq1}, $(\mathcal{H}_i)_{1\leq i\leq m-1}$ is the canonical factorization of $\mathcal{L}_{I_{m-1}}(\mathcal{H})$. We have:
\begin{itemize}
\item $|\mathcal{H}|=|\mathcal{L}_{I_{m-1}}(\mathcal{H})|\cdot|\mathcal{U}_m(\mathcal{H})|=|\mathcal{L}_{I_{m-1}}(\mathcal{H})|\cdot|\mathcal{H}_m|$ by Theorem \ref{theprod}. On the other hand, we have $|\mathcal{L}_{I_{m-1}}(\mathcal{H})|=|\mathcal{H}_1|\cdots|\mathcal{H}_{m-1}|$ from the induction hypothesis. Therefore, $|\mathcal{H}|=|\mathcal{H}_1|\cdots|\mathcal{H}_m|$.
\item For every $1\leq i\leq m$, let $C_i$ be a section of $\mathcal{H}_i$. From the induction hypothesis we get that $C_1\times\ldots\times C_{m-1}$ is a section of $\mathcal{L}_{I_{m-1}}(\mathcal{H})$. Now since $C_1\times\ldots\times C_{m-1}$ and $C_m$ are sections of $\mathcal{L}_{I_{m-1}}(\mathcal{H})$ and $\mathcal{U}_m(\mathcal{H})$ respectively, Lemma \ref{lembdhsbhsssdfgaaaq2} implies that $C=C_1\times\ldots\times C_m$ is a section of $\mathcal{H}$.
\end{itemize}
Therefore, the proposition is also true for $m$. We conclude that the proposition is true for every $m\geq 2$.
\end{proof}

\section{Conclusion}
An ergodic theory for binary operations was developed. This theory will be applied in Part II of this paper \cite{RajErgII} to provide a foundation for polarization theory. We will show that a binary operation $\ast$ is polarizing if and only if it is uniformity preserving and $/^{\ast}$ is strongly ergodic. A natural question to ask is whether the strong ergodicity of the right-inverse operation implies the strong ergodicity of the operation itself. It is easy to see that a uniformity preserving operation is ergodic (resp. irreducible, quasigroup operation) if and only if its right-inverse is ergodic (resp. irreducible, quasigroup operation). We do not know whether the same is true for strong ergodicity.

The potential applications of the ergodic theory of binary operations might extend beyond polarization theory. The mathematical framework that is developed here is fairly general and might be useful to areas outside polarization and information theory.

\appendices

\section{Proof of Proposition \ref{lemerg}}
\label{appA}
1) Trivial: For a quasigroup operation, all the elements of $\mathcal{X}$ are $\ast$-connectable to each other in one step.

\vspace*{2mm}
2) Suppose that $\ast$ is uniformity preserving but not irreducible, there exists two elements $a_1$ and $a_2$ of $\mathcal{X}$ such that $a_1$ is not $\ast$-connectable to $a_2$. Let $A_1=\{x\in\mathcal{X}:\; a_1\stackrel{\ast}{\longrightarrow} x\}$ and $A_2=\mathcal{X}\setminus A_1$. Clearly, $a_1\ast a_1\in A_1$ and $a_2\in A_2$. Therefore, $A_1$ and $A_2$ are two disjoint non-empty sets such that $A_1\cup A_2=\mathcal{X}$. Moreover, we have $A_1\ast \mathcal{X} \subset A_1$ from the definition of $A_1$. Now since $|A_1\ast\mathcal{X}|\geq |A_1|$, we must have $A_1\ast\mathcal{X}=A_1$. 

For every $x\in\mathcal{X}$, define $\pi_x:\mathcal{X}\rightarrow\mathcal{X}$ as $\pi_x(a)=a\ast x$ for all $a\in\mathcal{X}$. Since $\ast$ is uniformity preserving, $\pi_x$ is bijective for all $x\in\mathcal{X}$. Therefore, $|\pi_x(A_1)|=|A_1|$. On the other hand, $\pi_x(A_1)=A_1\ast x\subset A_1\ast\mathcal{X}=A_1$. This means that $\pi_x(A_1)=A_1$, which implies that $\pi_{x}(A_2)=\pi_{x}(\mathcal{X}\setminus A_1)=\mathcal{X}\setminus A_1=A_2$ since $\pi_x$ is bijective. Therefore, $A_2\ast x=A_2$ for every $x\in\mathcal{X}$, hence $A_2\ast\mathcal{X}=A_2$.

\vspace*{2mm}
3) Suppose that $\ast$ is irreducible, and let $a,b\in\mathcal{X}$. Since $\ast$ is irreducible, there exists $l_1,l_2\geq 0$ such that $a\stackrel{\ast,l_1}{\longrightarrow} b$ and $b\stackrel{\ast,l_2}{\longrightarrow} a$, so $a\stackrel{\ast,l_1+l_2}{\longrightarrow} a$ which means that $\per(\ast,a)$ divides $l_1+l_2$. Now for any integer $l>0$ satisfying $b\stackrel{\ast,l}{\longrightarrow} b$, we have that $a\stackrel{\ast,l_1+l+l_2}{\longrightarrow} a$. This shows that $\per(\ast,a)$ divides $l_1+l_2+l$, which implies that $\per(\ast,a)$ divides $l$ since we have just shown that $\per(\ast,a)$ divides $l_1+l_2$. But this is true for any $l>0$ satisfying $b\stackrel{\ast,l}{\longrightarrow} b$. We conclude that $\per(\ast,a)$ divides $\per(\ast,b)$. Similarly, we can show that $\per(\ast,b)$ divides $\per(\ast,a)$. Therefore, $\per(\ast,a)$ is the same for all $a\in\mathcal{X}$. Now since $\per(\ast)=\gcd\{\per(\ast,a):\;a\in\mathcal{X}\}$, we have $\per(\ast)=\per(\ast,a)$ for all $a\in\mathcal{X}$.

\vspace*{2mm}
4) Suppose that $\ast$ is irreducible and let $n=\per(\ast)$. Fix $a\in\mathcal{X}$ and define the subsets $H_0$, \ldots, $H_{n-1}$ of $\mathcal{X}$ as follows: for each $0\leq i< n$, $H_i=\big\{x\in\mathcal{X}:\; \exists l> 0,\; a\stackrel{\ast,l}{\longrightarrow} x\;\text{and}\;l\equiv i\bmod n\big\}$. We have the following:
\begin{itemize}
\item If $x\in\mathcal{X}$, then $a\stackrel{\ast,l_{a,x}}{\longrightarrow} x$ for some integer $l_{a,x}>0$ because of irreducibility. This shows that for every $x\in\mathcal{X}$, we have $\displaystyle x\in H_{l_{a,x}\bmod n}\subset \bigcup_{i=0}^{n-1} H_i$. Therefore, $\displaystyle \mathcal{X}\subset \bigcup_{i=0}^{n-1} H_i\subset\mathcal{X}$, hence $\displaystyle\bigcup_{i=0}^{n-1} H_i=\mathcal{X}$.
\item Let $x\in H_i$ and $y\in H_j$. We have $a\stackrel{\ast,l_{a,x}}{\longrightarrow} x$ for some $l_{a,x}>0$ satisfying $l_{a,x}\equiv i\bmod n$. Moreover, $x\stackrel{\ast,l_{x,a}}{\longrightarrow} a$ for some $l_{x,a}>0$, and so $a\stackrel{\ast,l_{a,x}+l_{x,a}}{\longrightarrow} a$. The definition of $\per(\ast)$ implies that $n$ divides $l_{a,x}+l_{x,a}$ and so $l_{x,a}\equiv -i \bmod n$. Now since $y\in H_j$, we have $a\stackrel{\ast,l_{a,y}}{\longrightarrow} y$ for some $l_{a,y}>0$ satisfying $l_{a,y}\equiv j\bmod n$. We conclude that $x\stackrel{\ast,l_{x,y}}{\longrightarrow} y$ where $l_{x,y}=l_{x,a}+l_{a,y}\equiv j-i\bmod n$.
\item Suppose there exists $i\neq j$ such that $H_i\cap H_j\neq \o$ and let $x\in H_i\cap H_j$, then $x\stackrel{\ast,l_{x,x}}{\longrightarrow} x$, where $l_{x,x}\equiv j-i \not\equiv 0 \bmod n$. The definition of $\per(\ast)$ implies that $n$ divides $l_{x,x}$ which is a contradiction since $l_{x,x}\not\equiv 0\bmod n$. We conclude that $H_i\cap H_j=\o$ for all $i\neq j$.
\item For any $0\leq i< n$ and any $y\in H_i\ast\mathcal{X}$, there exist $x\in H_i$ and $z\in\mathcal{X}$ such that $y=x\ast z$, which implies that $y\in H_{i+1\bmod n}$. Therefore $H_i\ast\mathcal{X}\subset H_{i+1\bmod n}$, and so $|H_{i+1\bmod n}|\geq |H_i\ast\mathcal{X}|\geq |H_i|$. Thus, $|H_0|\geq |H_{n-1}|\geq \ldots |H_1|\geq |H_0|$, which implies that $|H_0|=|H_1|=\ldots=|H_{n-1}|$.
\end{itemize}
Therefore, $\{H_0,\ldots,H_{n-1}\}$ is a partition of $\mathcal{X}$ satisfying $|H_0|=|H_1|=\ldots=|H_{n-1}|$.

Now let $0\leq i<n$. We have shown that $H_i\ast\mathcal{X}\subset H_{i+1\bmod n}$. On the other hand, we have $|H_i\ast\mathcal{X}|\geq |H_i|=|H_{i+1\bmod n}|$. Therefore, $H_i\ast\mathcal{X}= H_{i+1\bmod n}$.

\vspace*{2mm}
5) For every $x\in\mathcal{X}$ and every $j>0$ define $$K_{x,j}=\left\{y\in\mathcal{X}:\;x\stackrel{\ast,j}{\longrightarrow} y\right\}.$$ Since $K_{x,j+1}=K_{x,j}\ast \mathcal{X}$ and since the number of subsets of $\mathcal{X}$ is finite, there exists $d_x>0$ such that the sequence $(K_{x,j})_{j\geq d_x}$ is periodic. Let $\per_x$ be the period of $(K_{x,j})_{j\geq d_x}$. Now since $K_{x,j+1}=K_{x,j}\ast \mathcal{X}$, we have $|K_{x,j+1}|\geq|K_{x,j}|$. Therefore the sequence $(|K_{x,j}|)_{j\geq d_x}$ is both periodic and non-decreasing, which implies that it is constant. 

Fix $j\geq d_x$, and let $l>0$ be such that $x\stackrel{\ast,l}{\longrightarrow} x$. For every $x'\in K_{x,j}$ we have $x\stackrel{\ast,j}{\longrightarrow} x'$ which implies that $x\stackrel{\ast,l+j}{\longrightarrow} x'$ (since $x\stackrel{\ast,l}{\longrightarrow} x$) and so $x'\in K_{x,j+l}$. Therefore, $K_{x,j}\subset K_{x,j+l}$, which implies that $K_{x,j}=K_{x,j+l}$ (since we know that $|K_{x,j}|=|K_{x,j+l}|$). Now since this is true for any $j\geq d_x$, we conclude that $\per_x$ divides every $l>0$ satisfying $x\stackrel{\ast,l}{\longrightarrow} x$. Therefore, $\per_x$ divides $\gcd\{l>0:\;x\stackrel{\ast,l}{\longrightarrow} x\}=\per(\ast,x)=n$. Hence, 
\begin{equation}
K_{x,j}=K_{x,j+kn}\;\text{for\;all}\;j\geq d_x\;\text{and\;all}\;k\geq 0.
\label{dgfnjdfjgb235rurjsm}
\end{equation}

For every $x\in\mathcal{X}$, let $i_x$ be the unique index $0\leq i_x<n$ satisfying $x\in H_{i_x}$. Clearly, $K_{x,j}\subset H_{i_x+j\bmod n}$. Now let $x'\in K_{x,j}$ and $x''\in H_{i_x+j\bmod n}$, where $j\geq d_x$. Since both $x'$ and $x''$ are in $H_{i_x+j\bmod n}$, we know from the discussion of the fourth point that we have $x'\stackrel{\ast,l_{x',x''}}{\longrightarrow} x''$ for some $l_{x',x''}\equiv 0\bmod n$. Since $n$ divides $l_{x',x''}$, we have $K_{x,j+l_{x',x''}}=K_{x,j}$ from \eqref{dgfnjdfjgb235rurjsm}. Now since $x'\in K_{x,j}$ and $x'\stackrel{\ast,l_{x',x''}}{\longrightarrow} x''$, we have $x''\in K_{x,j+l_{x',x''}}=K_{x,j}$. But this is true for every $x''\in H_{i_x+j\bmod n}$. Therefore, $H_{i_x+j\bmod n}\subset K_{x,j}$, which implies that $K_{x,j}=H_{i_x+j\bmod n}$ as we already have $K_{x,j}\subset H_{i_x+j\bmod n}$.

Define $\displaystyle d=\max_{x\in\mathcal{X}} d_x$. Let $0\leq i<n$ and $x\in H_i$. We have $i_x=i$ (since $x\in H_i$) and $d\geq d_x$. Therefore, from the above discussion we have $H_{i+d\bmod n}=H_{i_x+d\bmod n}=K_{x,d}$. Hence, for every $y\in H_{i+d\bmod n}$, we have $y\in K_{x,d}$ and so $x\stackrel{\ast,d}{\longrightarrow} y$.

\vspace*{2mm}
6) We will prove the claim by induction on $s\geq \con(\ast)$. If $s=\con(\ast)$, the claim follows from 5). Now let $s>\con(\ast)$ and suppose that the claim is true for $s-1$. Let $0\leq i< n$, $x\in H_i$ and $y\in H_{i+s\bmod n}$. Since $H_{i+s\bmod n}=H_{i+s-1\bmod n}\ast\mathcal{X}$, there exists $y'\in H_{i+s-1\bmod n}$ such that $y'\stackrel{\ast,1}{\longrightarrow}y$. Now since $y'\in H_{i+s-1\bmod n}$, it follows from the induction hypothesis that $x\stackrel{\ast,s-1}{\longrightarrow}y'$. Therefore, $x\stackrel{\ast,s}{\longrightarrow}y$.

\vspace*{2mm}
7) Let $\ast$ be an irreducible operation of period $\per(\ast)=1$. Let $\mathcal{E}_{\ast}$ be the partition defined in 4). Since $\per(\ast)=1$, the partition $\mathcal{E}_{\ast}$ contains only one element $H_0$ which must be $\mathcal{X}$. Now 5) implies that there exists $d>0$ such that any element of $\mathcal{X}=H_0$ is $\ast$-connectable to any element of $H_{0+d\bmod 1}=H_0=\mathcal{X}$ in $d$ steps. Therefore, $\ast$ is ergodic.

Conversely, if $\ast$ is ergodic, let $d=\con(\ast)$ and $n=\per(\ast)$. Define $\mathcal{E}_\ast=\{H_0,\ldots,H_{n-1}\}$ as in 4) and let $a\in H_0$. Since $a\stackrel{\ast,d}{\longrightarrow} x$ for all $x\in\mathcal{X}$, then $\mathcal{X}\subset H_{d\bmod n}$ which implies that $\mathcal{X} = H_{d\bmod n}$. Now since $|H_0|=\ldots=|H_{n-1}|=|H_{d\bmod n}|=|\mathcal{X}|$, then $H_0=\ldots=H_{n-1}=\mathcal{X}$ and $\mathcal{E}_\ast=\{\mathcal{X}\}$. Therefore, $\per(\ast)=n=|\mathcal{E}_\ast|=1$.

\vspace*{2mm}
8) If $\ast$ is ergodic, then $\per(\ast)=1$ by 7). Therefore, $\mathcal{E}_{\ast}$ contains only one element $H_0$ which must be $\mathcal{X}$. Now 6) implies that for every $s\geq\con(\ast)$, any element of $\mathcal{X}=H_0$ is $\ast$-connectable to any element of $H_{0+s\bmod 1}=H_0=\mathcal{X}$ in $s$ steps.

\vspace*{2mm}
9) and 10) are trivial.

\section{Proofs for section \ref{sec4}}
\label{appB}
\begin{proof}[Proof of Proposition \ref{propPerPer} (1)]
For every $k>0$ and every sequence $H_0\in\mathcal{H}$, $H_1\in\mathcal{H}^\ast$, \ldots, $H_{k-1}\in\mathcal{H}^{(k-1)\ast}$, define
\begin{equation}
\label{eqDefHseq}
\mathcal{H}_{H_0,\ldots,H_{k-1}}:=\big\{(\ldots((H\ast H_0)\ast H_1)\ldots \ast H_{k-1}):\; H\in\mathcal{H}\big\}.
\end{equation}
We have:
\begin{align*}
\bigcup_{X\in \mathcal{H}_{H_0,\ldots,H_{k-1}}} X&=\bigcup_{H\in\mathcal{H}}(\ldots((H\ast H_0)\ast H_1)\ldots \ast H_{k-1})=\Big(\ldots\Big(\Big(\Big(\bigcup_{H\in\mathcal{H}} H\Big)\ast H_0\Big)\ast H_1\Big)\ldots \ast H_{k-1}\Big)\\
&=(\ldots((\mathcal{X}\ast H_0)\ast H_1)\ldots \ast H_{k-1})=\mathcal{X}.
\end{align*}
Therefore, $\mathcal{H}_{H_0,\ldots,H_{k-1}}$ covers $\mathcal{X}$ for any sequence $H_0\in\mathcal{H}$, $H_1\in\mathcal{H}^\ast$, \ldots, $H_{k-1}\in\mathcal{H}^{(k-1)\ast}$. Moreover, it is easy to see from \eqref{eqDefHseq} that $\mathcal{H}_{H_0,\ldots,H_{k-1}}\subset \mathcal{H}^{k\ast}$, which implies that $\mathcal{H}^{k\ast}$ covers $\mathcal{X}$.

Fix $n>0$ and suppose that $\mathcal{H}^{n\ast}$ is not a partition. Since we have shown that $\mathcal{H}^{n\ast}$ covers $\mathcal{X}$, there must exist $X_1,X_1'\in \mathcal{H}^{n\ast}$ such that $X_1\cap X_1'\neq\o$ and $X_1\neq X_1'$. We may assume without loss of generality that $|X_1|\leq|X_1'|$. If $X_1'\setminus X_1=\o$ then $X_1'\subset X_1$ which implies that $X_1'=X_1$ (because $|X_1|\leq|X_1'|$) which is a contradiction. Therefore, we must have $X_1'\setminus X_1\neq\o$.

Since $X_1\in\mathcal{H}^{n\ast}$, there exists $H\in\mathcal{H}$ and a sequence $H_0\in\mathcal{H}$, $H_1\in\mathcal{H}^\ast$, \ldots, $H_{n-1}\in\mathcal{H}^{(n-1)\ast}$ such that $X_1= (\ldots((H\ast H_0)\ast H_1)\ldots \ast H_{n-1})$ which implies that $X_1\in\mathcal{H}_{H_0,\ldots,H_{n-1}}$. Now since we have shown that $\mathcal{H}_{H_0,\ldots,H_{n-1}}$ covers $\mathcal{X}$ and since $X_1'\setminus X_1\neq\o$, there must exist $X_2\in \mathcal{H}_{H_0,\ldots,H_{n-1}}$ such that $X_2\cap (X_1'\setminus X_1)\neq\o$. Clearly, $X_1\neq X_2$ since $X_1\cap (X_1'\setminus X_1)=\o$ and $X_2\cap (X_1'\setminus X_1)\neq\o$.

Let $p>0$ be the smallest multiple of $\per(\mathcal{H})$ which is greater than $n$, i.e., $$p=\min\{k\cdot \per(\mathcal{H}):\;k> 0,\;k\cdot \per(\mathcal{H})>n\}.$$ We have $\mathcal{H}^{p\ast}=\mathcal{H}$ since $\per(\mathcal{H})$ divides $p$. Fix $H_n\in\mathcal{H}^{n\ast},H_{n+1}\in\mathcal{H}^{(n+1)\ast},\ldots,H_{p-1}\in\mathcal{H}^{(p-1)\ast}$ and define:
\begin{itemize}
\item $A=(\ldots((X_1\ast H_n)\ast H_{n+1})\ldots \ast H_{p-1})\in\mathcal{H}^{p\ast}=\mathcal{H}$.
\item $B=(\ldots((X_2\ast H_n)\ast H_{n+1})\ldots \ast H_{p-1})\in\mathcal{H}^{p\ast}=\mathcal{H}$.
\item $C=(\ldots((X_1'\ast H_n)\ast H_{n+1})\ldots \ast H_{p-1})\in\mathcal{H}^{p\ast}=\mathcal{H}$.
\end{itemize}

We have $X_1 \cap X_1'\neq \o$ and $X_2\cap X_1'\neq\o$, which imply that
$A\cap C\neq\o$ and $B\cap C\neq\o$. Now since $A,B,C$ are members of $\mathcal{H}$ which is a partition (i.e., the elements of $\mathcal{H}$ are non-empty, disjoint and cover $\mathcal{X}$), we must have $A=B=C$. We conclude that
\begin{equation}
\label{eqinjnot}
(\ldots((X_1\ast H_n)\ast H_{n+1})\ldots \ast H_{p-1})\;=\;(\ldots((X_2\ast H_n)\ast H_{n+1})\ldots \ast H_{p-1}).
\end{equation}

We have:
\begin{itemize}
\item $\mathcal{H}_{H_0,\ldots,H_{p-1}}\subset \mathcal{H}^{p\ast}$ from the definition of $\mathcal{H}_{H_0,\ldots,H_{p-1}}$ (see \eqref{eqDefHseq}). We have shown that $\mathcal{H}_{H_0,\ldots,H_{p-1}}$ covers $\mathcal{X}$ and we know that $\mathcal{H}^{p\ast}=\mathcal{H}$ is a partition. Therefore, we must have $\mathcal{H}_{H_0,\ldots,H_{p-1}}=\mathcal{H}^{p\ast}=\mathcal{H}$.
\item The mapping $\mathcal{H}_{H_0,\ldots,H_{n-1}}\rightarrow \mathcal{H}_{H_0,\ldots,H_{p-1}}$ defined by $X\rightarrow(\ldots((X\ast H_n)\ast H_{n+1})\ldots \ast H_{p-1})$ is surjective but not injective because of \eqref{eqinjnot}. This implies that $|\mathcal{H}_{H_0,\ldots,H_{p-1}}|< |\mathcal{H}_{H_0,\ldots,H_{n-1}}|$.
\item The mapping $\mathcal{H}\rightarrow \mathcal{H}_{H_0,\ldots,H_{n-1}}$ defined by $H\rightarrow(\ldots((H\ast H_0)\ast H_1)\ldots \ast H_{n-1})$ is surjective. Therefore, $|\mathcal{H}_{H_0,\ldots,H_{n-1}}|\leq |\mathcal{H}|$.
\end{itemize}
We conclude that $|\mathcal{H}|=|\mathcal{H}_{H_0,\ldots,H_{p-1}}|< |\mathcal{H}_{H_0,\ldots,H_{n-1}}|\leq |\mathcal{H}|$ which is a contradiction. Therefore, $\mathcal{H}^{n\ast}$ must be a partition. On the other hand, we have, $(\mathcal{H}^{n\ast})^{\per(\mathcal{H})\ast}=(\mathcal{H}^{\per(\mathcal{H})\ast})^{n\ast}=\mathcal{H}^{n\ast}$ which implies that $\mathcal{H}^{n\ast}$ is a periodic partition of period $\per(\mathcal{H}^{n\ast})\leq \per(\mathcal{H})$. But since $\mathcal{H}=\mathcal{H}^{p\ast}=(\mathcal{H}^{n\ast})^{(p-n)\ast}$, we must also have $\per(\mathcal{H})=\per(\mathcal{H}^{p\ast})\leq\per(\mathcal{H}^{n\ast})$. Therefore, $\per(\mathcal{H}^{n\ast})=\per(\mathcal{H})$ for every $n>0$.
\end{proof}

\begin{mylem}
\label{lemarb}
Let $\mathcal{H}$ be a periodic partition of $(\mathcal{X},\ast)$. For every $H_2\in\mathcal{H}$, we have $$\mathcal{H}^\ast=\mathcal{H}\ast\{H_2\}= \{H_1\ast H_2:\;H_1\in\mathcal{H}\}.$$
\end{mylem}
\begin{proof}
For every $H_2\in\mathcal{H}$, we have:
\begin{align*}
\mathcal{X} &= \mathcal{X}\ast H_2= \Big(\bigcup_{H_1\in\mathcal{H}}H_1\Big)\ast H_2 =\bigcup_{H_1\in\mathcal{H}}(H_1\ast H_2).
\end{align*}
Therefore, the set $\{H_1\ast H_2:\; H_1\in\mathcal{H}\}$ covers $\mathcal{X}$ and it is a subset of $\mathcal{H}^\ast$ which is a partition of $\mathcal{X}$ by Proposition \ref{propPerPer} (1). Therefore, we must have $\mathcal{H}^{\ast}=\{H_1\ast H_2:\; H_1\in\mathcal{H}\}$.
\end{proof}

\vspace*{5mm}

\begin{proof}[Proof of Proposition \ref{propPerPer} (2)]
For every $l\geq 0$, Proposition \ref{propPerPer} (1) shows that $\mathcal{H}^{l\ast}$ is a periodic partition. If we fix $H_2\in\mathcal{H}^{l\ast}$, then we have $\mathcal{H}^{(l+1)\ast}=\{H_1\ast H_2:\; H_1\in\mathcal{H}^{l\ast}\}$ by Lemma \ref{lemarb}. Therefore,
\begin{equation}
|\mathcal{H}^{(l+1)\ast}|=\left|\{H_1\ast H_2:\; H_1\in\mathcal{H}^{l\ast}\}\right|\leq \left|\{H_1:\; H_1\in\mathcal{H}^{l\ast}\}\right|=|\mathcal{H}^{l\ast}|.
\label{eqPropPerSize}
\end{equation}

Now fix $n>0$ and let $p>0$ be the smallest multiple of $\per(\mathcal{H})$ which is greater than $n$, i.e., $p=\min\{k\cdot \per(\mathcal{H}):\;k> 0,\;k\cdot \per(\mathcal{H})>n\}$. By \eqref{eqPropPerSize} we have
$$|\mathcal{H}|=|\mathcal{H}^{p\ast}|\leq |\mathcal{H}^{(p-1)\ast}|\leq\ldots\leq |\mathcal{H}^{n\ast}|\leq\ldots\leq|\mathcal{H}|.$$
Therefore, $|\mathcal{H}^{n\ast}|=|\mathcal{H}|$ for every $n>0$.
\end{proof}

\vspace*{5mm}

\begin{proof}[Proof of Proposition \ref{lemWedge}]
Since $\mathcal{H}_1$ and $\mathcal{H}_2$ are two partitions of $\mathcal{X}$, it is easy to see that $\mathcal{H}_1\wedge\mathcal{H}_2$ is also a partition of $\mathcal{X}$. Now let $H_1,H_1'\in \mathcal{H}_1$ and $H_2,H_2'\in \mathcal{H}_2$. If $H_1\cap H_2\neq \o$ and $H_1'\cap H_2'\neq \o$, we have:
\begin{equation}
\label{eqref}
(H_1\cap H_2)\ast (H_1'\cap H_2')\;\subset\;(H_1\ast H_1')\cap (H_2\ast H_2')\in \mathcal{H}_1^{\ast}\wedge\mathcal{H}_2^{\ast}.
\end{equation}
Fix $H_1'\in\mathcal{H}_1$ and $H_2'\in\mathcal{H}_2$ such that $H_1'\cap H_2'\neq \o$. Lemma \ref{lemarb} implies that $\mathcal{H}_1^{\ast}=\{H_1\ast H_1':\; H_1\in\mathcal{H}_1\}$ and $\mathcal{H}_2^{\ast}=\{H_2\ast H_2':\; H_2\in\mathcal{H}_2\}$. Since $\mathcal{H}_1^{\ast}$ and $\mathcal{H}_2^{\ast}$ are partitions of $\mathcal{X}$, we have:
$$|\mathcal{X}|=\sum_{A_1\in\mathcal{H}_1^{\ast},A_2\in\mathcal{H}_2^{\ast}}|A_1\cap A_2| = \sum_{\substack{H_1\in\mathcal{H}_1, H_2\in\mathcal{H}_2}}|(H_1\ast H_1')\cap (H_2\ast H_2')|,$$
which implies that
\begin{align}
\label{eqe1}
|\mathcal{X}| &\geq \sum_{\substack{H_1\in\mathcal{H}_1, H_2\in\mathcal{H}_2:\\ H_1\cap H_2\neq \o}}|(H_1\ast H_1')\cap (H_2\ast H_2')|\\
&\geq \label{eqe2}
\sum_{\substack{H_1\in\mathcal{H}_1, H_2\in\mathcal{H}_2:\\ H_1\cap H_2\neq \o}}|(H_1\cap H_2)\ast (H_1'\cap H_2')|,
\end{align}
where \eqref{eqe2} follows from \eqref{eqref}. Now since $H_1'\cap H_2'\neq \o$, we have 
\begin{equation}
\label{eqref2}
|(H_1\cap H_2)\ast (H_1'\cap H_2')|\geq |H_1\cap H_2|.
\end{equation}

Therefore,
\begin{align}
&\sum_{\substack{H_1\in\mathcal{H}_1, H_2\in\mathcal{H}_2:\\ H_1\cap H_2\neq \o}}|(H_1\cap H_2)\ast (H_1'\cap H_2')|\geq \sum_{\substack{H_1\in\mathcal{H}_1, H_2\in\mathcal{H}_2:\\ H_1\cap H_2\neq \o}}|H_1\cap H_2|. \label{eqe3}
\end{align}

Now since $\mathcal{H}_1$ and $\mathcal{H}_2$ are two partitions of $\mathcal{X}$, we have 
\begin{equation}
\sum_{\substack{H_1\in\mathcal{H}_1, H_2\in\mathcal{H}_2:\\ H_1\cap H_2\neq \o}}|H_1\cap H_2|=|\mathcal{X}|.
\label{eqsndkjvskx1325}
\end{equation}
We conclude that all the inequalities in \eqref{eqe1}, \eqref{eqe2}, \eqref{eqref2} and \eqref{eqe3} are in fact equalities because if one of them were a strict inequality, we would have a contradiction with \eqref{eqsndkjvskx1325}. Therefore, for all $H_1\in\mathcal{H}_1$ and $H_2\in\mathcal{H}_2$ satisfying $H_1\cap H_2\neq \o$, we have $|(H_1\cap H_2)\ast (H_1'\cap H_2')|=|(H_1\ast H_1')\cap (H_2\ast H_2')|$. Equation \eqref{eqref} now implies that $(H_1\cap H_2)\ast (H_1'\cap H_2')=(H_1\ast H_1')\cap (H_2\ast H_2')$. We conclude that for every $H_1,H_1'\in\mathcal{H}_1$ and $H_2,H_2'\in\mathcal{H}_2$ satisfying $H_1\cap H_2\neq \o$ and $H_1'\cap H_2'\neq \o$, we have $(H_1\cap H_2)\ast (H_1'\cap H_2')=(H_1\ast H_1')\cap (H_2\ast H_2')\in \mathcal{H}_1^{\ast}\wedge\mathcal{H}_2^{\ast}$. Hence  $(\mathcal{H}_1\wedge\mathcal{H}_2)^{\ast} \subset\mathcal{H}_1^{\ast}\wedge\mathcal{H}_2^{\ast}$. We have the following:
\begin{itemize}
\item $(\mathcal{H}_1\wedge\mathcal{H}_2)^{\ast}$ covers $\mathcal{X}$ since $\mathcal{H}_1\wedge\mathcal{H}_2$ covers $\mathcal{X}$.
\item $\mathcal{H}_1^{\ast}\wedge\mathcal{H}_2^{\ast}$ is a partition of $\mathcal{X}$.
\item $(\mathcal{H}_1\wedge\mathcal{H}_2)^{\ast} \subset\mathcal{H}_1^{\ast}\wedge\mathcal{H}_2^{\ast}$.
\end{itemize}
Therefore, we must have $(\mathcal{H}_1\wedge\mathcal{H}_2)^{\ast} =\mathcal{H}_1^{\ast}\wedge\mathcal{H}_2^{\ast}$.


It follows by induction that $(\mathcal{H}_1\wedge\mathcal{H}_2)^{n\ast} =\mathcal{H}_1^{n\ast}\wedge\mathcal{H}_2^{n\ast}$ for all $n\geq 0$. In particular, for $l=\lcm(\per(\mathcal{H}_1),\per(\mathcal{H}_2))$, we have $(\mathcal{H}_1\wedge\mathcal{H}_2)^{l\ast} =\mathcal{H}_1^{l\ast}\wedge\mathcal{H}_2^{l\ast}= \mathcal{H}_1\wedge\mathcal{H}_2$, which implies that $\mathcal{H}_1\wedge\mathcal{H}_2$ is a periodic partition of period at most $\lcm(\per(\mathcal{H}_1),\per(\mathcal{H}_2))$.
\end{proof}

\section{Proof of Theorem \ref{theres}}

\label{appC}

In order to prove Theorem \ref{theres}, we need several lemmas:
\begin{mylem}
\label{lemaugm}
For any stable partition $\mathcal{H}$, and for any $\mathcal{H}$-repeatable sequence $\mathfrak{X}$, there exists an integer $l>0$ such that $\mathfrak{X}^l$ is $\mathcal{H}$-augmenting.
\end{mylem}
\begin{proof}
Let $\mathfrak{X}=(X_i)_{0\leq i< k}$ and let $x_i\in X_i$ for $0\leq i< k$. Consider the mapping $\pi:\mathcal{X}\rightarrow\mathcal{X}$ defined by $\pi(x)=(\ldots((x\ast x_0)\ast x_1)\ldots)\ast x_{k-1})$. Since $\pi$ is a permutation, there exists an integer $l>0$ such that $\pi^{l}(x)=x$ for all $x\in\mathcal{X}$. For any $A\subset\mathcal{X}$, we have $A=\pi^l(A)\subset A\ast \mathfrak{X}^l$. Therefore, $\mathfrak{X}^l$ is $\mathcal{H}$-augmenting.
\end{proof}

\begin{mydef}
Let $A\subset\mathcal{X}$. We say that an $\mathcal{H}$-augmenting sequence $\mathfrak{X}$ \emph{connects} $A$ if for every $a\in A$ we have $A\subset a\ast \mathfrak{X}$.
\end{mydef}

\begin{mylem}
\label{lemsubsres}
If there exists an $\mathcal{H}$-augmenting sequence that connects a set $A\subset \mathcal{X}$, then there exists $H\in\mathcal{H}$ such that $A\subset H$.
\end{mylem}
\begin{proof}
Let $\mathfrak{X}$ be such an $\mathcal{H}$-augmenting sequence. Let $a\in A$ and $H'\in\mathcal{H}$ be such that $a\in H'$. Define $H=H'\ast\mathfrak{X}\in \mathcal{H}^{|\mathfrak{X}|\ast}$. Since $\mathfrak{X}$ is $\mathcal{H}$-augmenting, $|\mathfrak{X}|$ divides $\per(\mathcal{H})$ and so $\mathcal{H}^{|\mathfrak{X}|\ast}=\mathcal{H}$. Therefore, $H\in\mathcal{H}$. On the other hand, $\mathfrak{X}$ connects $A$, so we have $A\subset a\ast \mathfrak{X}\subset H'\ast \mathfrak{X}= H$.
\end{proof}

\begin{mylem}
\label{lemconn}
Let $x\in\mathcal{X}$ and let $\mathfrak{X}$ be an $\mathcal{H}$-augmenting sequence. For any $y\in x\ast \mathfrak{X}$, there exists an $\mathcal{H}$-augmenting sequence $\mathfrak{X}'$ which connects $\{x,y\}$.
\end{mylem}
\begin{proof}
Let $y\in x\ast \mathfrak{X}=(\ldots((x\ast X_0)\ast X_1)\ldots)\ast X_{k-1})$. There exist $x_i\in X_i$ ($0\leq i < k$) such that $y=(\ldots((x\ast x_0)\ast x_1)\ldots)\ast x_{k-1})$. Define the mapping $\pi:\mathcal{X}\rightarrow\mathcal{X}$ as $\pi(a)=(\ldots((a\ast x_0)\ast x_2)\ldots)\ast x_{k-1})$ for every $a\in\mathcal{X}$. Clearly, $\pi$ is a permutation. The fact that $y=\pi(x)$ implies that $x$ and $y$ belong to the same cycle of the permutation $\pi$. Therefore, there exists $s>0$ such that $x=\pi^{s}(y)$. Let $\mathfrak{X}'=\mathfrak{X}^s$. It is easy to see that $\mathfrak{X}'$ is $\mathcal{H}$-augmenting. Moreover, we have:
\begin{itemize}
\item $x\in y\ast \mathfrak{X}'$ because $x=\pi^{s}(y)$, and $y\in y\ast \mathfrak{X}'$ because $\mathfrak{X}'$ is $\mathcal{H}$-augmenting. Therefore, $\{x,y\}\subset y\ast\mathfrak{X}'$.
\item $y\in x\ast \mathfrak{X}$ by assumption and $x\in x\ast \mathfrak{X}$ since $\mathfrak{X}$ is $\mathcal{H}$-augmenting. Therefore, $\{x,y\}\subset x\ast \mathfrak{X}$. On the other hand, $x\ast \mathfrak{X}\subset(x\ast \mathfrak{X})\ast \mathfrak{X}^{s-1}$ since $\mathfrak{X}^{s-1}$ is $\mathcal{H}$-augmenting. Hence $\{x,y\}\subset (x\ast \mathfrak{X})\ast \mathfrak{X}^{s-1}=x\ast\mathfrak{X}'$.
\end{itemize}
We conclude that $\mathfrak{X}'$ connects $\{x,y\}$.
\end{proof}

\begin{mylem}
\label{lemtrans}
If there exists an $\mathcal{H}$-augmenting sequence that connects a set $A\subset \mathcal{X}$, and if there exists an $\mathcal{H}$-augmenting sequence that connects another set $B\subset \mathcal{X}$ such that $A\cap B\neq \o$, then there exists an $\mathcal{H}$-augmenting sequence that connects $A\cup B$.
\end{mylem}
\begin{proof}
Let $\mathfrak{X}$ be an $\mathcal{H}$-augmenting sequence that connects $A$, and let $\mathfrak{X}'$ be an $\mathcal{H}$-augmenting sequence that connects $B$. Let $\mathfrak{X}''=(\mathfrak{X},\mathfrak{X}',\mathfrak{X})$ be the $\mathcal{H}$-repeatable sequence that is obtained by concatenating $\mathfrak{X}$, $\mathfrak{X}'$ and $\mathfrak{X}$. Clearly, $\mathfrak{X}''$ is $\mathcal{H}$-augmenting. Fix $x\in A\cap B$ and let $y\in A\cup B$. We have the following:
\begin{itemize}
\item If $y\in A$, then $A\subset y\ast\mathfrak{X}$. In particular, $x\in y\ast\mathfrak{X}$. Now since $x\in B$ and since $\mathfrak{X}'$ connects $B$, we have $B\subset x\ast \mathfrak{X}'$. Therefore, $B\subset (y\ast \mathfrak{X})\ast \mathfrak{X}'$.
\item If $y\in B$, then $y\in y\ast \mathfrak{X}$ since $\mathfrak{X}$ is $\mathcal{H}$-augmenting. Now since $y\in B$ and since $\mathfrak{X}'$ connects $B$, we have $B\subset y\ast \mathfrak{X}'$. Therefore, $B\subset (y\ast \mathfrak{X})\ast \mathfrak{X}'$.
\end{itemize}
We conclude that for any $y\in A\cup B$, we have $B\subset (y\ast \mathfrak{X})\ast \mathfrak{X}'$. This implies that:
\begin{itemize}
\item $B\subset ((y\ast \mathfrak{X})\ast \mathfrak{X}')\ast \mathfrak{X}= y\ast \mathfrak{X}''$ since $\mathfrak{X}$ is $\mathcal{H}$-augmenting.
\item Since $B\subset (y\ast \mathfrak{X})\ast \mathfrak{X}'$, we have $x\in (y\ast \mathfrak{X})\ast \mathfrak{X}'$. Now since $x\in A$ and since $\mathfrak{X}$ connects $A$, we have $A\subset x\ast \mathfrak{X}$. Therefore, $A\subset ((y\ast \mathfrak{X})\ast \mathfrak{X}')\ast \mathfrak{X}= y\ast \mathfrak{X}''$.
\end{itemize}
We conclude that $A\cup B\subset y\ast \mathfrak{X}''$ for any $y\in A\cup B$. Hence $\mathfrak{X}''$ connects $A\cup B$.
\end{proof}

\begin{mydef}
For every stable partition $\mathcal{H}$ of $(\mathcal{X},\ast)$, define the \emph{connectivity relation} $R_{\mathcal{H}}$ of $\mathcal{H}$ on $\mathcal{X}$ as follows: $a R_{\mathcal{H}} b$ if and only if there exists an $\mathcal{H}$-augmenting sequence that connects $\{a,b\}$.
\end{mydef}

\begin{mylem}
For every stable partition $\mathcal{H}$, $R_{\mathcal{H}}$ is an equivalence relation.
\end{mylem}
\begin{proof}
Clearly, $R_{\mathcal{H}}$ is symmetric. Lemma \ref{lemtrans} shows that $R_{\mathcal{H}}$ is transitive. In order to show that $R_{\mathcal{H}}$ is reflexive, let $x\in\mathcal{X}$, and let $\mathfrak{X}$ be an arbitrary $\mathcal{H}$-repeatable sequence. Lemma \ref{lemaugm} implies that there exists $l>0$ such that $\mathfrak{X}^l$ is $\mathcal{H}$-augmenting. We have $x\in x\ast \mathfrak{X}^l$ and so $\mathfrak{X}^l$ connects $\{x\}$. Therefore, $x R_{\mathcal{H}} x$ for every $x\in\mathcal{X}$, hence $R_{\mathcal{H}}$ is reflexive. We conclude that $R_{\mathcal{H}}$ is an equivalence relation.
\end{proof}

\begin{mynot}
For every stable partition $\mathcal{H}$, we denote the set of equivalence classes of its connectivity relation $R_{\mathcal{H}}$ by $\mathcal{K}_{\mathcal{H}}$.
\end{mynot}

\begin{mylem}
\label{lemres1}
Let $\mathcal{H}$ be a stable partition and let $K\in\mathcal{K}_{\mathcal{H}}$. We have:
\begin{itemize}
\item For every $x\in K$ and every $\mathcal{H}$-augmenting sequence $\mathfrak{X}'$, $x\ast\mathfrak{X}'\subset K$.
\item There exists an $\mathcal{H}$-augmenting sequence $\mathfrak{X}$ satisfying $x\ast \mathfrak{X}=K$ for all $x\in K$.
\end{itemize}
\end{mylem}
\begin{proof}
For every $K\in\mathcal{K}_\mathcal{H}$, every $x\in K$, every $\mathcal{H}$-augmenting sequence $\mathfrak{X}'$, and every $y\in x\ast\mathfrak{X}'$, we have $x R_{\mathcal{H}} y$ because of Lemma \ref{lemconn}, so $y\in K$. This shows that $x\ast\mathfrak{X}'\subset K$.

Now fix $K\in\mathcal{K}_{\mathcal{H}}$ and let $K=\{a_1,\ldots,a_r\}$ where $r=|K|$. For each $1\leq i\leq r$, define $K_i:=\{a_1,\ldots,a_i\}$. Since $a_1R_{\mathcal{H}}a_1$ there exists an $\mathcal{H}$-augmenting sequence that connects $K_1$. Now let $1<i\leq r$ and suppose that there exists an $\mathcal{H}$-augmenting sequence that connects $K_{i-1}$. Since $a_{i-1}R_{\mathcal{H}}a_{i}$, there exists an $\mathcal{H}$-augmenting sequence that connects $\{a_{i-1},a_{i}\}$. Now since $K_{i-1}\cap\{a_{i-1},a_{i}\}=\{a_{i-1}\}\neq \o$, Lemma \ref{lemtrans} implies that there exists an $\mathcal{H}$-augmenting sequence that connects $K_{i-1}\cup\{a_{i-1},a_{i}\}=K_{i}$, and so the claim is true for $i$. By induction we conclude that the claim is true for every $1\leq i\leq r$. In particular, there exists an $\mathcal{H}$-augmenting sequence $\mathfrak{X}$ that connects $K_r=K$.

Let $x\in K$. Since $\mathfrak{X}$ connects $K$, we have $K\subset x\ast \mathfrak{X}$, which implies that $x\ast \mathfrak{X}= K$ as we already have $x\ast \mathfrak{X}\subset K$.
\end{proof}

\begin{mylem}
\label{lemres2}
If $\ast$ is an ergodic operation on $\mathcal{X}$, then for every stable partition $\mathcal{H}$, we have the following:
\begin{itemize}
\item $\mathcal{K}_{\mathcal{H}^{l\ast}}$ is a balanced partition and $\|\mathcal{K}_{\mathcal{H}^{l\ast}}\|=\|\mathcal{K}_{\mathcal{H}}\|$ for all $l\geq 0$.
\item For every $l\geq 0$, $K_1\in\mathcal{K}_{\mathcal{H}}$, $K_2\in\mathcal{K}_{\mathcal{H}^{l\ast}}$, and every $a\in K_1$, there exists an $\mathcal{H}$-sequence $\mathfrak{X}_{a,K_2}$ such that $|\mathfrak{X}_{a,K_2}|\equiv l \bmod n$ and $K_2=a\ast \mathfrak{X}_{a,K_2}=K_1\ast \mathfrak{X}_{a,K_2}$.
\end{itemize}
\end{mylem}
\begin{proof}
Let $K_1\in\mathcal{K}_{\mathcal{H}}$, $l\geq 0$ and $K_2\in\mathcal{K}_{\mathcal{H}^{l\ast}}$. Let $n=\per(\mathcal{H})$, $k_1=\con(\ast)n+l$ and $k_2=\con(\ast)n+(-l\bmod n)$. Choose $a\in K_1$ and $b\in K_2$. Since $\ast$ is ergodic and since $k_1\geq \con(\ast)$ and $k_2\geq \con(\ast)$, it follows from Proposition \ref{lemerg} that there exist $x_0,\ldots,x_{k_1-1}\in\mathcal{X}$ such that $b=(\ldots((a\ast x_0)\ast x_1)\ldots)\ast x_{k_1-1})$ and there exist $y_0,\ldots,y_{k_2-1}\in\mathcal{X}$ such that $a=(\ldots((b\ast y_0)\ast y_1)\ldots)\ast y_{k_2-1})$. Let $\mathfrak{X}_1=(X_i)_{0\leq i<k_1}$ and $\mathfrak{X}_2=(Y_i)_{0\leq i<k_2}$ be such that $x_i\in X_i\in \mathcal{H}^{i\ast}$ for $0\leq i<k_1$ and $y_i\in Y_i\in \mathcal{H}^{(l+i)\ast}$ for $0\leq i<k_2$. Clearly, $b\in a\ast\mathfrak{X}_1$ and $a\in b\ast\mathfrak{X}_2$. The concatenation $\mathfrak{X}=(\mathfrak{X}_1,\mathfrak{X}_2)$ is an $\mathcal{H}$-repeatable sequence since $n$ divides $k_1+k_2$. Lemma \ref{lemaugm} implies that there exists an integer $s>0$ such that $\mathfrak{X}^s$ is $\mathcal{H}$-augmenting. Lemma \ref{lemres1}, applied to $\mathcal{K}_{\mathcal{H}^{l\ast}}$, implies the existence of an $\mathcal{H}^{l\ast}$-augmenting sequence $\mathfrak{X}'$ such that $b\ast\mathfrak{X}'=K_2$.

Consider the sequence $\mathfrak{X}''=(\mathfrak{X}_1,\mathfrak{X}',\mathfrak{X}_2,\mathfrak{X}^{s-1})$. It is easy to see that $\mathfrak{X}''$ is $\mathcal{H}$-augmenting and so $K_1\subset K_1\ast\mathfrak{X}''$. On the other hand, since $\mathfrak{X}''$ is $\mathcal{H}$-augmenting, Lemma \ref{lemres1} shows that for every $x\in K_1$ we have $x\ast\mathfrak{X}''\subset K_1$, which means that $K_1\ast\mathfrak{X}''\subset K_1$. Therefore, $K_1= K_1\ast\mathfrak{X}''$. Moreover, since $b\in a\ast\mathfrak{X}_1$ and $b\ast\mathfrak{X}'=K_2$, we have 
\begin{equation}
\label{eqTTT328fhdew3A}
K_2\subset (a\ast\mathfrak{X}_1)\ast \mathfrak{X}'\subset (K_1\ast\mathfrak{X}_1)\ast \mathfrak{X}'
\end{equation}
which implies that $|K_2|\leq |(K_1\ast\mathfrak{X}_1)\ast \mathfrak{X}'| \leq |(((K_1\ast\mathfrak{X}_1)\ast \mathfrak{X}')\ast\mathfrak{X}_2)\ast \mathfrak{X}^{s-1}|= |K_1\ast\mathfrak{X}''|=|K_1|$. By exchanging the roles of $K_1$ and $K_2$, we get $|K_1|\leq |K_2|$. Therefore, $|K_2|=|K_1|$ for every $K_1\in\mathcal{K}_{\mathcal{H}}$ and every $K_2\in\mathcal{K}_{\mathcal{H}^{l\ast}}$. We conclude that both $\mathcal{K}_{\mathcal{H}}$ and $\mathcal{K}_{\mathcal{H}^{l\ast}}$ are balanced partitions and $\|\mathcal{K}_{\mathcal{H}}\|=\|\mathcal{K}_{\mathcal{H}^{l\ast}}\|$.

Now define $\mathfrak{X}_{a,K_2}=(\mathfrak{X}_1,\mathfrak{X}')$. Since $\mathfrak{X}_{a,K_2}$ is an initial segment of $\mathfrak{X}''$, we have $|K_1\ast\mathfrak{X}_{a,K_2}|\leq |K_1\ast\mathfrak{X}''|$. But we have shown that $K_1\ast\mathfrak{X}''=K_1$ and $|K_1|=|K_2|$, so we must have $|K_1\ast\mathfrak{X}_{a,K_2}|\leq |K_2|$. Moreover, we have $K_2\subset a\ast\mathfrak{X}_{a,K_2} \subset K_1\ast\mathfrak{X}_{a,K_2}$ from \eqref{eqTTT328fhdew3A}. We conclude that $K_2=a\ast \mathfrak{X}_{a,K_2}=K_1\ast \mathfrak{X}_{a,K_2}$.
\end{proof}

\begin{mylem}
\label{lemres3prime}
Let $\mathcal{H}$ be a stable partition of $(\mathcal{X},\ast)$ where $\ast$ is ergodic. For every $K\in\mathcal{K}_{\mathcal{H}}$ and every $\mathcal{H}$-sequence $\mathfrak{X}$, we have $|K\ast\mathfrak{X}|=|K|=\|\mathcal{K}_{\mathcal{H}}\|$.
\end{mylem}
\begin{proof}
Let $K'=K\ast \mathfrak{X}$ and $l=|\mathfrak{X}|$, and let $\mathfrak{X}'=(X_i')_{0\leq i< (-l\bmod n)}$ be an arbitrary $\mathcal{H}^{l\ast}$-sequence of length $(-l\bmod n)$. Clearly, $(\mathfrak{X},\mathfrak{X}')$ is $\mathcal{H}$-repeatable. Lemma \ref{lemaugm} implies that there exists an integer $s>0$ such that $(\mathfrak{X},\mathfrak{X}')^{s}$ is $\mathcal{H}$-augmenting. We have $K\subset K\ast (\mathfrak{X},\mathfrak{X}')^s$. On the other hand, Lemma \ref{lemres1} implies that $K\ast (\mathfrak{X},\mathfrak{X}')^s\subset K$. Therefore, $K=K\ast (\mathfrak{X},\mathfrak{X}')^{s}=K'\ast(\mathfrak{X}',(\mathfrak{X},\mathfrak{X}')^{s-1})$ which implies that $|K'|\leq |K|$. We also have $|K|\leq |K'|$ since $K'=K\ast\mathfrak{X}$. Thus, $|K'|=|K|=\|\mathcal{K}_{\mathcal{H}}\|$.
\end{proof}

\begin{mylem}
\label{lemres3}
Let $\mathcal{H}$ be a stable partition of $(\mathcal{X},\ast)$ where $\ast$ is ergodic. Let $K\in\mathcal{K}_{\mathcal{H}}$ and $l>0$. If $\mathfrak{X}=(X_i)_{0\leq i<l}$ is an $\mathcal{H}$-sequence, then $K\ast\mathfrak{X}\in \mathcal{K}_{\mathcal{H}^{l\ast}}$.
\end{mylem}
\begin{proof}
Let $K'=K\ast \mathfrak{X}$. Fix $x\in K'$ and let $K''\in\mathcal{K}_{\mathcal{H}^{l\ast}}$ be chosen so that $x\in K''$. Lemma \ref{lemres1} implies the existence of an $\mathcal{H}^{l\ast}$-augmenting sequence $\mathfrak{X}''$ such that $x\ast \mathfrak{X}''= K''$. We have $K''\subset K'\ast\mathfrak{X}''$ since $x\in K'$, and $K'\subset K'\ast\mathfrak{X}''$ since $\mathfrak{X}''$ is $\mathcal{H}^{l\ast}$-augmenting. Therefore, $K'\cup K''\subset K'\ast \mathfrak{X}''$. On the other hand, we have the following:
\begin{itemize}
\item $|K'|=|K\ast\mathfrak{X}|=|K|=\|\mathcal{K}_{\mathcal{H}}\|$ from Lemma \ref{lemres3prime}.
\item $(\mathfrak{X},\mathfrak{X}'')$ is an $\mathcal{H}$-sequence, so Lemma \ref{lemres3prime} implies that $|K\ast (\mathfrak{X},\mathfrak{X}'')|=|K|=\|\mathcal{K}_{\mathcal{H}}\|$. Now since $K'\ast \mathfrak{X}'' = K\ast (\mathfrak{X},\mathfrak{X}'')$, we deduce that $|K'\ast \mathfrak{X}''|=\|\mathcal{K}_{\mathcal{H}}\|$.
\item Lemma \ref{lemres2} implies that $\|\mathcal{K}_{\mathcal{H}}\|=\|\mathcal{K}_{\mathcal{H}^{l\ast}}\|$, so $|K''|=\|\mathcal{K}_{\mathcal{H}^{l\ast}}\|=\|\mathcal{K}_{\mathcal{H}}\|$.
\end{itemize}
Therefore, $|K''|=|K'|=|K'\ast \mathfrak{X}''|=\|\mathcal{K}_{\mathcal{H}}\|$ and $K'\cup K''\subset K'\ast \mathfrak{X}''$, hence $K'=K''$ and $K'\in \mathcal{K}_{\mathcal{H}^{l\ast}}$.
\end{proof}

\begin{mylem}
\label{propres}
Let $\mathcal{H}$ be a stable partition of $(\mathcal{X},\ast)$ where $\ast$ is ergodic. $\mathcal{K}_{\mathcal{H}}$ is a sub-stable partition of $\mathcal{H}$ and $\mathcal{K}_{\mathcal{H}^{l\ast}}={\mathcal{K}_{\mathcal{H}}}^{l\ast}$ for all $l\geq 0$.
\end{mylem}
\begin{proof}
We will prove that $\mathcal{K}_{\mathcal{H}^{l\ast}}={\mathcal{K}_{\mathcal{H}}}^{l\ast}$ by induction on $l\geq 0$. The statement is trivial for $l=0$. Now let $l>0$ and suppose that $\mathcal{K}_{\mathcal{H}^{(l-1)\ast}}={\mathcal{K}_{\mathcal{H}}}^{(l-1)\ast}$. Let $K\in{\mathcal{K}_{\mathcal{H}}}^{l\ast}=({\mathcal{K}_{\mathcal{H}} }^{(l-1)\ast})^{\ast}= (\mathcal{K}_{\mathcal{H}^{(l-1)\ast}})^{\ast}$. There exist $K_1,K_2\in \mathcal{K}_{\mathcal{H}^{(l-1)\ast}}={\mathcal{K}_{\mathcal{H}}}^{(l-1)\ast}$ such that $K=K_1\ast K_2$. Let $H_2\in\mathcal{H}^{(l-1)\ast}$ be chosen such that $K_2\subset H_2$ (Lemma \ref{lemsubsres} guarantees the existence of $H_2$). From Lemma \ref{lemres3}, we have $K_1\ast H_2\in \mathcal{K}_{\mathcal{H}^{l\ast}}$ and so $|K_1\ast H_2|=\|\mathcal{K}_{\mathcal{H}^{l\ast}}\|\stackrel{(a)}{=}\|\mathcal{K}_{\mathcal{H}^{(l-1)\ast}}\|=|K_1|$, where (a) follows from Lemma \ref{lemres2}. We have $K_1\ast K_2\subset K_1\ast H_2$ and $|K_1|\leq |K_1\ast K_2|\leq |K_1\ast H_2|=|K_1|$. Therefore, $K=K_1\ast K_2=K_1\ast H_2$ which implies that $K\in \mathcal{K}_{\mathcal{H}^{l\ast}}$. This shows that ${\mathcal{K}_{\mathcal{H}}}^{l\ast}\subset \mathcal{K}_{\mathcal{H}^{l\ast}}$, which implies that ${\mathcal{K}_{\mathcal{H}}}^{l\ast}= \mathcal{K}_{\mathcal{H}^{l\ast}}$ since ${\mathcal{K}_{\mathcal{H}}}^{l\ast}$ covers $\mathcal{X}$ and $\mathcal{K}_{\mathcal{H}^{l\ast}}$ is a partition of $\mathcal{X}$.

We conclude that ${\mathcal{K}_{\mathcal{H}}}^{l\ast}= \mathcal{K}_{\mathcal{H}^{l\ast}}$ for all $l\geq 0$. In particular, ${\mathcal{K}_{\mathcal{H}}}^{n\ast}=\mathcal{K}_{\mathcal{H}^{n\ast}}= \mathcal{K}_{\mathcal{H}}$, where $n=\per(\mathcal{H})$ so $\mathcal{K}_{\mathcal{H}}$ is periodic. Moreover, Lemma \ref{lemres2} shows that $\mathcal{K}_{\mathcal{H}}$ is balanced. Therefore, $\mathcal{K}_{\mathcal{H}}$ is a stable partition. Lemma \ref{lemsubsres} now implies that $\mathcal{K}_{\mathcal{H}}$ is a sub-stable partition of $\mathcal{H}$.
\end{proof}

\begin{myprop}
\label{lemres4}
Let $\mathcal{H}$ be a stable partition of $(\mathcal{X},\ast)$ where $\ast$ is ergodic, and let $\mathcal{K}$ be a partition of $\mathcal{X}$ which satisfies the following two conditions:
\begin{itemize}
\item For every $K\in\mathcal{K}$ and every $x\in K$, there exists an $\mathcal{H}$-augmenting sequence $\mathfrak{X}$ such that $x\ast \mathfrak{X}=K$.
\item For every $K\in\mathcal{K}$, every $x\in K$, and every $\mathcal{H}$-augmenting sequence $\mathfrak{X}'$, $x\ast \mathfrak{X}'\subset K$.
\end{itemize}
Then $\mathcal{K}=\mathcal{K}_{\mathcal{H}}$.
\end{myprop}
\begin{proof}
Fix $x\in\mathcal{X}$ and let $K_{1,x}\in\mathcal{K}_{\mathcal{H}}$ and $K_{2,x}\in\mathcal{K}$ be chosen such that $x\in K_{1,x}$ and $x\in K_{2,x}$. Lemma \ref{lemres1} implies the existence of an $\mathcal{H}$-augmenting sequence $\mathfrak{X}_1$ such that $x\ast \mathfrak{X}_1=K_{1,x}$, and the first condition of the proposition implies the existence of an $\mathcal{H}$-augmenting sequence $\mathfrak{X}_2$ such that $x\ast \mathfrak{X}_2=K_{2,x}$. The second condition of the proposition implies that $x\ast \mathfrak{X}_1\subset K_{2,x}$, and Lemma \ref{lemres1} implies that $x\ast \mathfrak{X}_2 \subset K_{1,x}$. Therefore, $K_{1,x}\subset K_{2,x}$ and $K_{2,x}\subset K_{1,x}$ which implies that $K_{1,x}=K_{2,x}$. Since this is true for all $x\in\mathcal{X}$, we conclude that $\mathcal{K}=\mathcal{K}_{\mathcal{H}}$.
\end{proof}
\vspace*{3mm}
Now we are ready to prove Theorem \ref{theres}:
\begin{proof}[Proof of Theorem \ref{theres}]
Lemma \ref{propres} shows that $\mathcal{K}_{\mathcal{H}}$ is a sub-stable partition of $\mathcal{H}$ satisfying ${\mathcal{K}_{\mathcal{H}}}^{l\ast}=\mathcal{K}_{\mathcal{H}^{l\ast}}$ for all $l\geq 0$. Moreover, we have:
\begin{itemize}
\item For every $K\in\mathcal{K}_{\mathcal{H}}$ and every $\mathcal{H}$-sequence $\mathfrak{X}$, we have $K\ast\mathfrak{X} \in \mathcal{K}_{\mathcal{H}^{|\mathfrak{X}|\ast}}={\mathcal{K}_{\mathcal{H}}}^{|\mathfrak{X}|\ast}$ by Lemma \ref{lemres3}.
\item For every $K\in\mathcal{K}_{\mathcal{H}}$ and every $x\in K$, Lemma \ref{lemres1} shows that there exists an $\mathcal{H}$-augmenting sequence $\mathfrak{X}$ such that $x\ast \mathfrak{X}=K$.
\item For every $K\in\mathcal{K}_{\mathcal{H}}$, every $x\in K$, and every $\mathcal{H}$-augmenting sequence $\mathfrak{X}'$, we have $x\ast \mathfrak{X}'\subset K$ by Lemma \ref{lemres1}.
\end{itemize}
This shows the existence part of Theorem \ref{theres}. The uniqueness follows from Proposition \ref{lemres4}.
\end{proof}

\section{Proof of Proposition \ref{stastaerg}}

\label{appD}

\begin{mydef}
Let $\mathcal{A}$ be an $\mathcal{X}$-cover. Define the relation $P_{\mathcal{A}}$ on $\mathcal{X}$ as follows: $x P_{\mathcal{A}} y$ if and only if there exists a finite sequence $(A_i)_{1\leq i\leq n}$ such that $x\in A_1$, $y\in A_n$, $A_i\in\mathcal{A}$ for all $1\leq i\leq n$, and $A_i\cap A_{i+1}\neq \o$ for all $1\leq i<n$. Clearly, $P_{\mathcal{A}}$ is an equivalence relation on $\mathcal{X}$. The set of equivalence classes of $P_{\mathcal{A}}$ (denoted by $\mathcal{P}(\mathcal{A})$) is called the partition of $\mathcal{X}$ \emph{generated by} $\mathcal{A}$.
\end{mydef}

\begin{mylem}
\label{lemlalababa}
Let $\mathcal{A}$ be an $\mathcal{X}$-cover. For every $B\in\mathcal{P}(\mathcal{A})$, there exists a finite sequence $(A_i)_{1\leq i\leq n}$ such that $\displaystyle B=\bigcup_{i=1}^n A_i$, $A_i\in\mathcal{A}$ for all $1\leq i\leq n$, and $A_i\cap A_{i+1}\neq \o$ for all $1\leq i<n$.
\end{mylem}
\begin{proof}
Let $B\in\mathcal{P}(\mathcal{A})$ and let $x\in B$. We say that a sequence $(A_i)_{1\leq i\leq n}$ is $(x,\mathcal{A})$-connected if $x\in A_1$, $A_i\in\mathcal{A}$ for all $1\leq i\leq n$, and $A_i\cap A_{i+1}\neq \o$ for all $1\leq i<n$. If $(A_i)_{1\leq i\leq n}$ is such a sequence, we clearly have $x P_{\mathcal{A}} y$ for every $y\in \displaystyle \bigcup_{i=1}^n A_i$. Therefore, $\displaystyle \bigcup_{i=1}^n A_i\subset B$. 

Let $A_1\in\mathcal{A}$ be such that $x\in A_1$. The sequence $(A_1)$ of length 1 is $(x,\mathcal{A})$-connected. Therefore, there exists at least one $(x,\mathcal{A})$-connected sequence. Now consider an $(x,\mathcal{A})$-connected sequence $(A_i)_{1\leq i\leq n}$ such that $\displaystyle \bigcup_{i=1}^n A_i$ is maximal. If $\displaystyle \bigcup_{i=1}^n A_i\neq B$, there exists $y\in B$ such that $\displaystyle y\notin \bigcup_{i=1}^n A_i$. Let $x'\in A_n$. Since $x',y\in B$, $x' P_{\mathcal{A}} y$ and so there exists a sequence $(A_i')_{1\leq i\leq m}$ such that $x'\in A_1'$, $y\in A_m'$, $A_i'\in\mathcal{A}$ for all $1\leq i\leq m$, and $A_i'\cap A_{i+1}'\neq \o$ for all $1\leq i<m$. Consider the sequence $(A_i'')_{1\leq i\leq n+m}$ defined by $A_i''=A_i$ for $1\leq i\leq n$ and $A_i''=A_{i-n}'$ for $n+1\leq i\leq n+m$. Since $x'\in A_n\cap A_1'= A_n''\cap A_{n+1}''$, $(A_i'')_{1\leq i\leq n+m}$ is $(x,\mathcal{A})$-connected. We have $\displaystyle \bigcup_{i=1}^n A_i\subsetneq \bigcup_{i=1}^{n+m} A_i''$ since $\displaystyle y\in \bigcup_{i=1}^{n+m} A_i''$ and $\displaystyle y\notin \bigcup_{i=1}^n A_i$. This contradicts the maximality of $\displaystyle \bigcup_{i=1}^n A_i$. Therefore, we must have $\displaystyle \bigcup_{i=1}^n A_i=B$.
\end{proof}

\begin{mylem}
\label{lemsublem}
Let $\ast$ be a uniformity preserving operation on a set $\mathcal{X}$, and let $\mathcal{A}$ be an $\mathcal{X}$-cover. For every $n>0$ and every $A\in \mathcal{A}^{n\ast}$, there exists $B\in\mathcal{P}(\mathcal{A})^{n\ast}$ such that $A\subset B$.
\end{mylem}
\begin{proof}
We will show the lemma by induction on $n$. The lemma is trivial for $n=0$.

Now let $n>0$ and suppose that the lemma is true for $n-1$. Let $A\in \mathcal{A}^{n\ast}$, there exists $A_1,A_2\in\mathcal{A}^{(n-1)\ast}$ such that $A=A_1\ast A_2$. The induction hypothesis implies the existence of two sets $B_1,B_2\in\mathcal{P}(\mathcal{A})^{(n-1)\ast}$ such that $A_1\subset B_1$ and $A_2\subset B_2$. We have $A=A_1\ast A_2\subset B_1\ast B_2$ and $B_1\ast B_2\in\mathcal{P}(\mathcal{A})^{n\ast}$.
\end{proof}

\begin{mylem}
\label{lemindkif}
Let $\ast$ be a uniformity preserving operation on a set $\mathcal{X}$, and let $\mathcal{A}$ be an $\mathcal{X}$-cover. For every $n\geq 0$, we have $\mathcal{P}\big(\mathcal{P}(\mathcal{A})^{n\ast}\big)=\mathcal{P}(\mathcal{A}^{n\ast})$.
\end{mylem}
\begin{proof}
We will show the lemma by induction on $n$. The lemma is trivial for $n=0$.

Now let $n>0$ and suppose that $\mathcal{P}\big(\mathcal{P}(\mathcal{A})^{(n-1)\ast}\big)=\mathcal{P}(\mathcal{A}^{(n-1)\ast})$, which means that for every $x,y\in\mathcal{X}$, we have $x P_{\mathcal{A}^{(n-1)\ast}}y$ if and only if $x P_{\mathcal{P}(\mathcal{A})^{(n-1)\ast}}y$.

Let $x,y\in\mathcal{X}$ be such that $x P_{\mathcal{P}(\mathcal{A})^{n\ast}} y$. There exists a sequence $(D_j)_{1\leq j\leq m}$ such that: $x\in D_1$, $y\in D_m$, $D_j\in\mathcal{P}(\mathcal{A})^{n\ast}$ for $1\leq j\leq m$, and $D_j\cap D_{j+1}\neq \o$ for $1\leq j< m$. Define $x_1=x$ and $x_{m+1}=y$, and for each $2\leq j\leq m$, choose $x_j\in D_{j-1}\cap D_j$. For every $1\leq j\leq m$, we have $x_j,x_{j+1}\in D_j$ and $D_j\in\mathcal{P}(\mathcal{A})^{n\ast}$. We are going to show that $x_j P_{\mathcal{A}^{n\ast}} x_{j+1}$ for every $1\leq j\leq m$ which will imply that $x P_{\mathcal{A}^{n\ast}} y$.

Fix $j\in\{1,\ldots,m\}$. Since $D_j\in\mathcal{P}(\mathcal{A})^{n\ast}$, there exist $D_j',D_j''\in\mathcal{P}(\mathcal{A})^{(n-1)\ast}$ such that $D_j=D_j'\ast D_j''$. Moreover, since $x_j,x_{j+1}\in D_j$ there exist $a_j',b_{j+1}'\in D_j'$ and $a_j'',b_{j+1}''\in D_j''$ such that $x_j=a_j'\ast a_j''$ and $x_{j+1}=b_{j+1}'\ast b_{j+1}''$. We have $a_j' P_{\mathcal{P}(\mathcal{A})^{(n-1)\ast}} b_{j+1}'$ and $a_j'' P_{\mathcal{P}(\mathcal{A})^{(n-1)\ast}} b_{j+1}''$. Therefore, from the induction hypothesis we have $a_j' P_{\mathcal{A}^{(n-1)\ast}} b_{j+1}'$ and $a_j'' P_{\mathcal{A}^{(n-1)\ast}} b_{j+1}''$. There exist two sequences $(A_i')_{1\leq i\leq m_j'}$ and $(A_i'')_{1\leq i\leq m_j''}$ such that:
\begin{itemize}
\item $a_j'\in A_1'$, $b_{j+1}'\in A_{m_j'}'$, $A_i'\in\mathcal{A}^{(n-1)\ast}$ for $1\leq i\leq m_j'$, and $A_i'\cap A_{i+1}'\neq \o$ for $1\leq i< m_j'$.
\item $a_j''\in A_1''$, $b_{j+1}''\in A_{m_j''}''$, $A_i''\in\mathcal{A}^{(n-1)\ast}$ for $1\leq i\leq m_j''$, and $A_i''\cap A_{i+1}''\neq \o$ for $1\leq i< m_j''$.
\end{itemize}
Now consider the sequence $(A_i)_{1\leq i\leq m_j'+m_j''}$ defined as $A_i=A_i'\ast A_1''$ for $1\leq i\leq m_j'$, and $A_i=A_{m_j'}'\ast A_{i-m_j'}''$ for $m_j'+1\leq i\leq m_j'+m_j''$. The sequence $(A_i)_{1\leq i\leq m_j'+m_j''}$ satisfies the following: $x_j=a_j'\ast a_j''\in A_1$, $x_{j+1}=b_{j+1}'\ast b_{j+1}''\in A_{m_j'+m_j''}$ and $A_i\in\mathcal{A}^{n\ast}$ for $1\leq i\leq m_j'+m_j''$. Moreover, it is easy to see that $A_i\cap A_{i+1}\neq \o$ for $1\leq i< m_j'+m_j''$. Therefore, $x_j P_{\mathcal{A}^{n\ast}} x_{j+1}$. Now since this is true for all $1\leq j\leq m$, we have $x_1 P_{\mathcal{A}^{n\ast}} x_{m+1}$ and so $x P_{\mathcal{A}^{n\ast}} y$. We conclude that for every $x,y\in\mathcal{X}$, $x P_{\mathcal{P}(\mathcal{A})^{n\ast}} y$ implies $x P_{\mathcal{A}^{n\ast}} y$.

Now let $x,y\in\mathcal{X}$ be such that $x P_{\mathcal{A}^{n\ast}} y$. There exists a sequence $(E_i)_{1\leq i\leq k}$ such that: $x\in E_1$, $y\in E_k$, $E_i\in \mathcal{A}^{n\ast}$ for $1\leq i\leq k$, and $E_i\cap E_{i+1}\neq \o$ for $1\leq i<k$. Now for every $1\leq i\leq k$, we can apply Lemma \ref{lemsublem} to get a set $F_i\in \mathcal{P}(\mathcal{A})^{n\ast}$ such that $E_i\subset F_i$. Clearly, we have $x\in F_1$, $y\in F_k$, $F_i\in \mathcal{P}(\mathcal{A})^{n\ast}$ for $1\leq i\leq k$, and $F_i\cap F_{i+1}\neq \o$ for $1\leq i<k$. Thus, $x P_{\mathcal{P}(\mathcal{A})^{n\ast}} y$.

We conclude that for every $x,y\in\mathcal{X}$, $x P_{\mathcal{P}(\mathcal{A})^{n\ast}} y$ if and only if $x P_{\mathcal{A}^{n\ast}} y$. Therefore, $\mathcal{P}\big(\mathcal{P}(\mathcal{A})^{n\ast}\big)=\mathcal{P}(\mathcal{A}^{n\ast})$.
\end{proof}

\begin{mylem}
\label{stasta1}
Let $\ast$ be an ergodic operation on a set $\mathcal{X}$. If $\mathcal{A}$ is a periodic $\mathcal{X}$-cover, then $\mathcal{P}(\mathcal{A})$ is a stable partition.
\end{mylem}
\begin{proof}
Let $n=\per(\mathcal{A})\cdot\con(\ast)$. Since $\per(\mathcal{A})$ divides $n$, we have $\mathcal{A}^{n\ast}=\mathcal{A}$. Let $A\in\mathcal{P}(\mathcal{A})$ be chosen so that $|A|$ is maximal, and let $B\in\mathcal{P}(\mathcal{A})$. We clearly have $|B|\leq |A|$. We also have $B\in \mathcal{P}(\mathcal{A}^{n\ast})$ since $\mathcal{A}^{n\ast}=\mathcal{A}$. From Lemma \ref{lemindkif} we have $\mathcal{P}\big(\mathcal{P}(\mathcal{A})^{n\ast}\big)=\mathcal{P}(\mathcal{A}^{n\ast})$, and so $B\in \mathcal{P}\big(\mathcal{P}(\mathcal{A})^{n\ast}\big)=\mathcal{P}(\mathcal{A}^{n\ast})=\mathcal{P}(\mathcal{A})$.

Fix $x\in A$ and $y\in B$. Since $n\geq \con(\ast)$, there exists a sequence $x_0,\ldots,x_{n-1}\in\mathcal{X}$ such that $y=(\ldots((x\ast x_0)\ast x_1)\ldots\ast x_{n-1})$. Now choose $X_0,\ldots,X_{n-1}$ such that $x_i\in X_i\in \mathcal{P}(\mathcal{A})^{i\ast}$ for $0\leq i<n$. Define $C:=(\ldots((A\ast X_0)\ast X_1)\ldots\ast X_{n-1})$. Clearly, $y\in C\in \mathcal{P}(\mathcal{A})^{n\ast}$. Now since $y\in B\in\mathcal{P}\big(\mathcal{P}(\mathcal{A})^{n\ast}\big)$ and $y\in C\in \mathcal{P}(\mathcal{A})^{n\ast}$, we must have $C=(\ldots((A\ast X_0)\ast X_1)\ldots\ast X_{n-1})\subset B$ and so $|A|\leq|C|\leq |B|$, which implies that $|A|=|B|=|C|$ since we already have $|B|\leq |A|$. Therefore, $C=B$ and so $B\in\mathcal{P}(\mathcal{A})^{n\ast}$ for every $B\in\mathcal{P}(\mathcal{A})$, from which we conclude that $\mathcal{P}(\mathcal{A})\subset \mathcal{P}(\mathcal{A})^{n\ast}$. On the other hand, since $|A|=|B|$ for every $B\in\mathcal{P}(\mathcal{A})$, $\mathcal{P}(\mathcal{A})$ is a balanced partition.

Now for every $C\in\mathcal{P}(\mathcal{A})^{n\ast}$, there exists a set $D\in\mathcal{P}(\mathcal{A})$ and a sequence $X_0,\ldots,X_{n-1}$ such that $X_i\in \mathcal{P}(\mathcal{A})^{i\ast}$ and $C=(\ldots((D\ast X_0)\ast X_1)\ldots\ast X_{n-1})$. We have $|D|\leq |C|$. On the other hand, Lemma \ref{lemsublem} (applied to the $\mathcal{X}$-cover $\mathcal{P}(\mathcal{A})^{n\ast}$) implies the existence of a set $B\in\mathcal{P}\big(\mathcal{P}(\mathcal{A})^{n\ast}\big)$ such that $C\subset B$. Therefore, $|D|\leq |C|\leq |B|$. Now since $\mathcal{P}\big(\mathcal{P}(\mathcal{A})^{n\ast}\big)=\mathcal{P}(\mathcal{A}^{n\ast})$ (by Lemma \ref{lemindkif}) and $\mathcal{A}^{n\ast}=\mathcal{A}$, we have $B\in \mathcal{P}\big(\mathcal{P}(\mathcal{A})^{n\ast}\big)=\mathcal{P}(\mathcal{A}^{n\ast})=\mathcal{P}(\mathcal{A})$. Therefore, $|D|=|B|$ since $D,B\in\mathcal{P}(\mathcal{A})$ and since  $\mathcal{P}(\mathcal{A})$ was shown to be a balanced partition. Thus, $|B|=|C|=|D|$ which implies that $C=B\in\mathcal{P}(A)$ since $C\subset B$. We conclude that $C\in \mathcal{P}(A)$ for every $C\in \mathcal{P}(\mathcal{A})^{n\ast}$. Therefore, $\mathcal{P}(\mathcal{A})^{n\ast}\subset\mathcal{P}(\mathcal{A})$. This means that $\mathcal{P}(\mathcal{A})^{n\ast}=\mathcal{P}(\mathcal{A})$ since we already have $\mathcal{P}(\mathcal{A})\subset \mathcal{P}(\mathcal{A})^{n\ast}$. We conclude that $\mathcal{P}(\mathcal{A})$ is a stable partition.
\end{proof}

\begin{mylem}
\label{stastaerg1}
Let $\ast$ be an ergodic operation on a set $\mathcal{X}$. If $\mathcal{A}$ is a periodic $\mathcal{X}$-cover, then $\mathcal{A}$ is stable. Moreover, for every $i\geq 0$, every $A\in\mathcal{A}$ and every $B\in\mathcal{A}^{i\ast}$, we have $|A|=|B|$.
\end{mylem}
\begin{proof}
The exact same proof of Lemma \ref{stasta} can be applied here to show the lemma.
\end{proof}

\begin{mylem}
\label{stastaerg2}
Let $\ast$ be an ergodic operation on a set $\mathcal{X}$, and let $\mathcal{A}$ be a periodic $\mathcal{X}$-cover. For every $A,B,C\in\mathcal{A}$, if $B\cap C\neq \o$ then $A\ast B=A\ast C$.
\end{mylem}
\begin{proof}
We have $A\ast B\in\mathcal{A}^\ast$, and from Lemma \ref{stastaerg1} we get $|A\ast B|=|A|$. On the other hand, since $\ast$ is uniformity preserving, we have $|A\ast x|=|A|$ for every $x\in \mathcal{X}$. Now since $\displaystyle A\ast B=\bigcup_{b\in B} A\ast b$, and since $|A\ast b|=|A|=|A\ast B|$ for every $b\in B$, we must have $A\ast B=A\ast b$ for every $b\in B$. Similarly, $A\ast C= A\ast c$ for every $c\in C$. We conclude that $A\ast B=A\ast C$ since $B\cap C\neq \o$ (take any $x\in B\cap C$, we have $A\ast B=A\ast x=A\ast C$).
\end{proof}

\begin{mylem}
\label{stastaerg3}
Let $\ast$ be an ergodic operation on a set $\mathcal{X}$, and let $\mathcal{A}$ be a periodic $\mathcal{X}$-cover. For every $A\in\mathcal{A}$ and every $B\in\mathcal{P}(\mathcal{A})$, we have $A\ast B\in\mathcal{A}^{\ast}$.
\end{mylem}
\begin{proof}
According to Lemma \ref{lemlalababa} there exists a finite sequence $(A_i)_{1\leq i\leq l}$ such that $\displaystyle B=\bigcup_{i=1}^l A_i$, $A_i\in\mathcal{A}$ for all $1\leq i\leq l$, and $A_i\cap A_{i+1}\neq \o$ for all $1\leq i<l$. Lemma \ref{stastaerg2} shows that $A\ast A_1=A\ast A_2=\ldots=A\ast A_l$. Therefore, $A\ast B=A\ast A_1\in\mathcal{A}^{\ast}$.
\end{proof}

\begin{mylem}
\label{stastaerg4}
Let $\ast$ be an ergodic operation on a set $\mathcal{X}$, and let $\mathcal{A}$ be a periodic $\mathcal{X}$-cover. For every $A\in\mathcal{A}$ and every $\mathcal{P}(\mathcal{A})$-sequence $\mathfrak{X}$, we have $A\ast\mathfrak{X}\in\mathcal{A}^{|\mathfrak{X}|\ast}$.
\end{mylem}
\begin{proof}
We will prove the lemma by induction on $k=|\mathfrak{X}|>0$. Lemma \ref{stastaerg3} implies that the statement is true for $k=1$. Now let $k>1$ and suppose that the lemma is true for $|\mathfrak{X}|=k-1$. Now let $\mathfrak{X}=(X_i)_{0\leq i<k}$ be a $\mathcal{P}(\mathcal{A})$-sequence of length $k$. Define $\mathfrak{X}'=(X_i)_{0\leq i<k-1}$. We have:
\begin{itemize}
\item $A'=A\ast\mathfrak{X}'\in \mathcal{A}^{(k-1)\ast}$ from the induction hypothesis.
\item Lemma \ref{stasta1} shows that $\mathcal{P}(\mathcal{A})$ is a stable partition, and so $\mathcal{P}(\mathcal{A})^{(k-1)\ast}$ is also a stable partition. In particular, $\mathcal{P}(\mathcal{A})^{(k-1)\ast}$ is a partition and so $\mathcal{P}(\mathcal{A})^{(k-1)\ast}=\mathcal{P}\big(\mathcal{P}(\mathcal{A})^{(k-1)\ast}\big)$. On the other hand, Lemma \ref{lemindkif} shows that $\mathcal{P}\big(\mathcal{P}(\mathcal{A})^{(k-1)\ast}\big)=\mathcal{P}(\mathcal{A}^{(k-1)\ast})$. Therefore, $\mathcal{P}(\mathcal{A})^{(k-1)\ast}=\mathcal{P}(\mathcal{A}^{(k-1)\ast})$. We conclude that $X_{k-1}\in \mathcal{P}(\mathcal{A}^{(k-1)\ast})$ since we have $X_{k-1}\in\mathcal{P}(\mathcal{A})^{(k-1)\ast}$.
\item Since $(\mathcal{A}^{(k-1)\ast})^{n\ast}=(\mathcal{A}^{n\ast})^{(k-1)\ast}= \mathcal{A}^{(k-1)\ast}$ (where $n=\per(\mathcal{A})$), $\mathcal{A}^{(k-1)\ast}$ is a periodic $\mathcal{X}$-cover.
\end{itemize}
Now since $A'\in \mathcal{A}^{(k-1)\ast}$ and $X_{k-1}\in \mathcal{P}(\mathcal{A}^{(k-1)\ast})$, and since $\mathcal{A}^{(k-1)\ast}$ is a periodic $\mathcal{X}$-cover, we can apply Lemma \ref{stastaerg3} to obtain $A'\ast X_{k-1}\in(\mathcal{A}^{(k-1)\ast})^\ast=\mathcal{A}^{k\ast}$. We conclude that $A\ast\mathfrak{X}=A'\ast X_{k-1}\in\mathcal{A}^{k\ast}$ which completes the induction argument.
\end{proof}

\vspace*{5mm}

Now we are ready to prove Proposition \ref{stastaerg}:

\begin{proof}[Proof of Proposition \ref{stastaerg}]
Let $\mathcal{A}$ be a periodic $\mathcal{X}$-cover. Lemma \ref{stasta1} shows that $\mathcal{P}(\mathcal{A})$ is a stable partition. Let $n=\per(\mathcal{A}).\scon(\ast)$. We have the following:
\begin{itemize}
\item $\mathcal{P}(\mathcal{A})^{n\ast}=\mathcal{P}\big( \mathcal{P}(\mathcal{A})^{n\ast}\big)$ since $\mathcal{P}(\mathcal{A})$ is a stable partition.
\item $\mathcal{P}\big(\mathcal{P}(\mathcal{A})^{n\ast}\big)=\mathcal{P}(\mathcal{A}^{n\ast})$ by Lemma \ref{lemindkif}.
\item $\mathcal{A}^{n\ast}=\mathcal{A}$ since $\per(\mathcal{A})$ divides $n$.
\end{itemize}
Therefore, $\mathcal{P}(\mathcal{A})^{n\ast}=\mathcal{P}(\mathcal{A}^{n\ast})=\mathcal{P}(\mathcal{A})$.

Fix $A\in \mathcal{A}$. From Lemma \ref{lemsublem} there exists $B\in\mathcal{P}(\mathcal{A})$ such that $A\subset B$. Fix $a\in A$. Since $a\in B\in\mathcal{P}(\mathcal{A})=\mathcal{P}(\mathcal{A})^{n\ast}$ and since $n\geq \scon(\ast)$, we can apply Theorem \ref{thestrong} to get a $\mathcal{P}(\mathcal{A})$-sequence of length $n$ such that $a\ast \mathfrak{X}=B\ast\mathfrak{X}=B$. Since $B=a\ast\mathfrak{X}\subset A\ast\mathfrak{X}\subset B\ast \mathfrak{X}=B$, we have $A\ast\mathfrak{X}=B$. Now from Lemma \ref{stastaerg4}, we have $A\ast\mathfrak{X}\in\mathcal{A}^{n\ast}=\mathcal{A}$, and Lemma \ref{stastaerg1} implies that $|A|=|A\ast\mathfrak{X}|=|B|$. Thus, $A=B$ since we have  $A\subset B$ and $|A|=|B|$.

We conclude that $A\in\mathcal{P}(\mathcal{A})$ for every $A\in\mathcal{A}$. Now since $\mathcal{P}(\mathcal{A})$ is a partition, we have $A\cap B=\o$ for every $A,B\in\mathcal{A}$. On the hand, $\mathcal{A}$ is an $\mathcal{X}$-cover. This shows that $\mathcal{A}$ itself is a partition, hence $\mathcal{A}=\mathcal{P}(\mathcal{A})$. Therefore, $\mathcal{A}$ is a stable partition.
\end{proof}

\section{Proofs for Section \ref{sec8}}

\label{appE}

\subsection{Proof of Theorem \ref{theprod}}

In order to prove Theorem \ref{theprod}, we need a few definitions and lemmas:

\begin{mydef}
\label{defProjGG}
Define the two projection mappings $P_1: \mathcal{X}\rightarrow\mathcal{X}_1$ and $P_2: \mathcal{X}\rightarrow\mathcal{X}_2$ as $P_1(x_1,x_2)=x_1$ and $P_2(x_1,x_2)=x_2$ for all $(x_1,x_2)\in\mathcal{X}$. Define the following:
\begin{itemize}
\item $\mathcal{U}_1(\mathcal{H})=\{P_1(H):\; H\in \mathcal{H}\}$.
\item $\mathcal{U}_2(\mathcal{H})=\{P_2(H):\; H\in \mathcal{H}\}$.
\end{itemize}
\end{mydef}

\begin{mylem}
\label{lemseqoaoa}
For every $x_2,x_2'\in\mathcal{X}_2$, there exists an $\mathcal{H}$-repeatable sequence $\mathfrak{X}$ such that:
\begin{itemize}
\item For every $x_1\in\mathcal{X}_1$, we have $(x_1,x_2')\in (x_1,x_2)\ast\mathfrak{X}$.
\item For every $X\subset \mathcal{X}$, we have $P_1(X)\subset P_1(X\ast\mathfrak{X})$.
\end{itemize}

We say that the sequence $\mathfrak{X}$ can take the second coordinate from $x_2$ to $x_2'$ while keeping the first coordinate unchanged.
\end{mylem}
\begin{proof}
Let $k=\per(\mathcal{H})\con(\ast_2)\geq\con(\ast_2)$. Choose arbitrarily a sequence of $k$ elements $x_{1,0},\ldots,x_{1,k-1}$ in $\mathcal{X}_1$ and define the mapping $\pi:\mathcal{X}_1\rightarrow\mathcal{X}_1$ as $\pi(x_1)=(\ldots((x_1\ast_1 x_{1,0})\ast_1 x_{1,1})\ldots\ast_1 x_{1,k-1})$. Since $\pi$ is a permutation of $\mathcal{X}_1$, there exists an integer $s>0$ such that $\pi^s(x_1)=x_1$ for all $x_1\in\mathcal{X}_1$. Let $l=ks$ and define the sequence $x_{1,i}$ for $k\leq i<l$ as $x_{1,i}=x_{1,i\bmod k}$. Clearly, 
\begin{equation}
(\ldots((x_1\ast_1 x_{1,0})\ast_1 x_{1,1})\ldots\ast_1 x_{1,l-1})=\pi^{s}(x_1)=x_1\;\text{for\;all}\;x_1\in\mathcal{X}_1.
\label{eq9247297asfnajsb1}
\end{equation}

Now since $l\geq k\geq \con(\ast_2)$ and since $\ast_2$ is ergodic, there exists a sequence $(x_{2,i})_{0\leq i<l}$ in $\mathcal{X}_2$ such that
\begin{equation}
x_2'=(\ldots((x_2\ast_2 x_{2,0})\ast_2 x_{2,1})\ldots\ast_2 x_{2,l-1}).
\label{eq9247297asfnajsb2}
\end{equation}

Define the $\mathcal{H}$-repeatable sequence $\mathfrak{X}=(X_i)_{0\leq i<l}$ such that $(x_{1,i},x_{2,i})\in X_i\in\mathcal{H}^{i\ast}$ for all $0\leq i<l$. For every $x_1\in\mathcal{X}_1$, we have:
$$(x_1,x_2')\stackrel{(a)}{=}(x_1,x_2)\ast\Big((x_{1,i},x_{2,i})_{0\leq i<l}\Big)\stackrel{(b)}{\in} (x_1,x_2)\ast\mathfrak{X},$$ where (a) follows from \eqref{eq9247297asfnajsb1} and \eqref{eq9247297asfnajsb2}, and (b) follows from the fact that $(x_{1,i},x_{2,i})\in X_i$ for all $0\leq i<l$.

Now let $X\subset\mathcal{X}$. We have:
$$P_1(X)\stackrel{(a)}{=}(\ldots((P_1(X)\ast_1 x_{1,0})\ast_1 x_{1,1})\ldots\ast_1 x_{1,l-1})\stackrel{(b)}{=}P_1(X\ast (x_{1,i},x_{2,i})_{0\leq i<l})\stackrel{(c)}{\subset} P_1(X\ast\mathfrak{X}),$$
where (a) follows from \eqref{eq9247297asfnajsb1}, (b) follows from the definition of $\ast$ and $P_1$, and (c) follows from the fact that $(x_{1,i},x_{2,i})\in X_i$ for all $0\leq i<l$.
\end{proof}

\begin{mylem}
Let $\mathfrak{X}$ be an $\mathcal{H}$-repeatable sequence which takes the second coordinate from $x_2$ to $x_2'$ while keeping the first coordinate unchanged as in Lemma \ref{lemseqoaoa}. If there exist $H,H'\in\mathcal{H}$ and $x_1\in\mathcal{X}_1$ such that $(x_1,x_2)\in H$ and $(x_1,x_2')\in H'$, then $H'=H\ast\mathfrak{X}$.
\label{lemseqoaoa1}
\end{mylem}
\begin{proof}
From Lemma \ref{lemseqoaoa} we have $(x_1,x_2')\in (x_1,x_2)\ast \mathfrak{X}\subset H\ast\mathfrak{X}$. Therefore, $H'\cap(H\ast\mathfrak{X})\neq\o$. On the other hand, we have $H'\in\mathcal{H}$ and $H\ast\mathfrak{X}\in\mathcal{H}^{|\mathfrak{X}|\ast}=\mathcal{H}$. Therefore, $H'=H\ast\mathfrak{X}$ since $\mathcal{H}$ is a partition.
\end{proof}

\begin{mylem}
$\mathcal{U}_1(\mathcal{H})$ (resp. $\mathcal{U}_2(\mathcal{H})$) is a partition of $\mathcal{X}_1$ (resp. $\mathcal{X}_2$).
\label{lemfdjfgh8345sfdds1}
\end{mylem}
\begin{proof}
Clearly, $\mathcal{U}_1(\mathcal{H})$ covers $\mathcal{X}_1$. Now suppose that there exist $A,B\in\mathcal{U}_1(\mathcal{H})$ such that $A\cap B\neq \o$ and let $x_1\in A\cap B$. Let $H_A,H_B\in\mathcal{H}$ be such that $P_1(H_A)=A$ and $P_1(H_B)=B$. There exist $x_{2,A}\in\mathcal{X}_2$ and $x_{2,B}\in\mathcal{X}_2$ such that $(x_1,x_{2,A})\in H_A$ and $(x_1,x_{2,B})\in H_B$. Using Lemma \ref{lemseqoaoa}, choose an $\mathcal{H}$-repeatable sequence $\mathfrak{X}$ which can take the second coordinate from $x_{2,A}$ to $x_{2,B}$ while keeping the first coordinate unchanged.

Lemma \ref{lemseqoaoa1} shows that $H_B=H_A\ast\mathfrak{X}$ and Lemma \ref{lemseqoaoa} implies that $P_1(H_A)\subset P_1(H_A\ast\mathfrak{X})$. We conclude that $A=P_1(H_A)\subset P_1(H_A\ast\mathfrak{X}) = P_1(H_B)=B$. By exchanging the roles of $A$ and $B$, we can also get $B\subset A$. Therefore, $A=B$. We conclude that $\mathcal{U}_1(\mathcal{H})$ is a partition of $\mathcal{X}_1$. A similar argument shows that $\mathcal{U}_2(\mathcal{H})$ is a partition of $\mathcal{X}_2$.
\end{proof}

\begin{mylem}
$\mathcal{U}_1(\mathcal{H})$ (resp. $\mathcal{U}_2(\mathcal{H})$) is a stable partition of $\mathcal{X}_1$ (resp. $\mathcal{X}_2$) of period at most $\per(\mathcal{H})$. Moreover, for every $i\geq 0$, we have $\mathcal{U}_1(\mathcal{H})^{i\ast_1}=\mathcal{U}_1(\mathcal{H}^{i\ast})$ and $\mathcal{U}_2(\mathcal{H})^{i\ast_2}=\mathcal{U}_2(\mathcal{H}^{i\ast})$.
\label{lemProdUStable}
\end{mylem}
\begin{proof}
We will only prove the lemma for $\mathcal{U}_1(\mathcal{H})$ since the proof for $\mathcal{U}_2(\mathcal{H})$ is similar. We will start by showing by induction on $i\geq 0$ that $\mathcal{U}_1(\mathcal{H})^{i\ast_1}=\mathcal{U}_1(\mathcal{H}^{i\ast})$. The claim is trivial for $i=0$. Now let $i>0$ and suppose that the claim is true for $i-1$. We have:
\begin{align*}
\mathcal{U}_1(\mathcal{H})^{i\ast_1}&= \big(\mathcal{U}_1(\mathcal{H})^{(i-1)\ast_1} \big)^{\ast_1}\stackrel{(a)}{=}\big(\mathcal{U}_1(\mathcal{H}^{(i-1)\ast}) \big)^{\ast_1}=\{H_1'\ast_1 H_1'':\; H_1',H_1''\in \mathcal{U}_1(\mathcal{H}^{(i-1)\ast})\}\\
&=\{P_1(H')\ast_1 P_1(H''):\; H',H''\in \mathcal{H}^{(i-1)\ast}\}\stackrel{(b)}{=}\{P_1(H'\ast H''):\; H',H''\in \mathcal{H}^{(i-1)\ast}\}\\
&=\{P_1(H):\; H\in \mathcal{H}^{i\ast}\}=\mathcal{U}_1(\mathcal{H}^{i\ast}),
\end{align*}
where (a) follows from the induction hypothesis and (b) follows from the identity $P_1(H')\ast_1 P_1(H'')=P_1(H'\ast H'')$ which is very easy to check. We conclude that we have $\mathcal{U}_1(\mathcal{H})^{i\ast_1}=\mathcal{U}_1(\mathcal{H}^{i\ast})$ for all $i\geq 0$. In particular, for $p=\per(\mathcal{H})$, we have $\mathcal{U}_1(\mathcal{H})^{p\ast_1}=\mathcal{U}_1(\mathcal{H}^{p\ast})= \mathcal{U}_1(\mathcal{H})$.

Lemma \ref{lemfdjfgh8345sfdds1} shows that $\mathcal{U}_1(\mathcal{H})$ is a partition, and we have just shown that $\mathcal{U}_1(\mathcal{H})^{p\ast_1}= \mathcal{U}_1(\mathcal{H})$. Therefore, $\mathcal{U}_1(\mathcal{H})$ is periodic  of period at most $p$. Lemma \ref{stasta} now implies that $\mathcal{U}_1(\mathcal{H})$ is a stable partition of $\mathcal{X}_1$.
\end{proof}

\begin{mydef}
\label{defProjEE}
Let $X\subset\mathcal{X}$, $x_1\in\mathcal{X}_1$ and $x_2\in\mathcal{X}_2$. Define the sets $P_{1|x_2}(X)\subset \mathcal{X}_1$ and $P_{2|x_1}(X)\subset \mathcal{X}_2$ as:
\begin{itemize}
\item $P_{1|x_2}(X)=\{x_1\in\mathcal{X}_1:\;(x_1,x_2)\in X\}=P_1\big(X\cap (\mathcal{X}_1\times\{x_2\})\big)$.
\item $P_{2|x_1}(X)=\{x_2\in\mathcal{X}_2:\;(x_1,x_2)\in X\}=P_2\big(X\cap (\{x_1\}\times \mathcal{X}_2)\big)$.
\end{itemize}
Define the following:
\begin{itemize}
\item $\mathcal{L}_1(\mathcal{H})=\{P_{1|x_2}(H):\; H\in \mathcal{H},\; x_2\in\mathcal{X}_2,\;P_{1|x_2}(H)\neq \o\}$.
\item $\mathcal{L}_2(\mathcal{H})=\{P_{2|x_1}(H):\; H\in \mathcal{H},\; x_1\in\mathcal{X}_1,\;P_{2|x_1}(H)\neq \o\}$.
\end{itemize}
\end{mydef}

\begin{mylem}
$\mathcal{L}_1(\mathcal{H})$ (resp. $\mathcal{L}_2(\mathcal{H})$) is a partition of $\mathcal{X}_1$ (resp. $\mathcal{X}_2$).
\label{lemvdbs7sdf7ds7a1}
\end{mylem}
\begin{proof}
Clearly, $\mathcal{L}_1(\mathcal{H})$ covers $\mathcal{X}_1$. Suppose that there exist $A,B\in\mathcal{L}_1(\mathcal{H})$ such that $A\cap B\neq\o$ and let $x_1\in A\cap B$. Let $H_A,H_B\in\mathcal{H}$ and $x_{2,A},x_{2,B}\in\mathcal{X}_2$ be such that $A=P_{1|x_{2,A}}(H_A)$ and $B=P_{2|x_{2,B}}(H_B)$. Using Lemma \ref{lemseqoaoa}, choose an $\mathcal{H}$-repeatable sequence $\mathfrak{X}$ which can take the second coordinate from $x_{2,A}$ to $x_{2,B}$ while keeping the first coordinate unchanged.

Since $x_1\in A=P_{1|x_{2,A}}(H_A)$ and $x_1\in B=P_{1|x_{2,B}}(H_B)$, we have $(x_1,x_{2,A})\in H_A$ and $(x_1,x_{2,B})\in H_B$. It follows from Lemma \ref{lemseqoaoa1} that $H_B=H_A\ast\mathfrak{X}$.

Now for every $x_1'\in A=P_{1|x_{2,A}}(H_A)$, we have $(x_1',x_{2,A})\in H_A$ and so by Lemma \ref{lemseqoaoa} we have $(x_1',x_{2,B})\in (x_1',x_{2,A})\ast\mathfrak{X}\subset H_A\ast\mathfrak{X}=H_B$. We conclude that $x_1'\in P_{1|x_{2,B}}(H_B)=B$ for every $x_1'\in A$. Therefore, $A\subset B$. By exchanging the roles of $A$ and $B$ we can also get $B\subset A$ which implies that $A=B$. We conclude that $\mathcal{L}_1(\mathcal{H})$ is a partition of $\mathcal{X}_1$. A similar argument shows that $\mathcal{L}_2(\mathcal{H})$ is a partition of $\mathcal{X}_2$.
\end{proof}

\begin{mylem}
$\mathcal{L}_1(\mathcal{H})$ (resp. $\mathcal{L}_2(\mathcal{H})$) is a balanced partition of $\mathcal{X}_1$ (resp. $\mathcal{X}_2$).
\label{lemvdbs7sdf7ds7a2}
\end{mylem}
\begin{proof}
Let $A,B\in\mathcal{L}_1(\mathcal{H})$. There exist $H_A,H_B\in\mathcal{H}$ and $x_{2,A},x_{2,B}\in\mathcal{X}_2$ such that $A=P_{1|x_{2,A}}(H_A)$ and $B=P_{2|x_{2,B}}(H_B)$. Fix $x_{1,A}\in A$ and $x_{1,B}\in B$ and define $k=\per(\mathcal{H})\cdot\max\{\con(\ast_1),\con(\ast_2)\}$. Clearly, $(x_{1,A},x_{2,A})\in H_A$ and $(x_{1,B},x_{2,B})\in H_B$.

Since $k\geq\con(\ast_1)$ and $k\geq\con(\ast_2)$, and since $\ast_1$ and $\ast_2$ are ergodic, there exist a sequence $(x_{1,i})_{0\leq i<k}$ in $\mathcal{X}_1$ and a sequence $(x_{2,i})_{0\leq i<k}$ in $\mathcal{X}_2$ such that:
\begin{equation}
\label{eqleqteqdas4q1}
\begin{aligned}
(\ldots((x_{1,A}\ast_1 x_{1,0})\ast_1 x_{1,1})\ldots\ast_1 x_{1,k-1})&=x_{1,B},\\
(\ldots((x_{2,A}\ast_2 x_{2,0})\ast_2 x_{2,1})\ldots\ast_2 x_{2,k-1})&=x_{2,B}.
\end{aligned}
\end{equation}
Now define the $\mathcal{H}$-repeatable sequence $\mathfrak{X}=(X_i)_{0\leq i<k}$ such that $(x_{1,i},x_{2,i})\in X_i\in\mathcal{H}^{i\ast}$ for all $0\leq i<k$. We have:
$$(x_{1,B},x_{2,B})\stackrel{(a)}{=}(x_{1,A},x_{2,A})\ast \Big((x_{1,i},x_{2,i})_{0\leq i<k}\Big)\stackrel{(b)}{\in} H_A\ast\mathfrak{X},$$
where (a) follows from \eqref{eqleqteqdas4q1} and (b) follows from the fact that $(x_{1,A},x_{2,A})\in H_A$ and $(x_{1,i},x_{2,i})\in X_i$ for every $0\leq i<k$. We conclude that $H_B\cap (H_A\ast\mathfrak{X})\neq \o$. On the other hand, we have $H_B\in\mathcal{H}$ and $H_A\ast\mathfrak{X}\in \mathcal{H}^{k\ast}=\mathcal{H}$. Therefore, $H_B=H_A\ast\mathfrak{X}$ since $\mathcal{H}$ is a partition.

Define the mapping $\pi_1:\mathcal{X}_1\rightarrow\mathcal{X}_1$ as $\pi_1(x_1)=(\ldots((x_1\ast_1 x_{1,0})\ast_1 x_{1,1})\ldots\ast_1 x_{1,k-1})$ for every $x_1\in\mathcal{X}_1$ and the mapping $\pi_2:\mathcal{X}_2\rightarrow\mathcal{X}_2$ as $\pi_2(x_2)=(\ldots((x_2\ast_2 x_{2,0})\ast_2 x_{2,1})\ldots\ast_2 x_{2,k-1})$ for every $x_2\in\mathcal{X}_2$.

Now let $x_1\in A=P_{1|x_{2,A}}(H_A)$, we have:
$$(\pi_1(x_1),x_{2,B})\stackrel{(a)}{=}(\pi_1(x_1),\pi_2(x_{2,A}))\stackrel{(b)}{=}(x_1,x_{2,A})\ast \Big((x_{1,i},x_{2,i})_{0\leq i<k}\Big)\stackrel{(c)}{\in} H_A\ast\mathfrak{X}=H_B,$$ where (a) follows from \eqref{eqleqteqdas4q1}, (b) follows from the definition of $\pi_1$ and $\pi_2$ and (c) follows from the fact that $(x_1,x_{2,A})\in H_A$ and $(x_{1,i},x_{2,i})\in X_i$ for every $0\leq i<k$.

We conclude that $\pi_1(x_1)\in P_{1|x_{2,B}}(H_B)=B$ for every $x_1\in A$. Therefore, $\pi_1(A)\subset B$, which implies that $|A|\stackrel{(a)}{=}|\pi_1(A)|\leq |B|$, where (a) follows from the fact that $\pi_1$ is a permutation. By exchanging the roles of $A$ and $B$ we can also get $|B|\leq |A|$ which implies that $|A|=|B|$. We conclude that $\mathcal{L}_1(\mathcal{H})$ is a balanced partition of $\mathcal{X}_1$ as Lemma \ref{lemvdbs7sdf7ds7a1} already showed that $\mathcal{L}_1(\mathcal{H})$ is a partition. A similar argument shows that $\mathcal{L}_2(\mathcal{H})$ is a balanced partition of $\mathcal{X}_2$.
\end{proof}

\begin{mylem}
For every $i\geq 0$ and every $A\in \mathcal{L}_1(\mathcal{H})^{i\ast_1}$, there exists $B\in \mathcal{L}_1(\mathcal{H}^{i\ast})$ such that $A\subset B$.
\label{lemvdbs7sdf7ds7a3}
\end{mylem}
\begin{proof}
We will prove the lemma by induction on $i\geq 0$. The lemma is trivial for $i=0$.

Now let $i>0$ and suppose that the lemma is true for $i-1$. Let $A\in \mathcal{L}_1(\mathcal{H})^{i\ast_1}$, there exist $A',A''\in \mathcal{L}_1(\mathcal{H})^{(i-1)\ast_1}$ such that $A=A'\ast_1 A''$. From the induction hypothesis, there exist $B',B''\in \mathcal{L}_1(\mathcal{H}^{(i-1)\ast})$ such that $A'\subset B'$ and $A''\subset B''$. This means that there exist $H',H''\in \mathcal{H}^{(i-1)\ast}$ and $x_2',x_2''\in\mathcal{X}_2$ such that $B'=P_{1|x_2'}(H')$ and $B''=P_{1|x_2''}(H'')$. We have:
$$A=A'\ast_1 A''\subset B'\ast_1 B''=P_{1|x_2'}(H')\ast_1 P_{1|x_2''}(H'') \stackrel{(a)}{\subset} P_{1|x_2'\ast_2 x_2''}(H'\ast H''),$$
where (a) follows from the fact that for every $x_1'\in P_{1|x_2'}(H')$ and $x_1''\in P_{1|x_2''}(H'')$, we have $(x_1',x_2')\in H'$ and $(x_1'',x_2'')\in H''$, and so $(x_1'\ast_1 x_1'',x_2'\ast_2 x_2'')=(x_1',x_2')\ast (x_1'',x_2'')\in H'\ast H''$, which implies that $x_1'\ast_1 x_1''\in P_{1|x_2'\ast_2 x_2''}(H'\ast H'')$.

If we define $B=P_{1|x_2'\ast_2 x_2''}(H'\ast H'')\in\mathcal{L}_1(\mathcal{H}^{i\ast})$, we get $A\subset B$. We conclude that the lemma is true for all $i\geq 0$.
\end{proof}

\begin{mylem}
$\mathcal{L}_1(\mathcal{H})$ (resp. $\mathcal{L}_2(\mathcal{H})$) is a stable partition of $\mathcal{X}_1$ (resp. $\mathcal{X}_2$) of period at most $\per(\mathcal{H})$. Moreover, for every $i\geq 0$, we have $\mathcal{L}_1(\mathcal{H})^{i\ast_1}=\mathcal{L}_1(\mathcal{H}^{i\ast})$ and $\mathcal{L}_2(\mathcal{H})^{i\ast_2}=\mathcal{L}_2(\mathcal{H}^{i\ast})$.
\label{lemProdLStable}
\end{mylem}
\begin{proof}
We will only prove the lemma for $\mathcal{L}_1(\mathcal{H})$ since the proof for $\mathcal{L}_2(\mathcal{H})$ is similar.

Let $p=\per(\mathcal{H})$. According to Lemma \ref{lemvdbs7sdf7ds7a3}, for every $A\in \mathcal{L}_1(\mathcal{H})^{p\ast_1}$, there exists $B\in \mathcal{L}_1(\mathcal{H}^{p\ast})=\mathcal{L}_1(\mathcal{H})$ such that $A\subset B$. On the other hand, we have:
$$|A|\geq \|\mathcal{L}_1(\mathcal{H})^{p\ast_1}\|_{\wedge}\stackrel{(a)}{\geq} \|\mathcal{L}_1(\mathcal{H})\|_{\wedge}\stackrel{(b)}{=}\|\mathcal{L}_1(\mathcal{H})\|=|B|,$$
where (a) follows from Lemma \ref{lemsize} and (b) follows from the fact that $\mathcal{L}_1(\mathcal{H})$ is a balanced partition (Lemma \ref{lemvdbs7sdf7ds7a2}). We conclude that $A=B\in \mathcal{L}_1(\mathcal{H})$ since $|A|\geq |B|$ and $A\subset B$. Now since this is true for every $A\in \mathcal{L}_1(\mathcal{H})^{p\ast_1}$, we have $\mathcal{L}_1(\mathcal{H})^{p\ast_1}\subset \mathcal{L}_1(\mathcal{H})$ which implies that $\mathcal{L}_1(\mathcal{H})^{p\ast_1}=\mathcal{L}_1(\mathcal{H})$ since $\mathcal{L}_1(\mathcal{H})$ is a partition of $\mathcal{X}_1$ and $\mathcal{L}_1(\mathcal{H})^{p\ast_1}$ is an $\mathcal{X}_1$-cover. We conclude that $\mathcal{L}_1(\mathcal{H})$ is a stable partition of period at most $p=\per(\mathcal{H})$. Now since this is true for every stable partition and since $\mathcal{H}^{i\ast}$ is a stable partition for every $i\geq 0$, we conclude that $\mathcal{L}_1(\mathcal{H}^{i\ast})$ is a stable partition for every $i\geq 0$. This implies that $\mathcal{L}_1(\mathcal{H}^{i\ast})^{j\ast_1}$ is a stable partition for every $i\geq 0$ and every $j\geq 0$.

For every $i> 0$, Lemma \ref{lemvdbs7sdf7ds7a3} (applied to $\mathcal{H}^{(i-1)\ast}$) implies that $\mathcal{L}_1(\mathcal{H}^{(i-1)\ast})^{\ast_1}$ is a sub-stable partition of $\mathcal{L}_1(\mathcal{H}^{i\ast})$ and so $\|\mathcal{L}_1(\mathcal{H}^{(i-1)\ast})\|=\|\mathcal{L}_1(\mathcal{H}^{(i-1)\ast})^{\ast_1}\|\leq \|\mathcal{L}_1(\mathcal{H}^{i\ast})\|$. Therefore,
$$\|\mathcal{L}_1(\mathcal{H})\|\leq  \|\mathcal{L}_1(\mathcal{H}^\ast)\|\leq\ldots\leq \|\mathcal{L}_1(\mathcal{H}^{p\ast})\|=\|\mathcal{L}_1(\mathcal{H})\|.$$
We conclude that $\|\mathcal{L}_1(\mathcal{H}^{i\ast})\|=\|\mathcal{L}_1(\mathcal{H}^{(i\bmod p)\ast})\|=\|\mathcal{L}_1(\mathcal{H})\|$ for every $i\geq 0$. Moreover, since $\mathcal{L}_1(\mathcal{H})$ is stable, we have $\|\mathcal{L}_1(\mathcal{H})^{i\ast_1}\|=\|\mathcal{L}_1(\mathcal{H})\|$, which implies that $\|\mathcal{L}_1(\mathcal{H})^{i\ast_1}\|=\|\mathcal{L}_1(\mathcal{H}^{i\ast})\|$ for every $i\geq 0$.

Now for every $i\geq 0$, $\mathcal{L}_1(\mathcal{H})^{i\ast_1}$ is a sub-stable partition of $\mathcal{L}_1(\mathcal{H}^{i\ast})$ (by Lemma \ref{lemvdbs7sdf7ds7a3}) and we have just shown that $\|\mathcal{L}_1(\mathcal{H})^{i\ast_1}\|=\|\mathcal{L}_1(\mathcal{H}^{i\ast})\|$.  We conclude that $\mathcal{L}_1(\mathcal{H})^{i\ast_1}=\mathcal{L}_1(\mathcal{H}^{i\ast})$ for every $i\geq 0$.
\end{proof}

\vspace*{3mm}
Now we are ready to prove Theorem \ref{theprod}:

\begin{proof}[Proof of Theorem \ref{theprod}]
Lemma \ref{lemProdLStable} shows that $\mathcal{L}_1(\mathcal{H})$ and $\mathcal{L}_2(\mathcal{H})$ are stable partitions of $\mathcal{X}_1$ and $\mathcal{X}_2$ respectively, and Lemma \ref{lemProdUStable} shows that $\mathcal{U}_1(\mathcal{H})$ and $\mathcal{U}_2(\mathcal{H})$ are stable partitions of $\mathcal{X}_1$ and $\mathcal{X}_2$ respectively. Moreover, Lemma \ref{lemProdLStable} shows that $\mathcal{L}_1(\mathcal{H})^{i\ast_1}=\mathcal{L}_1(\mathcal{H}^{i\ast})$ and $\mathcal{L}_2(\mathcal{H})^{i\ast_2}=\mathcal{L}_2(\mathcal{H}^{i\ast})$ for every $i>0$, and Lemma \ref{lemProdUStable} shows that $\mathcal{U}_1(\mathcal{H})^{i\ast_1}=\mathcal{U}_1(\mathcal{H}^{i\ast})$ and $\mathcal{U}_2(\mathcal{H})^{i\ast_2}=\mathcal{U}_2(\mathcal{H}^{i\ast})$ for every $i>0$.

It is easy to see that $\mathcal{L}_1(\mathcal{H})\preceq\mathcal{U}_1(\mathcal{H})$ and $\mathcal{L}_2(\mathcal{H})\preceq\mathcal{U}_2(\mathcal{H})$. Now we turn to show that $\mathcal{L}_1(\mathcal{H})\otimes \mathcal{L}_2(\mathcal{H})\preceq \mathcal{H}\preceq\mathcal{U}_1(\mathcal{H})\otimes \mathcal{U}_2(\mathcal{H})$. Let $A\times B\in \mathcal{L}_1(\mathcal{H})\otimes \mathcal{L}_2(\mathcal{H})$ (i.e., $A\in\mathcal{L}_1(\mathcal{H})$ and $B\in\mathcal{L}_2(\mathcal{H})$), and fix $x_1\in A$ and $x_2\in B$. Let $H\in\mathcal{H}$ be such that $(x_1,x_2)\in H$. We have $x_1\in P_{1|x_2}(H)$ as $(x_1,x_2)\in H$. Therefore, $P_{1|x_2}(H)\cap A\neq\o$ which implies that $A=P_{1|x_2}(H)$ since both $A$ and $P_{1|x_2}(H)$ are in $\mathcal{L}_1(\mathcal{H})$ which was shown to be a stable partition.

Now fix $(x_A,x_B)\in A\times B$. Since $x_A\in A=P_{1|x_2}(H)$, we have $(x_A,x_2)\in H$ which means that $x_2\in P_{2|x_A}(H)$. Therefore, $B\cap P_{2|x_A}(H)\neq \o$ which implies that $B= P_{2|x_A}(H)$ since both $B$ and $P_{2|x_A}(H)$ are in $\mathcal{L}_2(\mathcal{H})$ which was shown to be a stable partition. Now since $x_B\in B= P_{2|x_A}(H)$, we conclude that $(x_A,x_B)\in H$. But this is true for all $(x_A,x_B)\in A\times B$, hence $A\times B\subset H$. Therefore, $\mathcal{L}_1(\mathcal{H})\otimes \mathcal{L}_2(\mathcal{H})\preceq \mathcal{H}$.

In order to prove that $\mathcal{H}\preceq\mathcal{U}_1(\mathcal{H})\otimes \mathcal{U}_2(\mathcal{H})$, let $H\in\mathcal{H}$, $A'=P_1(H)\in\mathcal{U}_1(\mathcal{H})$ and $B'=P_2(H)\in\mathcal{U}_2(\mathcal{H})$. Clearly, $H\subset A'\times B'$, hence $\mathcal{H}\preceq\mathcal{U}_1(\mathcal{H})\otimes \mathcal{U}_2(\mathcal{H})$.

Now let $H\in\mathcal{H}$. Since $\mathcal{L}_1(\mathcal{H})\otimes \mathcal{L}_2(\mathcal{H})\preceq \mathcal{H}$, there exist an integer $n_H>0$ and $n_H$ sets $H_1,\ldots,H_{n_H}\in \mathcal{L}_1(\mathcal{H})\otimes \mathcal{L}_2(\mathcal{H})$ such that $H_1,\ldots,H_{n_H}$ are disjoint and $H=H_1\cup\ldots\cup H_{n_H}$. Since $H_1,\ldots,H_{n_H}\in \mathcal{L}_1(\mathcal{H})\otimes \mathcal{L}_2(\mathcal{H})$, there exist $n_H$ sets $H_{1,1},\ldots,H_{1,n_H}\in\mathcal{L}_1(\mathcal{H})$ and $n_H$ sets $H_{2,1},\ldots,H_{2,n_H}\in\mathcal{L}_2(\mathcal{H})$ such that $H_1=H_{1,1}\times H_{2,1}$, \ldots, and $H_{n_H}=H_{1,n_H}\times H_{2,n_H}$. Clearly, $H_{1,i}=P_1(H_i)$ and $H_{2,i}=P_2(H_i)$ for every $1\leq i\leq n_H$. We have:
\begin{itemize}
\item $H_{1,1}\cup \ldots \cup H_{1,n_H}= P_1(H_1)\cup\ldots\cup P_1(H_{n_H})= P_1(H_1\cup\ldots\cup H_{n_H})=P_1(H)\in\mathcal{U}_1(\mathcal{H})$.
\item $H_{2,1}\cup \ldots \cup H_{2,n_H}= P_2(H_1)\cup\ldots\cup P_2(H_{n_H})= P_2(H_1\cup\ldots\cup H_{n_H})=P_2(H)\in\mathcal{U}_2(\mathcal{H})$.
\item Suppose that $H_{1,i}=H_{1,j}$ for some $i\neq j$ and let $x_1\in H_{1,i}=H_{1,j}$, then $H_{2,i}\cup H_{2,j}\subset P_{2|x_1}(H)\in\mathcal{L}_2(\mathcal{H})$ which cannot happen unless $H_{2,i}=H_{2,j}=P_{2|x_1}(H)$. This is a contradiction since $(H_{1,i}\times H_{2,i})$ and $(H_{1,j}\times H_{2,j})$ are disjoint. We conclude that $H_{1,1},\ldots,H_{1,n_H}$ are disjoint. Similarly, $H_{2,1},\ldots,H_{2,n_H}$ are also disjoint.
\end{itemize}
Now since $H_{1,1},\ldots,H_{1,n_H}$ are disjoint, we have $\|\mathcal{U}_1(\mathcal{H})\|=|P_1(H)|=|H_{1,1}|+\ldots+|H_{1,n_H}|=n_H\|\mathcal{L}_1(\mathcal{H})\|$. Therefore, $n_H=\frac{\|\mathcal{U}_1(\mathcal{H})\|}{\|\mathcal{L}_1(\mathcal{H})\|}$. Similarly,  $n_H=\frac{\|\mathcal{U}_2(\mathcal{H})\|}{\|\mathcal{L}_2(\mathcal{H})\|}$. We conclude that $n_H$ is the same for all $H\in\mathcal{H}$. Let us denote this common integer as $n$. It is now easy to see that $\|\mathcal{H}\|=n\cdot\|\mathcal{L}_1(\mathcal{H})\|\cdot \|\mathcal{L}_2(\mathcal{H})\|=\|\mathcal{L}_1(\mathcal{H})\|\cdot \|\mathcal{U}_2(\mathcal{H})\|=\|\mathcal{U}_1(\mathcal{H})\|\cdot \|\mathcal{L}_2(\mathcal{H})\|$.

Now in order to prove the uniqueness of $\mathcal{L}_1(\mathcal{H})$, $\mathcal{L}_2(\mathcal{H})$, $\mathcal{U}_1(\mathcal{H})$ and $\mathcal{U}_2(\mathcal{H})$, suppose that $\mathcal{H}_1$, $\mathcal{H}_2$, $\mathcal{H}_1'$, $\mathcal{H}_2'$, and $n'>0$ satisfy the conditions of the theorem (i.e. $\mathcal{H}_1$, $\mathcal{H}_2$, $\mathcal{H}_1'$, $\mathcal{H}_2'$ and $n'$ play the roles of $\mathcal{L}_1(\mathcal{H})$, $\mathcal{L}_2(\mathcal{H})$, $\mathcal{U}_1(\mathcal{H})$, $\mathcal{U}_2(\mathcal{H})$ and $n$ respectively). Let $H\in\mathcal{H}$, then there exist $n'$ disjoint sets $H_{1,1}',\ldots,H_{1,n'}'\in \mathcal{H}_1$ and $n'$ disjoint sets $H_{2,1}',\ldots,H_{2,n'}'\in \mathcal{H}_2$ such that:
\begin{itemize}
\item $H_{1,1}'\cup \ldots \cup H_{1,n'}'\in\mathcal{H}_1'$.
\item $H_{2,1}'\cup \ldots \cup H_{2,n'}'\in\mathcal{H}_2'$.
\item $H=(H_{1,1}'\times H_{2,1}')\cup \ldots \cup (H_{1,n'}'\times H_{2,n'}')$.
\end{itemize}
Since $H=(H_{1,1}'\times H_{2,1}')\cup \ldots \cup (H_{1,n'}'\times H_{2,n'}')$, we have $P_1(H)=H_{1,1}'\cup \ldots \cup H_{1,n'}'\in\mathcal{H}_1'$. But this is true for every $H\in\mathcal{H}$. Therefore, $\mathcal{U}_1(\mathcal{H})\subset \mathcal{H}_1'$ which implies that $\mathcal{H}_1'=\mathcal{U}_1(\mathcal{H})$ since $\mathcal{H}_1'$ and $\mathcal{U}_1(\mathcal{H})$ are partitions. Similarly, $\mathcal{H}_2'= \mathcal{U}_2(\mathcal{H})$.

Now let $x_2\in\mathcal{X}_2$ be such that $P_{1|x_2}(H)\neq \o$. Clearly, $x_2\in H_{2,i}'$ for some $1\leq i\leq n'$ and so $P_{1|x_2}(H)=H_{1,i}'\in\mathcal{H}_1$ since $H=(H_{1,1}'\times H_{2,1}')\cup \ldots \cup (H_{1,n'}'\times H_{2,n'}')$ and since $H_{2,1}',\ldots,H_{2,n'}'$ are disjoint. Therefore, for every $x_2\in\mathcal{X}_2$ satisfying $P_{1|x_2}(H)\neq \o$, we have $P_{1|x_2}(H) \in\mathcal{H}_1$. We conclude that $\mathcal{L}_1(\mathcal{H})\subset \mathcal{H}_1$ which implies that $\mathcal{H}_1=\mathcal{L}_1(\mathcal{H})$ since $\mathcal{H}_1$ and $\mathcal{L}_1(\mathcal{H})$ are partitions. Similarly, $\mathcal{H}_2=\mathcal{L}_2(\mathcal{H})$. Moreover, $n'=\frac{\|\mathcal{H}_1'\|}{\|\mathcal{H}_1\|}=\frac{\|\mathcal{U}_1(\mathcal{H})\|}{\|\mathcal{L}_1(\mathcal{H})\|}=n$.

We conclude that the stable partitions $\mathcal{L}_1(\mathcal{H})$, $\mathcal{L}_2(\mathcal{H})$, $\mathcal{U}_1(\mathcal{H})$, $\mathcal{U}_2(\mathcal{H})$ are unique.
\end{proof}

\subsection{Proof of Theorem \ref{theprodstrong}}

For Theorem \ref{theprodstrong}, we will first prove it for $m=2$ using two lemmas. The general result can then be proven by induction on $m\geq 2$.

\begin{mylem}
\label{lemProdStrongNec}
If $\ast=\ast_1\otimes\ast_2$ is a strongly ergodic operation on $\mathcal{X}=\mathcal{X}_1\times\mathcal{X}_2$, then $\ast_1$ and $\ast_2$ are strongly ergodic.
\end{mylem}
\begin{proof}
Let $\mathcal{H}_1$ be a stable partition of $\mathcal{X}_1$, then $\mathcal{H}=\mathcal{H}_1\otimes \{\mathcal{X}_2\}$ is a stable partition of $\mathcal{X}_1\times \mathcal{X}_2$. Fix $x_2\in\mathcal{X}_2$ and let $x_1\in\mathcal{X}_1$. Since $\ast$ is strongly ergodic, then by Definition \ref{defdef} there exists $n=n(x_1,x_2,\mathcal{H})>0$ such that for any $H\in \mathcal{H}^{n\ast}$, there exists an $\mathcal{H}$-sequence $\mathfrak{X}=(X_i)_{0\leq i<n}$ satisfying $(x_1,x_2)\ast\mathfrak{X}=H$. Let $H_1\in\mathcal{H}_1^{n\ast_1}$. Clearly, $H_1\times\mathcal{X}_2\in\mathcal{H}_1^{n\ast_1}\otimes \{\mathcal{X}_2\}=(\mathcal{H}_1\otimes \{\mathcal{X}_2\})^{n\ast}=\mathcal{H}^{n\ast}$.

Since $H_1\times\mathcal{X}_2\in\mathcal{H}^{n\ast}$, there exists an $\mathcal{H}$-sequence $\mathfrak{X}=(X_i)_{0\leq i<n}$ such that $(x_1,x_2)\ast\mathfrak{X}=H_1\times\mathcal{X}_2$. For every $0\leq i< n$, $X_i\in(\mathcal{H}_1\otimes\{\mathcal{X}_2\})^{i\ast}= \mathcal{H}_1^{i\ast_1}\otimes\{\mathcal{X}_2\}$ and so there exists $X_{1,i}\in\mathcal{H}_1^{i\ast_1}$ such that $X_i=X_{1,i}\times \mathcal{X}_2$. By projecting the equation $(x_1,x_2)\ast\mathfrak{X}=H_1\times\mathcal{X}_2$ on the first coordinate, we get $x_1\ast_1 \mathfrak{X}_1=H_1$, where $\mathfrak{X}_1$ is the $\mathcal{H}_1$-sequence $(X_{1,i})_{0\leq i<n}$. By fixing $x_2\in\mathcal{X}_2$, $n$ will depend only on $x_1$ and $\mathcal{H}_1$ as required in the definition of strong ergodicity. This proves that $\ast_1$ is strongly ergodic. A similar argument shows that $\ast_2$ is also strongly ergodic.
\end{proof}

\begin{mylem}
\label{lemProdStrongSuf}
If $\ast_1$ and $\ast_2$ are two strongly ergodic operations on $\mathcal{X}_1$ and $\mathcal{X}_2$ respectively, then $\ast=\ast_1\otimes\ast_2$ is a strongly ergodic operation on $\mathcal{X}=\mathcal{X}_1\times\mathcal{X}_2$.
\end{mylem}
\begin{proof}
Fix a stable partition $\mathcal{H}$ of $\mathcal{X}$. Since $\ast_1$ and $\ast_2$ are strongly ergodic, they are ergodic and so Theorem \ref{theprod} can be applied. Let $\mathcal{L}_1(\mathcal{H})$, $\mathcal{L}_2(\mathcal{H})$, $\mathcal{U}_1(\mathcal{H})$ and $\mathcal{U}_2(\mathcal{H})$ be defined as in Theorem \ref{theprod}, and let $P_1$ and $P_2$ be the projection onto the first and second coordinate respectively as in Definition \ref{defProjGG}.

Let $(x_1,x_2)\in H\in\mathcal{H}$. We will construct an $\mathcal{H}$-augmenting sequence $\mathfrak{X}$ satisfying $H\subset (x_1,x_2)\ast\mathfrak{X}$ in two steps: We first construct an $\mathcal{H}$-augmenting sequence $\mathfrak{X}_U$ such that $P_1(H)\subset P_1\big((x_1,x_2)\ast\mathfrak{X}_U\big)$, i.e., $\mathfrak{X}_U$ stretches $\{(x_1,x_2)\}$ in the direction of the first coordinate to cover $P_1(H)$. In the second step, we construct an $\mathcal{H}$-augmenting sequence $\mathfrak{X}_L$ such that $H\subset \big((x_1,x_2)\ast\mathfrak{X}_U\big)\ast \mathfrak{X}_L$, i.e., $\mathfrak{X}_L$ stretches $(x_1,x_2)\ast\mathfrak{X}_U$ in the direction of the second coordinate to cover $H$.

\vspace*{2mm}

\emph{Step 1:} Let $H_1=P_1(H)\in\mathcal{U}_1(\mathcal{H})$. Since $\ast_1$ is strongly ergodic, there exists a $\mathcal{U}_1(\mathcal{H})$-augmenting sequence $\mathfrak{X}_1$ such that $x_1\ast_1\mathfrak{X}_1=H_1$. Let $\mathfrak{X}_1'=(X_{1,i}')_{0\leq i<k'}=(\mathfrak{X}_1)^{\per(\mathcal{H})}$. For every $0\leq i< k'=|\mathfrak{X}_1'|$, we have $X_{1,i}'\in \mathcal{U}_1(\mathcal{H})^{i\ast_1}= \mathcal{U}_1(\mathcal{H}^{i\ast})$, and so from Definition \ref{defProjGG} there exists $X_i'\in\mathcal{H}^{i\ast}$ such that $P_1(X_i')=X_{1,i}'$. Define the $\mathcal{H}$-sequence $\mathfrak{X}_U'=(X_i')_{0\leq i< k'}$. The sequence $\mathfrak{X}_U'$ is $\mathcal{H}$-repeatable since $\per(\mathcal{H})$ divides $|\mathfrak{X}_U'|=k'=|\mathfrak{X}_1|\cdot\per(\mathcal{H})$. By Lemma \ref{lemaugm}, there exists $l>0$ such that $\mathfrak{X}_U:=(\mathfrak{X}_U')^l$ is $\mathcal{H}$-augmenting. We have:
\begin{equation}
\label{eqsdj863wqd32wed}
\begin{aligned}
H_1 &\stackrel{(a)}{\subset} H_1\ast_1(\mathfrak{X}_1)^{\per(\mathcal{H})l-1} =(x_1\ast_1 \mathfrak{X}_1)\ast_1 (\mathfrak{X}_1)^{\per(\mathcal{H})l-1}=x_1\ast_1 (\mathfrak{X}_1)^{\per(\mathcal{H})l}=x_1\ast_1 (\mathfrak{X}_1')^l\\
&=x_1\ast_1 \big((X_{1,i}')_{0\leq i<k'}\big)^l = P_1\big((x_1,x_2)\big)\ast_1 \Big(\big(P_1(X_i')\big)_{0\leq i<k'}\Big)^l = P_1\Big((x_1,x_2)\ast \big((X_i')_{0\leq i<k'}\big)^l\Big) \\
&= P_1\big((x_1,x_2)\ast (\mathfrak{X}_U')^l\big)=P_1\big((x_1,x_2)\ast \mathfrak{X}_U\big),
\end{aligned}
\end{equation}
where (a) follows from the fact that $\mathfrak{X}_1$ is $\mathcal{U}_1(\mathcal{H})$-augmenting.

\vspace*{2mm}

\emph{Step 2:} Define $X_U= (x_1,x_2)\ast \mathfrak{X}_U$. Since $\mathfrak{X}_U$ is $\mathcal{H}$-augmenting, we must have $X_U\subset K$, where $K\in\mathcal{K}_{\mathcal{H}}$ is such that $(x_1,x_2)\in K$ (see Theorem \ref{theres}). Now since $\mathcal{K}_{\mathcal{H}}$ is a sub-stable partition of $\mathcal{H}$ (by Theorem \ref{theres}) and since $(x_1,x_2)\in K\cap H$, we must have $K\subset H$. Therefore, $X_U\subset H$. On the other hand, from \eqref{eqsdj863wqd32wed} we have $H_1\subset P_1(X_U)$. We conclude that for every $a\in H_1$, we have $a\in P_1(X_U)$ and so there exists $b_a\in\mathcal{X}_2$ such that $(a,b_a)\in X_U\subset H$.

According to Theorem \ref{theprod}, there exist $n$ disjoint sets $H_{1,1},\ldots,H_{1,n}\in\mathcal{L}_1(\mathcal{H})$ and $n$ disjoint sets $H_{2,1},\ldots,H_{2,n}\in\mathcal{L}_2(\mathcal{H})$ such that $H=(H_{1,1}\times H_{2,1})\cup \ldots \cup (H_{1,n}\times H_{2,n})$. For every $a\in H_1=H_{1,1}\cup\ldots\cup H_{1,n}$, there exists a unique $1\leq i_a\leq n$ such that $a\in H_{1,i_a}$. We have:
\begin{equation}
\label{eqsadn24ar3fsH}
H=\bigcup_{1\leq i\leq n} (H_{1,i}\times H_{2,i})=\bigcup_{1\leq i\leq n} \bigcup_{a\in H_{1,i}} (\{a\}\times H_{2,i})= \bigcup_{1\leq i\leq n} \bigcup_{a\in H_{1,i}} (\{a\}\times H_{2,i_a})= \bigcup_{a\in H_1} (\{a\}\times H_{2,i_a}).
\end{equation}

Fix $a\in H_1$. Since $\displaystyle (a,b_a)\in H= \bigcup_{a'\in H_1} (\{a'\}\times H_{2,i_{a'}})$, we must have $b_a\in H_{2,i_a}\in \mathcal{L}_2(\mathcal{H})$. Now since $\ast_2$ is strongly ergodic, there exists an $\mathcal{L}_2(\mathcal{H})$-augmenting sequence $\mathfrak{X}_{2,a}$ such that $b_a \ast_2 \mathfrak{X}_{2,a}= H_{2,i_a}$. Let $\mathfrak{X}_{2,a}'=(X_{2,a,i}')_{0\leq i<k_a'}=(\mathfrak{X}_{2,a})^{\per(\mathcal{H})}$.  For every $0\leq i< k_a'$, we have $X_{2,a,i}'\in \mathcal{L}_2(\mathcal{H})^{i\ast_2}=\mathcal{L}_2(\mathcal{H}^{ i\ast})$, and so from Definition \ref{defProjEE} there exist $x_{1,a,i}'\in\mathcal{X}_1$ and $X_{a,i}'\in\mathcal{H}^{i\ast}$ such that $X_{2,a,i}'=P_{2|x_{1,a,i}'}(X_{a,i}')$. Define the $\mathcal{H}$-sequence $\mathfrak{X}_a'=(X_{a,i}')_{0\leq i<k_a'}$. The sequence $\mathfrak{X}_a'$ is $\mathcal{H}$-repeatable since $\per(\mathcal{H})$ divides $|\mathfrak{X}_a'|=k_a'=|\mathcal{X}_{2,a}|\cdot\per(\mathcal{H})$.

Define the mapping $\pi_a:\mathcal{X}_1\rightarrow\mathcal{X}_1$ as $\pi_a(x)=(((x\ast_1 x_{1,a,0}')\ast_1 x_{1,a,1}')\ldots\ast_1 x_{1,a,k_a'-1}')$ for every $x\in\mathcal{X}_1$. Since $\pi_a$ is a permutation, there exists $p_a>0$ such that $\pi_a^{p_a}(x)=x$ for every $x\in\mathcal{X}_1$. $(\mathfrak{X}_a')^{p_a}$ is $\mathcal{H}$-repeatable since $\mathfrak{X}_a'$ is $\mathcal{H}$-repeatable. Now by Lemma \ref{lemaugm} there exists $l_a>0$ such that $\mathfrak{X}_a:=(\mathfrak{X}_a')^{p_al_a}$ is $\mathcal{H}$-augmenting. We have:
\begin{align*}
\{a\}\times H_{2,i_a}&\stackrel{(a)}{\subset} \{a\}\times \big(H_{2,i_a}\ast_2 (\mathfrak{X}_{2,a})^{\per(\mathcal{H})p_al_a-1}\big)=\{a\}\times \big((b_a\ast_2\mathfrak{X}_{2,a})\ast_2 (\mathfrak{X}_{2,a})^{\per(\mathcal{H})p_al_a-1}\big)\\
&\stackrel{(b)}{=}\{\pi_a^{p_al_a}(a)\}\times\big(b_a\ast_2(\mathfrak{X}_{2,a})^{\per(\mathcal{H})p_al_a}\big)= \{\pi_a^{p_al_a}(a)\}\times\big(b_a\ast_2(\mathfrak{X}_{2,a}')^{p_al_a}\big)\\
&\stackrel{(c)}{=} (a,b_a)\ast \Big(\big(\{x_{1,a,i}'\}\times X_{2,a,i}'\big)_{0\leq i< k_a}\Big)^{p_al_a} \stackrel{(d)}{\subset}
(a,b_a)\ast \big((X_{a,i}')_{0\leq i< k_a}\big)^{p_al_a}\\
&= (a,b_a)\ast (\mathfrak{X}_a')^{p_al_a}= (a,b_a)\ast \mathfrak{X}_a.
\end{align*}
(a) follows from the fact that $\mathfrak{X}_{2,a}$ is $\mathcal{L}_2(\mathcal{H})$-augmenting, hence $(\mathfrak{X}_{2,a})^{\per(\mathcal{H})p_al_a-1}$ is $\mathcal{L}_2(\mathcal{H})$-augmenting (by Remark \ref{remtemdem123}), and so $H_{2,i_a}\subset H_{2,i_a}\ast_2 (\mathfrak{X}_{2,a})^{\per(\mathcal{H})p_al_a-1}$. (b) follows from the fact that $\pi_a^{p_a}(x)=x$ for every $x\in\mathcal{X}_1$, which implies that $\pi_a^{p_al_a}(a)=a$. (c) follows from the definition of $\pi_a$ and from the fact that $\mathfrak{X}_{2,a}'=(X_{2,a,i}')_{0\leq i<k_a}$. (d) follows from the fact that $P_{2|x_{1,a,i}'}(X_{a,i}')=X_{2,a,i}'$, which implies that $\{x_{1,a,i}'\}\times X_{2,a,i}' \subset X_{a,i}'$ for every $0\leq i< k_a$.

Now let $\mathfrak{X}_L=(\mathfrak{X}_a)_{a\in H_1}$ be the $\mathcal{H}$-augmenting sequence obtained by concatenating the $\mathcal{H}$-augmenting sequences $\mathfrak{X}_a$ for all $a\in H_1$ (the order of the concatenation is not important). Since $\{a\}\times H_{2,i_a}\subset (a,b_a)\ast \mathfrak{X}_a$ for every $a\in H_1$, we must have
\begin{equation}
\label{eqgsdjge31urqwhf34}
\{a\}\times H_{2,i_a}\subset (a,b_a)\ast \mathfrak{X}_L\text{\;for\;every\;}a\in H_1.
\end{equation}

Define $\mathfrak{X}=(\mathfrak{X}_U,\mathfrak{X}_L)$. We have
$(x_1,x_2)\ast\mathfrak{X}=\big((x_1,x_2)\ast \mathfrak{X}_U\big)\ast\mathfrak{X}_L=X_U\ast\mathfrak{X}_L$. For every $a\in H_1$, we have already shown that $(a,b_a)\in X_U$ and so it follows from \eqref{eqgsdjge31urqwhf34} that:
$$\{a\}\times H_{2,i_a}\subset (a,b_a)\ast \mathfrak{X}_L\subset X_U\ast\mathfrak{X}_L=(x_1,x_2)\ast\mathfrak{X}.$$
Since this is true for every $a\in H_1$, we have:
$$H\stackrel{(a)}{=}\bigcup_{a\in H_1}\{a\}\times H_{2,i_a}\subset (x_1,x_2)\ast\mathfrak{X},$$
where (a) follows from \eqref{eqsadn24ar3fsH}.

Now since $\mathfrak{X}$ is $\mathcal{H}$-augmenting, Theorem \ref{theres} implies that $(x_1,x_2)\ast\mathfrak{X}\subset K$, where $K\in\mathcal{K}_{\mathcal{H}}$ is such that $(x_1,x_2)\in K$. Therefore, $\|\mathcal{H}\|=|H|\leq | (x_1,x_2)\ast\mathfrak{X}|\leq|K|=\|\mathcal{K}_{\mathcal{H}}\|$. Now since $\mathcal{K}_{\mathcal{H}}$ is a sub-stable partition of $\mathcal{H}$, we conclude that $\mathcal{K}_{\mathcal{H}}=\mathcal{H}$. But this is true for every stable partition $\mathcal{H}$ of $\mathcal{X}$, hence $\ast$ is strongly ergodic.
\end{proof}

\vspace*{3mm}

Now we are ready to prove Theorem \ref{theprodstrong}:
\begin{proof}[Proof of Theorem \ref{theprodstrong}]
Lemmas \ref{lemProdStrongNec} and \ref{lemProdStrongSuf} show that Theorem \ref{theprodstrong} is true for $m=2$. Now let $m>2$ and suppose that the theorem is true for $m-1$.

Let $\ast_1,\ldots,\ast_m$ be $m$ binary operations such that $\ast_1\otimes\ldots\otimes \ast_m$ is strongly ergodic. It is easy to see that $\ast_1\otimes\ldots\otimes \ast_m$ can be identified to $(\ast_1\otimes\ldots\otimes \ast_{m-1})\otimes \ast_m$ (see Notation \ref{notIdent}). Therefore, $(\ast_1\otimes\ldots\otimes \ast_{m-1})\otimes \ast_m$ is strongly ergodic. Lemma \ref{lemProdStrongNec} implies that $\ast_1\otimes\ldots\otimes \ast_{m-1}$ and $\ast_m$ are strongly ergodic. It then follows from the induction hypothesis that $\ast_1,\ldots,\ast_{m-1}$ are strongly ergodic. Therefore, $\ast_1,\ldots,\ast_m$ are strongly ergodic.

Conversely, let $\ast_1,\ldots,\ast_m$ be $m$ strongly ergodic operations. From the induction hypothesis, we get that $\ast_1\otimes\ldots\otimes\ast_{m-1}$ is strongly ergodic. Lemma \ref{lemProdStrongSuf} implies that $(\ast_1\otimes\ldots\otimes\ast_{m-1})\otimes \ast_m$ is strongly ergodic. But since $(\ast_1\otimes\ldots\otimes\ast_{m-1})\otimes \ast_m$ can be identified to $\ast_1\otimes\ldots\otimes \ast_m$, we conclude that $\ast_1\otimes\ldots\otimes \ast_m$ is strongly ergodic.

Therefore, Theorem \ref{theprodstrong} is true for all $m\geq 2$.
\end{proof}

\section*{Acknowledgment}
I would like to thank Emre Telatar for enlightening discussions and for his helpful feedback on the paper. 

\bibliographystyle{IEEEtran}
\bibliography{bibliofile}
\end{document}